\documentclass[a4paper,11pt,reqno]{amsart}
\let\medskip\relax

\usepackage[utf8]{inputenc}
\usepackage[T1]{fontenc}
\usepackage{lmodern}
\usepackage[english]{babel}
\usepackage{microtype}

\usepackage{amsmath,amssymb,amsfonts,amsthm,esint}
\usepackage{mathtools,accents}
\usepackage{mathrsfs}
\usepackage{aliascnt}
\usepackage{braket}
\usepackage{bm}

\usepackage[a4paper,margin=3cm]{geometry}
\usepackage[citecolor=blue,colorlinks]{hyperref}

\usepackage{enumitem}
\usepackage{xcolor}

\makeatletter
\g@addto@macro\@floatboxreset\centering
\makeatother

\makeatletter
\def\newaliasedtheorem#1[#2]#3{
	\newaliascnt{#1@alt}{#2}
	\newtheorem{#1}[#1@alt]{#3}
	\expandafter\newcommand\csname #1@altname\endcsname{#3}
}
\makeatother

\numberwithin{equation}{section}

\newtheoremstyle{slanted}{\topsep}{\topsep}{\slshape}{}{\bfseries}{.}{.5em}{}

\theoremstyle{plain}
\newtheorem{theorem}{Theorem}[section]
\newaliasedtheorem{proposition}[theorem]{Proposition}
\newaliasedtheorem{lemma}[theorem]{Lemma}
\newaliasedtheorem{corollary}[theorem]{Corollary}
\newaliasedtheorem{counterexample}[theorem]{Counterexample}

\theoremstyle{definition}
\newaliasedtheorem{definition}[theorem]{Definition}
\newaliasedtheorem{question}[theorem]{Question}
\newaliasedtheorem{openquestion}[theorem]{Open Question}
\newaliasedtheorem{conjecture}[theorem]{Conjecture}

\theoremstyle{remark}
\newaliasedtheorem{remark}[theorem]{Remark}
\newaliasedtheorem{example}[theorem]{Example}

\newcommand{\setN}{\mathbb{N}}
\newcommand{\N}{\mathbb{N}}

\newcommand{\setR}{\mathbb{R}}
\newcommand{\R}{\mathbb{R}}
\newcommand{\Z}{\mathbb{Z}}

\newcommand{\eps}{\varepsilon}
\renewcommand{\epsilon}{\varepsilon}
\renewcommand{\phi}{\varphi}
\renewcommand{\bar}{\overline}
\renewcommand{\hat}{\widehat}
\renewcommand{\tilde}{\widetilde}

\newcommand{\de}{\partial}
\newcommand{\mz}{\frac{1}{2}}

\DeclareMathOperator{\Hess}{Hess}

\newcommand{\di}{\mathop{}\!\mathrm{d}}

\newcommand{\diam}{{\rm diam}}

\DeclareMathOperator{\Lip}{Lip}

\DeclareMathOperator{\Ric}{Ric}

\newcommand{\haus}{\mathscr{H}}
\newcommand{\Leb}{\mathscr{L}}

\newcommand{\dist}{\mathsf{d}}

\newcommand{\meas}{\mathfrak{m}}

\newcommand{\Test}{{\rm Test}}

\DeclareMathOperator{\CD}{CD}
\DeclareMathOperator{\RCD}{RCD}

\newfont{\tmpf}{cmsy10 scaled 2500}

\def\XXint#1#2#3{{\setbox0=\hbox{$#1{#2#3}{\int}$ }
		\vcenter{\hbox{$#2#3$ }}\kern-.6\wd0}}

\newcommand{\ang}[1]{\langle #1 \rangle}

\begin{document}

\title[Topological regularity of noncollapsed Ricci limits]{Topological regularity and stability \\ of noncollapsed spaces \\ with Ricci curvature bounded below}

\author{Elia Bru\`{e}, Alessandro Pigati, and Daniele Semola}

\address{\parbox{\linewidth}{Bocconi University, Department of Decision Sciences.\\
	Via Sarfatti 25,
	20136 Milano -- Italy\\[-4pt]\phantom{a}}}
\email{elia.brue@unibocconi.it}

\address{\parbox{\linewidth}{Bocconi University, Department of Decision Sciences.\\
	Via Sarfatti 25,
	20136 Milano -- Italy\\[-4pt]\phantom{a}}}
\email{alessandro.pigati@unibocconi.it}

\address{\parbox{\linewidth}{FIM-ETH Z\"urich.\\
	R\"amistrasse 101,
    8092 Z\"urich -- Switzerland\\[-4pt]\phantom{a}}}
\email{daniele.semola@math.ethz.ch}

\begin{abstract}
 We investigate the topological regularity and stability of noncollapsed Ricci limit spaces $(M_i^n,g_i,p_i)\xrightarrow{\rm{GH}} (X^n,\dist)$. We confirm a conjecture proposed by Colding and Naber in dimension $n=4$, showing that the cross-sections of tangent cones at a given point $x\in X^4$ are all homeomorphic to a fixed spherical space form $S^3/\Gamma_x$, and $\Gamma_x$ is trivial away from a $0$-dimensional set. In dimensions $n>4$, we show an analogous statement at points where all tangent cones are $(n-4)$-symmetric.
 
Furthermore, we prove that $(n-3)$-symmetric noncollapsed Ricci limits are topological manifolds, thus confirming a particular case of a conjecture due to Cheeger, Colding, and Tian.

Our analysis relies on two key results, whose importance goes beyond their applications in the study of cross-sections of noncollapsed Ricci limit spaces:

    \begin{itemize}
        \item[(i)] A new manifold recognition theorem for noncollapsed $\RCD(-2,3)$ spaces.

        \item[(ii)] A cone rigidity result ruling out noncollapsed Ricci limit spaces of the form $\R^{n-3}\times C(\mathbb{RP}^2)$.
    \end{itemize}
\end{abstract}

\maketitle

\tableofcontents

\section{Introduction}

We consider {\it noncollapsed Ricci limit spaces}, i.e., pointed Gromov--Hausdorff limits of smooth complete $n$-dimensional manifolds 
\begin{equation}\label{Riccilimits}
(M_i^n, g_i ,p_i) \to (X, \dist, p)\, ,
\end{equation}
with a uniform lower bound on the Ricci curvature and on the volume of balls
\begin{equation}
\Ric_{g_i}\ge -(n-1)\, ,\quad \mathrm{vol}_{g_i}(B_1(p_i))\ge v>0\, .
\end{equation}
By Gromov's compactness theorem, limit spaces always exist in the category of metric spaces, even dropping the uniform noncollapsing assumption $\mathrm{vol}_{g_i}(B_1(p_i))\ge v>0$. The systematic study of their structure and regularity began in the nineties with the works of Cheeger and Colding \cite{CheegerColding96, CheegerColding97I, CheegerColding00II,Coldingvol}, with earlier insights due to Fukaya \cite{Fukaya87}, and Anderson \cite{Anderson90},
and it continues to be an active research field today: see for instance \cite{ColdingNaberHolder,ColdingNabercones,CheegerJiangNaber,PanWei,BrueNaberSemolabdry}.

\medskip

By the volume convergence theorem (see \cite[Theorem 5.9]{CheegerColding97I} after \cite{Coldingvol}), the volume measures $\mathrm{vol}_{g_i}$ converge to the $n$-dimensional Hausdorff measure $\haus^n$ of $(X,\dist)$ in the above setting. The latter plays a central role in the fine analysis of the structure of noncollapsed Ricci limit spaces. Indeed, the ``volume cone implies metric cone'' theorem from \cite{CheegerColding96} can be employed to show that blow-ups at any point $x\in X$
  \begin{equation}
  	(X,r^{-1}_j\dist,x ) \to
  	(C(Z), \dist, o) \, , \quad r_j\downarrow 0 \, ,
  \end{equation}
  are {\it metric cones} with tip point $o\in C(Z)$. The Euclidean symmetries of tangent cones can be used in conjuction with the splitting theorem from \cite{CheegerColding96} to break up the space $X$ into a {\it regular set} $\mathcal{R}(X)$, where all tangents are isometric to $\R^n$, and a family of {\it singular strata} $\mathcal{S}^k(X)$ with $\mathcal{S}^{k-1}(X)\subseteq \mathcal{S}^k(X)$ for $k = 0, \ldots, n-1$, where no tangent cone splits a Euclidean factor $\setR^{k+1}$.
  It is known since \cite{CheegerColding97I} that $\mathcal{S}^{n-1}(X)\setminus \mathcal{S}^{n-2}(X) = \emptyset$ if the manifolds $M^n_i$ have empty boundaries, and $\dim_{\mathcal{H}} \mathcal{S}^k(X) \le k$ for every $k\le n-1$. Furthermore, a Reifenberg-type result of Cheeger and Colding \cite[Theorem A.1.1]{CheegerColding97I} shows that there is an open neighbourhood of $\mathcal{R}(X)$ which is biH\"older homeomorphic to a smooth manifold. By the more recent quantitative estimates on the singular strata due to Cheeger, Jiang, and Naber the complement of this manifold set can be taken to be $(n-2)$-rectifiable with locally finite $\mathcal{H}^{n-2}$-measure \cite[Theorem 1.14]{CheegerJiangNaber}.

\medskip

 Conjecturally, noncollapsed Ricci limit spaces might be homeomorphic to manifolds away from a closed subset of Hausdorff codimension at least $4$: see
\cite[Conjecture 0.7]{CheegerColding97I}, \cite[Remark 10.23]{CheegerFermi} and \cite[Remark 1.19]{CheegerColdingTian}.

The case $n=2$ in this conjecture is classical and originally due to Alexandrov. We note that if $n=2$ the statement follows also from Perelman's stability theorem \cite{Perelman99}, which deals with noncollapsed limits with a uniform lower bound on the sectional curvature in any dimension.

In dimension $n=3$, there is again a complete understanding of the topology of noncollapsed Ricci limit spaces. In \cite{Zhu93}, Zhu proved that noncollapsed three-dimensional Ricci limits are homology manifolds, based on the uniform local contractibility of smooth $3$-manifolds with lower Ricci and volume bounds and the methods in \cite{Petersen90}. More recently, Simon \cite{Simon14} and Simon--Topping \cite{SimonTopping22a,SimonTopping22} have shown that three-dimensional noncollapsed Ricci limit spaces are biH\"older homeomorphic to smooth Riemannian manifolds, using Ricci flow techniques.

If $n\ge 4$, there is a gap of two dimensions between the conjectural picture about the topological regularity of noncollapsed Ricci limit spaces and the present state of the art.

 A variant of this conjecture for noncollapsed limits of manifolds with two-sided bounds on the Ricci curvature (see for instance \cite[Conjecture 2.3]{Anderson93}) became known as the \emph{codimension four conjecture} and was settled by Cheeger and Naber in \cite{CheegerNaber15}.

\subsection{Topological regularity and stability of tangent cones}

The first main result of this paper concerns the topological regularity and stability of cross-sections of tangent cones of $4$-dimensional noncollapsed Ricci limit spaces.

\begin{theorem}[Tangent cones of $4$-dimensional limits]
\label{cor:4d}
Let $(M_i^4, g_i ,p_i)\to (X^4,\dist,p)$ be a $4$-dimensional noncollapsed Ricci limit space.
For each $x\in X$ there exists a topological $3$-manifold $\Sigma_x$ with universal cover homeomorphic to $S^3$ such that all tangent cones at $x$ have cross-section homeomorphic to $\Sigma_x$. Furthermore, the set of points $x\in X$ such that $\Sigma_x$ is not homeomorphic to $S^3$ has Hausdorff dimension $0$. 
\end{theorem}

\autoref{cor:4d} confirms in particular a conjecture by Colding and Naber \cite[Conjecture 1.2]{ColdingNabercones} in dimension $n=4$. An analogous topological regularity and stability statement holds for blow-downs of smooth complete $4$-manifolds with nonnegative Ricci curvature and Euclidean volume growth, with the very same proof.

\begin{remark}[Eguchi--Hanson metric]\label{rm:EguchiHanson}
It is well known that in dimension $n=4$ a noncollapsed Ricci limit space might have points where the cross-section of the tangent cone is not homeomorphic to $S^3$. 
A classical example is the blow-down of the Eguchi--Hanson metric \cite{EguchiHanson},  which is a complete Ricci flat metric $g$ with Euclidean volume growth over the cotangent bundle of $S^2$. The blow-down of this manifold is the metric cone over $\mathbb{RP}^3$ where $\mathbb{RP}^3$ is endowed with a metric with constant sectional curvature $1$. The example also shows that $4$ would be the sharp codimension for the non-manifold set of noncollapsed Ricci limit spaces.
\end{remark}

Analogously, the blow-down of the product between the Eguchi--Hanson metric and a line is a metric cone with cross-section the suspension over $\mathbb{RP}^3$, which is not a topological manifold. In particular, the topological manifold regularity of cross-sections does not extend to dimensions $n\ge 5$.

\begin{remark}[Colding--Naber example]\label{rm:exCN}
In dimension $n=5$, Colding and Naber \cite{ColdingNabercones} have constructed a noncollapsed Ricci limit space $(X^5,\dist)$ such that at a point $x\in X$ there are two distinct tangent cones with non-homeomorphic cross-sections.
In particular, the topological stability part of \autoref{cor:4d} is dimensionally sharp, in the sense that it does not generalize to noncollapsed Ricci limits of dimensions $n\ge 5$.
\end{remark}

In higher dimensions $n>4$, we can partially settle \cite[Conjecture 1.2]{ColdingNabercones}, as follows.

\begin{theorem}\label{thm:omeo n-4 sym}
Let $(X^n,\dist)$ be a noncollapsed Ricci limit space of dimension $n\ge 4$.
\begin{itemize}
    \item[(i)] If $\R^{n-4}\times C(Z^3)$ is an $(n-4)$-symmetric tangent cone at $x\in X$, then $(Z^3,\dist_Z)$ is homeomorphic to a topological $3$-manifold whose universal cover is $S^3$.

    \item[(ii)] If all tangent cones at $x\in X$ are $(n-4)$-symmetric, i.e., each one is isometric to $\mathbb{R}^{n-4}\times C(Z)$ for some metric space $Z$, then all the cross-sections $Z$ must be homeomorphic to each other.
\end{itemize}
\end{theorem}

\begin{remark}
A conjecture by Naber \cite[Conjecture 2.16]{Naberconj} predicts that for a noncollapsed Ricci limit space $(X^n,\dist)$ all tangent cones at a given point should be $k$-symmetric away from a set of Hausdorff dimension less than $k-1$. If confirmed, in combination with \autoref{thm:omeo n-4 sym} (ii), this would establish \cite[Conjecture 1.2]{ColdingNabercones}.
\end{remark}

\subsection{Topology of $(n-3)$-symmetric limits}

As we already mentioned, when $n\ge 4$ there is a gap of two dimensions between the conjectural topological manifold regularity away from codimension $4$ for noncollapsed Ricci limits and the present state of the art. Four would be the sharp codimension for the non-manifold set, as \autoref{rm:EguchiHanson} illustrates. We note also that Menguy constructed examples of noncollapsed Ricci limits in dimension $4$ which are not topological manifolds even though all tangent cones are homeomorphic to $\setR^4$, see \cite[Theorem 0.6]{Menguy inftop}.

\medskip

Our next results represent partial progress towards the conjectural topological regularity of noncollapsed Ricci limits away from sets of codimension $4$. As a first step, we can rule out the existence of topological singularities of the form $\mathbb{R}^{n-3} \times C(\mathbb{RP}^2)$.

\begin{theorem}\label{thm:noRP2}
Let $(M_i^n, g_i, p_i)\to (X^n,\dist, p)$ be a noncollapsed Ricci limit space with $n\ge 3$. Assume that $X^n =\setR^{n-3}\times C(Z^2)$ is an $(n-3)$-symmetric cone. Then $(Z^2,\dist_Z)$ is homeomorphic to the $2$-sphere $S^2$.  
\end{theorem}

We note that the case $n=3$ in \autoref{thm:noRP2} follows already from \cite{Zhu93}, or alternatively from \cite{SimonTopping22}, in combination with the work of Lytchak--Stalder \cite{LytchakStadler}. However, their techniques do not seem to adapt to the case $n\ge 4$. In the case of noncollapsed limits with two-sided bounds on the Ricci curvature, the analogous result is one of the key steps towards the resolution of the codimension four conjecture: see in particular \cite[Theorem 5.12]{CheegerNaber15}.
\medskip

By relying on \autoref{thm:noRP2} and the manifold recognition \autoref{thm:RCDtopma} that we shall discuss in the next section, we can prove that \cite[Conjecture 0.7]{CheegerColding97I} is verified for $(n-3)$-symmetric noncollapsed Ricci limit spaces. 

\begin{theorem}\label{thm:top n-3}
Let $(M_i^n, g_i, p_i)\to (X^n,\dist, p)$ be a noncollapsed Ricci limit space with $n\ge 3$. 
Assume that $X^n =\setR^{n-3}\times Z^3$ as metric measure spaces.
Then $Z^3$ is homeomorphic to a topological $3$-manifold. 
\end{theorem}

In particular we recover, with a different proof, the topological regularity of noncollapsed three-dimensional Ricci limit spaces originally obtained in \cite{Simon14,SimonTopping22}. In this regard, we note that the methods in \cite{Simon14,SimonTopping22a,SimonTopping22} heavily exploit the invariance of lower Ricci curvature bounds under Ricci flow, which is very much specific to dimension three. Moreover, it is presently an open question whether any metric space $(Z^3,\dist_Z)$ as in the statement of \autoref{thm:top n-3} is a noncollapsed three-dimensional Ricci limit space when $n\ge 4$.

\subsection{The topology of three-dimensional $\RCD$ spaces}

In analogy with the theory of Alexandrov spaces with sectional curvature bounded from below, a {\it synthetic approach} to the study of metric measure spaces with Ricci curvature bounded from below was put forward in the seminal works of Sturm \cite{Sturm07I,Sturm07II} and independently Lott--Villani \cite{LottVillani}. The approach is based on the interplay between lower bounds on the Ricci curvature and the Optimal Transport problem, and the resulting spaces are known as $\CD$ spaces, standing for \emph{Curvature-Dimension}. 
More recently, the class of metric measure spaces satisfying the \emph{Riemannian Curvature-Dimension} condition $\RCD$ has attracted significant attention. We address the reader to the survey papers \cite{AmbrosioICM,GigliSurvey,SturmECM} and to the references therein for an overview of the subject.

\medskip
	
  Recently, De Philippis and Gigli introduced the notion of noncollapsed $\RCD$ space in \cite{DePhilippisGigli2}. 
  An $\RCD(K,n)$ metric measure space $(X,\dist,\meas)$ is said to be noncollapsed (or $n$-dimensional) if $\meas=\haus^n$. The class of noncollapsed $\RCD(K,n)$ spaces includes all $n$-dimensional noncollapsed Ricci limit spaces and the inclusion is strict. For instance, the metric cone $C(\mathbb{RP}^2)$ over a projective plane endowed with a metric of constant sectional curvature $1$ is a noncollapsed $\RCD(0,3)$ space which is not a noncollapsed Ricci limit space, by \cite{Zhu93} or alternatively \cite{SimonTopping22}.

  On the other hand, combining the work of Ketterer \cite{Ketterer15} with the results of \cite{DePhilippisGigli,DePhilippisGigli2}, we understand that the cross-section of each tangent cone to a noncollapsed $\RCD(K,n)$ space is a noncollapsed $\RCD(n-2,n-1)$ space. In particular, this is the case for cross-sections of tangent cones of noncollapsed Ricci limit spaces, while such cross-sections are not known to be Ricci limit spaces themselves. 
  This statement motivates our interest in the topological regularity and stability of three-dimensional $\RCD$ spaces.

  In this regard, we introduce three fundamental new results: a manifold recognition theorem for $3$-dimensional $\RCD$ spaces, the uniform local contractibility of noncollapsed $\RCD(-2,3)$ spaces that are topological manifolds, and a topological stability theorem under Gromov--Hausdorff convergence within the same class.

\begin{theorem}[Manifold recognition]\label{thm:RCDtopma}
Let $(X,\dist,\haus^3)$ be an $\RCD(-2,3)$ space. Then $(X,\dist)$ is a topological $3$-manifold without boundary if and only if each point of $X$ has a tangent cone with cross-section homeomorphic to $S^2$.
\end{theorem}

In the statement of \autoref{thm:RCDtopma} it is equivalent to assume that the cross-section of \emph{every} tangent cone is homeomorphic to $S^2$ at each point. Indeed, for an $\RCD(-2,3)$ space $(X,\dist,\haus^3)$, the cross-sections of tangent cones at a given point are all homeomorphic to each other. We also note that the statement can be localized to open sets.
\medskip

A conjecture due to Mondino predicts that a noncollapsed $\RCD(-2,3)$ space should be homeomorphic to an orbifold, possibly with a boundary. The conjecture might be rephrased by saying that the local topology should be determined by the topology of tangent cones. 
Our manifold recognition \autoref{thm:RCDtopma} establishes this conjecture in the case where all tangent cones are homeomorphic to $\setR^3$. In a forthcoming work we aim at addressing the conjecture in its full generality. A characterization of the local topology in terms of tangent cones cannot extend to dimensions $n\ge 4$, as the examples in \cite{Menguy inftop} show.

\medskip

A byproduct of the proof of \autoref{thm:RCDtopma}, together with the solution to the Poincaré conjecture due to Perelman \cite{PerelmanPoincare}, is the following topological rigidity result for noncollapsed $\RCD(0,3)$ manifolds with Euclidean volume growth.

\begin{theorem}\label{thm:volumegrowthR^3}
    Let $(Z^3,\dist_Z)$ be a noncollapsed $\RCD(0,3)$ topological manifold with Euclidean volume growth. Then $Z^3$ is homeomorphic to $\R^3$.
\end{theorem}

In particular, a contractible $3$-manifold not homeomorphic to $\mathbb{R}^3$, such as the Whitehead manifold, does not admit any $\RCD(0,3)$ structure with Euclidean volume growth. This provides a positive answer in a special case to a question asked by Besson \cite[Question 4.1]{Besson21}.

\medskip

Our next result generalizes to $\RCD(-2,3)$ spaces $(X,\dist,\haus^3)$ that are topological manifolds a uniform local contractibility statement originally proved for smooth three-manifolds with lower Ricci curvature and volume bounds by Zhu in \cite{Zhu93}.

\begin{theorem}[Uniform local contractibility]\label{thm:unifcontrtopRCD3}
Let $v>0$ be fixed. There exist constants $C=C(v)>0$ and $\rho=\rho(v)>0$ such that if $(X,\dist,\haus^3)$ is an $\RCD(-2,3)$ topological manifold with $\haus^3(B_1(p))\ge v$ for any $p\in X$, then the ball $B_r(p)$ is contractible inside $B_{Cr}(p)$ for every $r\le\rho$ and every $p\in X$.
\end{theorem}

As above, we remark that an analogous local uniform contractibility result cannot hold in dimension $n\ge 4$. Indeed the examples constructed by Otsu in \cite{Otsu} show that the statement can fail already at the level of the fundamental group. 

\medskip

The uniform local contractibility and the topological manifold regularity can be combined with some abstract results in geometric topology proved in \cite{Petersen90} and \cite{Jakobsche}, relying on the positive resolution of the Poincar\'e conjecture, to obtain a topological stability theorem under Gromov--Hausdorff convergence. 

\begin{theorem}[Topological stability]\label{thm:topstabintro}
Let $(X,\dist,\haus^3)$ be a compact $\RCD(-2,3)$ space which is a topological manifold. There exists $\epsilon=\epsilon(X)>0$ such that if $(Y,\dist_Y,\haus^3)$ is an $\RCD(-2,3)$ space which is a topological manifold and $\dist_{\rm{GH}}(X,Y)<\epsilon$, then $X$ and $Y$ are homeomorphic.    
\end{theorem}

\begin{remark}
    The homeomorphisms constructed in \autoref{thm:topstabintro} are actually perturbations of almost GH-isometries. More precisely, we can prove that any $\delta$-GH isometry is uniformly $\eps$-close to a homeomorphism, provided $\delta \le \delta_0(X,\eps)$.
\end{remark}

The above statement partially generalizes Perelman's stability theorem \cite{Perelman99}, which deals with Alexandrov spaces with curvature bounded from below in any dimension. Again, we stress that an analogous statement does not hold in the case of Ricci curvature lower bounds when the dimension is larger or equal than $4$.

\medskip

As anticipated, \autoref{thm:RCDtopma} and \autoref{thm:topstabintro} are fundamental tools for proving the main results about noncollapsed Ricci limit spaces and the cross-sections of their tangent cones. 
In essence, \autoref{thm:RCDtopma} provides the topological manifold structure, while \autoref{thm:topstabintro} yields the topological stability part in \autoref{cor:4d} and \autoref{thm:omeo n-4 sym}. In order to verify that the assumption about tangent cones in \autoref{thm:RCDtopma} is met in this setting we heavily rely on \autoref{thm:noRP2}. The details will be discussed in Section \ref{sec:Proofmaintheorems}.

\subsection{Strategy of proof}

The technical core of the paper consists of the proofs of the manifold recognition theorem for $3$-dimensional $\RCD$ spaces, namely \autoref{thm:RCDtopma}, and of \autoref{thm:noRP2}, which rules out noncollapsed Ricci limit spaces of the form $\setR^{n-3}\times C(\mathbb{RP}^2)$. The uniform local contractibility within the class of noncollapsed $\RCD(-2,3)$ spaces that are homeomorphic to manifolds (see \autoref{thm:unifcontrtopRCD3}) follows from the methods introduced in the proof of \autoref{thm:RCDtopma}. As already mentioned, \autoref{thm:topstabintro} follows from \autoref{thm:unifcontrtopRCD3} by relying on some abstract results in geometric topology due to \cite{Petersen90,Jakobsche} and on the resolution of the Poincar\'e conjecture \cite{PerelmanPoincare}. Given \autoref{thm:RCDtopma} and \autoref{thm:noRP2}, all the other statements will follow from known 
results and methods: see Section \ref{sec:Proofmaintheorems} for the details.

\medskip

The proofs of \autoref{thm:RCDtopma} and \autoref{thm:noRP2} require several steps and the introduction of many new ideas. The goal of this section is to provide a roadmap to these steps, referring to the relevant sections for the specific statements and arguments.

\medskip

The first ingredient is a variant of the slicing theorem due to Cheeger and Naber \cite[Theorem 1.23]{CheegerNaber15}. In the original formulation, the slicing theorem asserts that a harmonic almost-splitting map $u:B_1(p)\to\setR^{n-2}$ remains almost-splitting at all scales up to composition with a linear transformation, for all points in ``most'' level sets. Here $B_1(p)\subset M^n$ and $(M^n,g)$ has almost nonnegative Ricci curvature. Moreover, ``most'' is understood in a measure theoretic sense. The statement can be thought of as an effective version of Sard's theorem for $(n-2)$-almost splitting maps on $n$-dimensional manifolds with almost nonnegative Ricci curvature. It was the key tool for the proof of the codimension four conjecture in \cite{CheegerNaber15}. 
In Section \ref{sec:slicing}, we establish a variant of the slicing theorem (see \autoref{trans} for the precise statement) where one of the components of the harmonic almost-splitting map is replaced by a ``Green-type distance'' associated with the local Green function of the Laplacian. We address the reader to Section \ref{sec:Green} for the relevant background and terminology. 

To discuss the role of the latter in our work, we first concentrate on the three-dimensional scenario. Consider $(X^3, \dist)$, a noncollapsed $\RCD(-2,3)$ space, and a $\delta$-conical ball $B_r(p) \subset X$ (i.e., a ball $\delta$-GH close to a metric cone). In \autoref{thm:prop_Green} we construct a Green-type distance $b_p$ that approximates and regularizes the distance function from $p$. A similar use of the Green function of the Laplacian in the context of manifolds with Ricci bounded from below was made in \cite{ColdingGreen,NaberJiang}. 
The slicing theorem (\autoref{trans n=3}) indicates that, for most level sets of $b_p$, the function $b_p$ (up to normalization) induces an almost splitting of the ambient space $X$ in every sufficiently small ball centered at a point belonging to one of those level sets. 
The overall strategy of the proofs of \autoref{trans} and \autoref{trans n=3} follows \cite{CheegerNaber15}, although some steps require a number of nontrivial adjustments which are carried out in Sections \ref{sec:slicing} and \ref{sec:proofslicing}. 

\medskip

In the context of noncollapsed limits with bounded Ricci curvature, the slicing theorem was used to rule out the presence of codimension two (metric) singularities in the first place \cite[Theorem 5.2]{CheegerNaber15}. The topology of good slices could then be controlled in the second place by $\epsilon$-regularity \cite[Theorem 5.12]{CheegerNaber15}. In Section \ref{sec:topgoodlevels} we prove that a lower bound on the Ricci curvature and on the volume as in \eqref{Riccilimits} are actually sufficient to control the topology of good slices in a very effective way. Achieving this control requires a new approach with respect to \cite{CheegerNaber15} as the presence of metric singularities of codimension two cannot be ruled out in the present setting.
As a result, we prove that good slices, referred to in the following as (good) {\it Green-spheres}, are locally uniformly contractible closed topological surfaces.

\medskip

A similar construction can be carried out on $(n-3)$-symmetric balls of $n$-dimensional manifolds $M^n$ satisfying lower Ricci curvature and volume bounds as in \eqref{Riccilimits}. In this case, we employ a Green-type distance from a reference point along with $n-3$ almost splitting maps. As before, the slicing \autoref{trans} ensures the existence of numerous good Green-spheres, which are shown to be locally uniformly contractible closed topological surfaces in \autoref{prop:topsurf}.
Combined with the stability of almost splitting maps and Green-type distances, this is the key tool to establish \autoref{thm:noRP2}. When a smooth manifold satisfying \eqref{Riccilimits} is sufficiently Gromov--Hausdorff close to the model space $\mathbb{R}^{n-3}\times C(Z^2)$, the good Green-spheres approximate in the Gromov--Hausdorff sense the cross-section $Z^2$. Thanks to the local uniform contractibility and the results in \cite{Petersen90}, they are actually homeomorphic to $Z^2$. However, since they bound a three-dimensional submanifold in the smooth ambient manifold $M^n$, they cannot be homeomorphic to $\mathbb{RP}^2$.

It is worth mentioning that, in the rigid case where the ambient space splits $\setR^{n-3}$ and it is conical, local uniform contractibility and topological regularity follow from the work of Lytchak--Stadler \cite{LytchakStadler}, which identifies two-dimensional $\RCD(K,2)$ spaces and Alexandrov spaces with curvature bounded from below by $K$.

\medskip

The existence of many good Green-spheres with effectively controlled topology plays a central role also for the proof of the manifold recognition \autoref{thm:RCDtopma}, which we outline below.
We consider a three-dimensional $\RCD(-2,3)$ space $(X^3,\dist)$ whose tangent cones have cross-sections homeomorphic to $S^2$ at each point. The goal is to show that the set of non-manifold points $\mathcal{S}_{\rm top}(X)$ is empty. 

\medskip

As a first step we are going to prove that any such $(X^3,\dist)$ is locally uniformly contractible. Since the covering dimension of $X$ is equal to $3$, it is enough to show that $X$ is locally uniformly $3$-connected. In the proof, we will establish local uniform $k$-connectedness for every $k\le 3$ by a finite induction over $k$. The argument is based on three main points: 
\begin{itemize}
\item[(i)] that good Green-balls, i.e., the sub-level sets of Green-type distances bounded by good Green-spheres, are simply connected if their boundary is homeomorphic to $S^2$; 
\item[(ii)] that good Green-balls disjoint from $\mathcal{S}_{\rm top}(X)$ are contractible manifolds with boundary; 
\item[(iii)] that the contractibility part of the previous statement holds also without the latter restriction on the intersection with the non-manifold set.
\end{itemize}
\medskip

The first item in the above list corresponds to the step $k=1$, in the inductive proof of the local uniform contractibility. For the sake of illustration, we discuss the following simplified situation. We consider an $\RCD(0,3)$ space $(X^3,\dist,\haus^3)$ with Euclidean volume growth such that all the cross-sections of tangent cones at infinity are homeomorphic to $S^2$. In this case, we can prove that $X^3$ is simply connected as follows. Consider a global Green-type distance $b_p : X \to \mathbb{R}$ induced by the Green function of the Laplacian with pole at $p\in X$, and a scale $R>0$ big enough so that $B_R(p)$ is sufficiently close to a cone at infinity of $X$. We can assume without loss of generality that the Green-sphere $\mathbb{S}_R = \{b_p=R\}$ is a good slice for $b_p$ and hence it is homeomorphic to $S^2$ by the results of Section \ref{sec:topgoodlevels} and \cite{Petersen90}. We can lift $b_p$ to the universal covering space $\pi : \widetilde{X}^3 \to X^3$, obtaining a function $\widetilde{b}_p$ with similar properties. In particular, the level set $\widetilde{\mathbb{S}}_R = \{\widetilde{b}_p = R \}$ must be a connected topological surface homeomorphic to the cross-section of any blow-down of $\widetilde{X}$. On the other hand, the map $\pi : \widetilde{\mathbb{S}}_R \to \mathbb{S}_R$ is a covering map. Since $\mathbb{S}_R$ is simply connected, $\pi$ must be trivial, i.e., $X$ is simply connected. 
In Section \ref{subsec:locsc}, we will illustrate how to localize the previous argument, showing that Green-balls whose boundary Green-sphere is homeomorphic to $S^2$ are simply connected. 

It is worth remarking that this line of reasoning is different from the proof of the simple-connectedness of smooth complete three-manifolds with nonnegative Ricci curvature and Euclidean volume growth in \cite{Zhu93}. In that case, following the earlier work of Schoen and Yau \cite{SchoenYau}, the author argues that the universal covering space must be contractible in the first place by relying on the sphere theorem in $3$-manifolds topology. This is enough to rule out the presence of torsion elements in the fundamental group. The simple-connectedness then follows by \cite{Li86} or \cite{Anderson90}. In the present setting, we are forced to argue the other way around as it is not a priori clear whether the space is a $3$-manifold.

\medskip

In order to address item (ii) in the above list, we exploit in a different way the local splitting at all scales along good Green-spheres in order to check that they are tamely embedded in the manifold part of $X$. We stress that it is fundamental that this property holds everywhere in the manifold part and not only in the Reifenberg regular part, where an easier argument applies. The proof of this tameness property heavily relies on Bing's work \cite{Bingtame,Bing 1ULC, Bing appr} in $3$-manifolds topology; see Section \ref{subsec:topGgm} for the details. In particular, a good Green-ball that does not intersect $\mathcal{S}_{\rm top}(X)$ is a $3$-manifold whose boundary is homeomorphic to $S^2$. By (i), it is simply connected and hence by Poincar\'e--Lefschetz duality and Whitehead's theorem it is contractible. We note that items (i) and (ii) are already enough for the proof of the local uniform contractibility \autoref{thm:unifcontrtopRCD3}. The tameness of Green-spheres and their topological stability are also sufficient to prove the converse implication in \autoref{thm:RCDtopma}, from the manifold regularity to the structure of cross-sections. See the proof of \autoref{cor:impl1} for the details.

\medskip

The local contractibility in the general case where the non-manifold set is a priori not empty requires a few additional ideas. 

Arguing by contradiction and relying on Baire's category theorem, we can reduce ourselves to the case where the space is uniformly conical at all scales smaller than $1$ around points on the non-manifold set. Then we prove by induction over $k$ that each singular $k$-cycle with support contained in a good Green-ball is homologically trivial. The base step $k=1$ corresponds to item (i) above. For the inductive step the key point is to show that any such $k$-cycle supported in $\mathbb{B}_1(p)$ is homologous to a sum of $k$-cycles supported in a finite family of good Green-balls $\mathbb{B}_c(p_i)$ for some $c<1$. This procedure can then be iterated to obtain triviality in homology, with an argument similar in spirit to \cite{Perelmanmaximal} and the more recent \cite{WangRCD}; see \autoref{cinese.plus} for the precise technical statement. 

The key step mentioned above relies on a delicate covering argument, whose details are discussed in Section \ref{subsec:loccontr}. Roughly speaking, we split the given $k$-cycle into a homologous sum of $k$-cycles using the Mayer--Vietoris sequence and the inductive assumption. Then we rely on item (ii) to obtain triviality of those cycles whose support is far away from the non-manifold set on one hand, and on the uniform conicality to keep pushing the other cycles closer and closer to $\mathcal{S}_{\rm{top}}(X)$ on the other hand. We note that by the general theory of noncollapsed $\RCD$ spaces the set of non-manifold points $\mathcal{S}_{\mathrm{top}}(X)$ is contained in the effective singular stratum $\mathcal{S}^1_{\epsilon}(X)$ for some $\epsilon>0$. Moreover, at a conical scale the effective singular stratum is packed into a tubular neighbourhood of finitely many segments of size which is scale-invariantly much smaller than the scale: see Section \ref{subsec:partcontr} for the precise statements.

\medskip

Once the local uniform contractibility has been established, we prove that a punctured neighbourhood of each point deformation retracts onto a Green-sphere. In particular, it is homotopically equivalent to $S^2$ and hence the relative homology of $X$ at each point is the same as for $(\setR^3,\setR^3\setminus\{0\})$. In the language of geometric topology, we can prove that $X$ is a generalized $3$-manifold (without boundary), as discussed in Section \ref{subsec:locrelhom}.

\medskip

The second part of the proof of \autoref{thm:RCDtopma} amounts to upgrading the topological regularity from generalized manifold to manifold.
We address the reader to the survey papers \cite{Cannon} and \cite{Cavicchiolietal} and to Section \ref{subsec:recogn} below for some background about the recognition problem for topological manifolds among generalized manifolds. In the present setting, this step requires a few additional ideas with respect to those outlined above and it will heavily rely on the works of Thickstun \cite{Thickstuna,Thickstunb}, Daverman and Repovš \cite{DavermanRepovs}, and on Perelman's resolution of the Poincar\'e conjecture.

\medskip

To illustrate more concretely how to rule out the existence of non-manifold points, let us look at a vastly simplified scenario, where the use of the abstract recognition theorems from \cite{Thickstuna,Thickstunb,DavermanRepovs} can be avoided. Namely, we assume that $\mathcal{S}_{\rm{top}}(X)$ is discrete, i.e., the non-manifold points are isolated. We fix any $p\in \mathcal{S}_{\rm{top}}(X)$. For any given $\delta>0$ there is no harm in assuming that $B_s(p)$ is $\delta$-conical for every $s<r$, for some $r>0$, by \cite{DePhilippisGigli} (see also the earlier \cite{CheegerColding96}). In particular, we can construct a sequence of Green-spheres $\mathbb{S}_{r_i} \subset B_{2 r_i}(p)\setminus B_{r_i/2}(p)$ and consider the ``annular'' regions $A_i$ enclosed between two consecutive ones. The topological stability of Green-spheres and the assumption on the cross-sections of tangent cones ensure that all the $\mathbb{S}_{r_i}$ are homeomorphic to $S^2$. Moreover, they are tamely embedded in the manifold part of $X$. Hence the regions $A_i$ are topological manifolds with boundary, with two boundary components homeomorphic to $S^2$. It is worth stressing that this conclusion requires some ``exterior'' regularity of the embedding of the Green-spheres into $X$ to avoid pathological examples such as the Alexander horned sphere (cf.\ \autoref{rm:Alexander}). This moral point is important also in the general case.
By relying on the already established uniform local contractibility, we can argue that the regions $A_i$ are homotopically equivalent to $S^2$. Then we use the solution to the Poincaré conjecture to deduce that each $A_i$ is homeomorphic to the Euclidean annular region $\bar B_1(0^3)\setminus B_{1/2}(0^3)$. Finally, we can rescale and glue together these homeomorphisms to conclude that $B_r(p)$ is homeomorphic to the Euclidean ball $B_1(0^3)$, hence $p$ is a manifold point.

\medskip

The general case when $\mathcal{S}_{\rm{top}}(X)$ is not assumed to be discrete is considerably more delicate. One of the key morals is that the non-manifold set, if present, has general-position dimension one in $X$ according to \autoref{def:GPDO}. More precisely, we are going to prove that any continuous map $f:\overline{D}^2\to X$ can be approximated arbitrarily well with maps $f_{\epsilon}:\overline{D}^2\to X$ whose image intersects $\mathcal{S}_{\rm{top}}(X)$ at finitely many points. Verifying that this condition holds in the present setting is highly nontrivial, as $\mathcal{S}_{\rm{top}}(X)$ is only known to have Hausdorff dimension less than $1$ in this generality, and it requires a different use of the existence and structure of Green-spheres. We do not delve into the details here, referring the reader to Section \ref{sec:topman}.

\subsection{Related literature}

In this section we briefly review some of the literature related to the results and the methods of this paper, without the aim of being complete.

\medskip

The statements of \autoref{cor:4d} and \autoref{thm:omeo n-4 sym} in the case where the lower Ricci curvature bound is strenghtened to a two-sided bound on the Ricci curvature follow from the solution of the codimension four conjecture by Cheeger and Naber \cite{CheegerNaber15}. See also \cite{NaberJiang} for a sharper statement in this setting.

Analogous results for noncollapsed limits of K\"ahler manifolds with two-sided bounds on the Ricci curvature had been obtained earlier by Cheeger, Colding and Tian in \cite{CheegerColdingTian} and later refined by Cheeger in \cite{Cheeger03}. More recently Zhou has established in \cite[Appendix A]{Zhou23} a stronger version of \autoref{cor:4d} for noncollapsed limits of K\"ahler surfaces with Ricci curvature bounded from below, relying on earlier contributions due to Liu--Sz\'ekelyhidi \cite{LiuSzekeI,LiuSzekeII}. In that setting, the limit surfaces are homeomorphic to holomorphic orbifolds.

\medskip

For noncollapsed limits of manifolds with sectional curvature bounded from below, a theorem of Kapovitch \cite{Kapovitchcross} says that the cross-section of the tangent cone at any given point is homeomorphic to the sphere. Furthermore any such cross-section is a noncollapsed limit of smooth spheres with a uniform lower bound on the sectional curvature itself. This should be contrasted with the examples discussed in \autoref{rm:EguchiHanson} and \autoref{rm:exCN} in the context of lower Ricci curvature bounds.

\medskip

In the case of collapsed Ricci limit spaces, where the assumption on the lower volume bound in \eqref{Riccilimits} is dropped, the existence of an open and dense subset homeomorphic to a topological manifold has been an open question since the early developments of the theory: see for instance \cite[Open Problem 3.4]{Naberconj}. Recently, Hupp, Naber and Wang have constructed examples of (collapsed) Ricci limit spaces where no manifold point exists \cite{HuppNaberWang}, thus providing a negative answer to the question above when the rectifiable dimension is $\ge 4$. An even more recent example due to Zhou \cite{Zhoub} shows that we might have no manifold points also for collapsed Ricci limit spaces with rectifiable dimension $3$.

\medskip

Broadly speaking, \autoref{thm:RCDtopma}, \autoref{thm:unifcontrtopRCD3} and \autoref{thm:topstabintro} fit into the study of the topological structure of metric spaces satisfying some curvature bound in synthetic sense, such as Alexandrov, $\rm{CAT}$, and $\RCD$ spaces. This is a subject that has received a lot of attention in recent years: see for instance \cite{BuragoGromovPerelman92,Perelman99,KapovitchPer,LytchakNagano19,LytchakNagano22,KapovitchMondino,LytchakStadler,BrueNaberSemolabdry,WangRCD}.

\medskip

In the framework of Alexandrov spaces with curvature bounded from below, Perelman's conical neighbourhood theorem \cite{Perelman99,PerelmanII} guarantees that each point has a neighbourhood homeomorphic to the tangent cone at that point. In particular, the manifold set coincides with the set of points where the tangent cone is homeomorphic to the Euclidean space. The complement of this set has Hausdorff codimension at least three if the space has empty boundary. For the same reason, three-dimensional Alexandrov spaces are homeomorphic to orbifolds with boundary. \\
More in general, Perelman's stability theorem \cite{Perelman99,KapovitchPer} asserts that two Alexandrov spaces with curvature bounded below of the same dimension which are sufficiently close in the Gromov--Hausdorff sense are homeomorphic. 

For smooth Riemannian manifolds with sectional curvature and volume uniformly bounded from below, a weaker stability theorem with ``homotopically equivalent'' replacing ``homeomorphic'' had been established earlier by Grove and Petersen in \cite{GrovePetersen88}. The uniform local contractibility within this class (in any dimension) is due to Petersen \cite{Petersen90}. Our main results generalize some of these statements to three-dimensional $\RCD$ spaces. As we already argued, the analogous statements fail in dimension $n\ge 4$ in the case of Ricci curvature bounded from below.

\medskip

In the class of $\rm{CAT}$ spaces, i.e., metric spaces with sectional curvature bounded from above, Lytchak and Nagano have investigated the metric and topological structure in \cite{LytchakNagano19,LytchakNagano22}. Among other results, they prove a manifold recognition theorem stating that a locally compact space with an upper curvature bound is a topological manifold if and only if all of its spaces of directions are homotopy equivalent and not contractible. The use of the abstract manifold recognition theorems from geometric topology in the present work was partly inspired by \cite{LytchakNagano22}. There are also some analogies with the more recent work of Lytchak, Nagano, and Stadler on the topology of $\rm{CAT}(0)$ $4$-manifolds \cite{LytchakNaganoStadler}.

\subsection{Open questions}
We conclude the introduction with a list of open questions related to the topology of noncollapsed Ricci limits and $\RCD$ spaces that arose in this work, without aiming to be exhaustive.

\medskip

In the framework of four-dimensional Ricci limit spaces, \autoref{cor:4d} ensures that all cross-sections of tangent cones are homeomorphic to a spherical space form. 
The blow-down of the Eguchi--Hanson metric shows that space forms with nontrivial fundamental group such as $\mathbb{RP}^3$ might appear as cross-sections of tangent cones even in the Ricci flat setting. From \autoref{cor:4d}, we know that the set of points where the cross-sections are not homeomorphic to $S^3$ is zero dimensional.

\begin{conjecture}
    Let $(X^4,\dist)$ be a noncollapsed Ricci limit space. The set of points $x\in X^4$ where cross-sections of tangent cones are not homeomorphic to $S^3$ is discrete.
\end{conjecture}

The conjecture is satisfied for noncollapsed limits with two-sided bounds on the Ricci curvature by \cite{CheegerNaber15}, and for noncollapsed limits of K\"ahler surfaces with Ricci bounded from below by the very recent \cite{Zhou23}.
We believe that, more in general, in any dimension $n\ge 4$ the set of points where the cross-section of some tangent cone is not homeomorphic to $S^{n-1}$ should be $(n-4)$-rectifiable with locally finite $\haus^{n-4}$ measure.

\medskip

Kronheimer in \cite{Kronheimer} proved that, for every discrete group $\Gamma < \rm{U}(2)$ acting freely on $S^3$, there exists a Ricci flat $4$-manifold with Euclidean volume growth whose cone at infinity is isometric to $C(S^3/\Gamma)$.

\begin{remark}[Poincar\'e homology sphere]
Let $\Gamma$ be the binary icosahedral group. It is known that $\Gamma < \rm{SU}(2)$ and $S^3/\Gamma$ is the Poincaré homology sphere. In particular, by \cite{Kronheimer}, the cone over the Poincaré homology sphere is a noncollapsed Ricci limit space.
Notice that $C(S^3/\Gamma)$ is a contractible generalized $4$-manifold which is not a topological $4$-manifold.  
\end{remark}

To the best of our knowledge, there is presently no known restriction to the possible spherical space forms that might arise as cross-sections of tangent cones for nonncollapsed Ricci limits in dimension $4$. In this regard we pose the following.

\begin{question}
    Let $\Gamma < \rm{O}(4)$ be a discrete group acting freely on $S^3$. Is there an $\RCD(2,3)$ metric over $S^3/\Gamma$ such that $C(S^3/\Gamma)$ is a noncollapsed Ricci limit space?
\end{question}

The above question is a particular instance of the broad inquiry about the \emph{smoothability} of $\RCD$ spaces. Indeed, in the above setting $C(S^3/\Gamma)$ would be an $\mathrm{RCD}(0,4)$ space. In the context of smoothability, we ask the following.

\begin{question}\label{q:smoothing}
    Let $(X^3,\dist)$ be a noncollapsed $\RCD(-2,3)$ spaces. Assume that $X^3$ is a topological manifold. Is $(X^3,\dist)$ a noncollapsed Ricci limit space?
\end{question}
No restriction to the smoothability in this setting is known to the authors. We mention that the analogous question for Alexandrov spaces has been around since the eighties and remains widely open: see \cite[Question 1.9]{Kapovitchcross}, \cite[Section 13.7]{BuragoGromovPerelman92}, \cite[Open Problem 2.4]{Lebedevaetal}.

As a special case of \autoref{q:smoothing} one might consider the case where $(X^3,\dist)$ is the cross-section of the tangent cone of a nonncollapsed four-dimensional Ricci limit space. 

\medskip

By \cite{Simon14,SimonTopping22}, three-dimensional noncollapsed Ricci limit spaces are biH\"older homeomorphic to smooth manifolds. On the other hand, the manifold recognition \autoref{thm:RCDtopmain} only provides a $C^0$ structure.

\begin{conjecture}\label{conj:biholder}
Any noncollapsed $\RCD(-2,3)$ space with Euclidean tangent cones is biH\"older homeomorphic to a smooth Riemannian manifold.
\end{conjecture}
The proof of the topological rigidity of $\RCD(0,3)$ spaces with Euclidean volume growth that are homeomorphic to manifolds, as stated in \autoref{thm:volumegrowthR^3}, relies on the solution of the Poincaré conjecture. We believe that in order to make progress towards a resolution of \autoref{conj:biholder} it might be important to find a different proof avoiding the latter.

\begin{question}
Is there a proof of \autoref{thm:volumegrowthR^3} that does not rely on the solution to the Poincar\'e conjecture?
\end{question}
The fact that a smooth complete three-manifold with nonnegative Ricci curvature and Euclidean volume is homeomorphic to $\setR^3$ follows from \cite{Liu13}, which however relies on the solution of the Poincar\'e conjecture. See also the earlier work of Schoen and Yau \cite{SchoenYau} for the case of positive Ricci curvature. The statement (in the smooth case) should also follow from \cite{SimonTopping22} taking into account that each blow-down of any such manifold is homeomorphic to $\setR^3$, by \cite{CheegerColding96,Ketterer15,LytchakStadler} and \cite{SimonTopping22} again.

\subsection*{Acknowledgements}

EB is grateful for the financial support provided by Bocconi University. Additionally, part of this work was conducted during EB's visit to FIM-ETH Zürich, and he acknowledges their financial support. 
\\
DS is supported by the FIM-ETH Z\"urich through a Hermann Weil Instructorship. Part of this work was completed during a visit of DS at Bocconi University, whose financial support is gratefully acknowledged.  
\\
We thank Camillo Brena for useful discussions and for suggesting numerous corrections on a preliminary version of the paper.

\section{Notation and terminology}

\begin{tabular}{p{1.75cm}p{10cm}p{1cm}}
$(X,\dist)$ & Metric space \\
$\dist_{\rm{GH}}(X,Y)$ & Gromov--Hausdorff distance between $X$ and $Y$\\
$\haus^n$ & Hausdorff measure of dimension $n$\\
$\omega_n$ & Lebesgue measure of the unit ball in $\setR^n$\\
$B_r(x)$ & Open ball of radius $r$ centered at $x\in X$\\
$\overline{B}_r(x)$&  Closed ball of radius $r$ centered at $x\in X$\\
$C(Z)$ & Metric cone over $Z$\\
$\mathcal{R}(X)$& Regular set of $X$\\
$\mathcal{R}_\eps(X)$ & Quantitative regular set of $X$\\
$\mathcal{S}^k(X)$& $k$-dimensional singular set of $X$\\
$\mathcal{S}^k_\eps(X)$& $k$-dimensional effective singular set of $X$\\
$\mathcal{R}_{\rm top}(X)$& Set of manifold points of $X$\\
$\mathcal{S}_{\rm top}(X)$& Set of non-manifold points of $X$\\
$\mathcal{R}_{\rm gm^+}(X)$ & Generalized manifold points of $X$\\
$\mathcal{S}_{\rm gm^+}(X)$ & Non-generalized manifold points of $X$\\
$\mathcal{G}_p$ & Set of good radii at $p$\\
$\mathbb{B}_r(p)$ & Green-ball centered at $p$ with radius $r$\\
$\mathbb{S}_r(p)$ & Green-sphere centered at $p$ with radius $r$
\end{tabular}

\section{Preliminaries}\label{sec:prel}

In this preliminary section we gather some mostly well-known results that will be important later throughout this work. We assume the reader to have some familiarity with the notion of $\RCD$ space.

\subsection{Two-dimensional $\RCD$ spaces}\label{subsec:2dRCD}

The following is the main result of \cite{LytchakStadler}. It confirms that in dimension two a lower bound on the Ricci curvature and a lower bound on the sectional curvature are the same also in the singular case.

\begin{theorem}\label{thm:RCD2Alex}
Let $(X,\dist,\haus^2)$ be an $\RCD(K,2)$ metric measure space for some $K\in\setR$. Then $(X,\dist)$ is an Alexandrov space with curvature bounded from below by $K$.
\end{theorem}

The well-established theory of Alexandrov surfaces then yields the following.

\begin{corollary}\label{cor:Alsurf}
Let $(X,\dist,\haus^2)$ be an $\RCD(K,2)$ metric measure space for some $K\in\setR$. Then $X$ is a topological surface, possibly with boundary.
\end{corollary}

\begin{proof}
The statement is well known and originally due to Alexandrov. However, it follows also from Perelman's conical neighbourhood theorem and the elementary classification of cones in dimension $2$: see \cite[Theorem 10.10.2, Corollary 10.10.3]{BuragoBuragoIvanov} and the references therein indicated. 
\end{proof}

\begin{corollary}\label{cor:2dK>0}
Let $(X,\dist,\haus^2)$ be an $\RCD(1,2)$ metric measure space. If $X$ has empty boundary then it is homeomorphic either to $S^2$ or to $\mathbb{RP}^2$. 
\end{corollary}

\begin{proof}
By \autoref{thm:RCD2Alex}, \autoref{cor:Alsurf} and the Bonnet--Myers theorem for Alexandrov spaces, see \cite[Theorem 8.44]{AlexanderKapovitchPetrunin}, the universal cover of $(X,\dist)$ is a compact surface without boundary. Hence, it is homeomorphic to $S^2$. The statement follows.
\end{proof}

\begin{corollary}\label{cor:2dAVR}
Let $(X,\dist,\haus^2)$ be an $\RCD(0,2)$ metric measure space with quadratic volume growth. If $X$ has empty boundary, then it is homeomorphic to $\setR^2$. 
\end{corollary}

\begin{proof}
By \cite[Theorem 1.6]{MondinoWei}, $X$ has finite fundamental group. 
The conclusion is certainly classical, and can be obtained as follows.
Endowing $X$ with a structure of Riemann surface, its universal cover $\tilde X$ is either $\mathbb{C}$ or $\mathbb{H}$, since $X$ is not compact. To conclude that $\tilde X=X$, it suffices to check that an automorphism of $\mathbb{C}$ or $\mathbb{H}$ with finite order and no fixed points must be the identity. In the case of $\mathbb{C}$, this is clear since automorphisms without fixed points are translations. In the case of $\mathbb{H}$, an automorphism has the form $z\mapsto\frac{az+b}{cz+d}$, with $a,b,c,d\in\R$ and $ad-bc=1$;
if it has finite order, then the matrix $A$ with these entries has $A^k=\pm I$ for some $k\ge1$. Hence, $A$ is diagonalizable and, unless $A=\pm I$, its eigenvalues are roots of unity different from $\pm1$. In particular, the eigenvalues are not real: this is equivalent to $(a-d)^2+4bc<0$, which in turn implies the existence of $z\in\mathbb{H}$ such that $\frac{az+b}{cz+d}=z$.
\end{proof}

\begin{proposition}\label{prop:Ale2loccontr}
Let $K\in\setR$ and $v>0$ be fixed. Let $\mathcal{F}_{K,v}$ be the class of all $\RCD(K,2)$ spaces such that $\haus^2(B_1(p))\ge v$ for any $p\in X$. Then $\mathcal{F}_{K,v}$ is a family of locally uniformly contractible metric spaces. 
\end{proposition}

\begin{proof}
The statement is well known. It follows for instance from the analogous statement for smooth Riemannian manifolds in \cite{GrovePetersen88} and the fact that two-dimensional Alexandrov spaces with curvature bounded from below can be approximated in the Gromov--Hausdorff sense by smooth Riemannian $2$-manifolds with a uniform lower bound on the sectional curvature.
\end{proof}

\begin{definition}\label{def:effreg}
Let $(X,\dist,\haus^n)$ be an $\RCD(K,n)$ space. Given $\epsilon>0$ and $r>0$ we shall denote
\begin{equation}
\mathcal{R}_{\epsilon,r}(X)=\mathcal{R}_{\epsilon,r}:=\{x\in X\, :\, \dist_{\rm{GH}}(B_s(x),B_s(0^n))<\epsilon s\text{ for all } 0<s<2r \}\, ,
\end{equation}
and 
\begin{equation}
\mathcal{R}_{\epsilon}(X)=\mathcal{R}_{\epsilon}:=\bigcup_{r>0}\mathcal{R}_{\epsilon,r}\, .
\end{equation}
\end{definition}

A consequence of the metric Reifenberg theorem \cite{CheegerColding97I} and of the estimates for the quantitative strata of Alexandrov spaces with curvature bounded below (see \cite[Theorem 1.3]{LiNaber}) that will be particularly relevant for us is the following.

\begin{proposition}\label{prop:sing2dsmall}
Let $v>0$ and $\epsilon<\epsilon_0$ be fixed such that Reifenberg's theorem applies. There exists $N=N(\epsilon)$ such that the following holds.
For every Alexandrov space $(X,\dist,\haus^2)$ with curvature bounded below by $-1$, empty boundary, and such that $\haus^2(B_1(p))\ge v$, for every $r<r_0(v)$ there exists a finite set $\{x_i\}\subset B_1(p)$ with less than $N$ elements such that
\begin{equation}
\left(X\setminus \mathcal{R}_{\epsilon,r}\right)\cap B_1(p)\subset \bigcup_i B_{3r}(x_i)\, .
\end{equation}
\end{proposition}

\begin{remark}\label{rm:finitesingset}
Note that $N=N(\epsilon)$ is independent of $r$ in the statement of \autoref{prop:sing2dsmall}. In particular, the statement is an effective version of the (more classical) fact that, for an Alexandrov surface with curvature bounded from below by $-1$, for every $\epsilon>0$ the number of singular points $x\in B_1(p)$ with cone angle less than $2\pi-\epsilon$ is uniformly bounded by $N(\epsilon)$.
\end{remark}

\subsection{Almost volume-cone implies almost metric-cone and consequences}\label{subsec:almostvcalmostmc}

We address the reader to \cite[Chapter 3]{BuragoBuragoIvanov} for the relevant terminology about metric cones over metric spaces.  

It follows from \cite{Ketterer15} that a metric measure cone $(C(Y),\dist_{C(Y)},\haus^n)$ is an $\RCD(0,n)$ space if and only if the cross-section $(Y,\dist_Y,\haus^{n-1})$ is an $\RCD(n-2,n-1)$ space. 

A standard compactness argument, in combination with the so-called ``volume-cone implies metric-cone'' theorem from \cite{DePhilippisGigli}, leads to the following ``almost volume-cone implies almost metric-cone'' theorem for $\RCD$ spaces. A similar statement was obtained with a completely different method for smooth Riemannian manifolds earlier in \cite{CheegerColding96}. We address the reader to \cite{AntonelliBrueSemola} for a detailed proof.

\begin{theorem}\label{thm:almvcalmmc}
Let $\delta>0$, $v>0$ and $n\in\setN$, $n\ge 2$ be fixed. There exists $\epsilon=\epsilon(\delta,v,n)>0$ such that for every $\RCD(-\epsilon^2,n)$ space $(X,\dist,\haus^n)$ with
\begin{equation}
\haus^n(B_1(p))\ge v\, 
\end{equation}
the following holds. If 
\begin{equation}
\frac{\haus^n(B_4(p))}{\haus^n(B_2(p))}\ge (1-\epsilon)2^n\, ,
\end{equation}
then there exists an $\RCD(n-2,n-1)$ space $(Y,\dist_Y,\haus^{n-1})$ with $\rm{diam}(Y)\le \pi$ such that 
\begin{equation}
\dist_{\rm{GH}}(B_1(p),B_1(o))\le \delta\, ,
\end{equation}
where $B_1(o)\subset C(Y)$ is the ball centered at the tip.
\end{theorem}

\begin{definition}
Let $(X,\dist,\haus^n)$ be an $\RCD(-\delta^2,n)$ space and let $p\in X$. We shall say that the ball $B_1(p)\subseteq X$ is \emph{$\delta$-conical} (or \emph{$(0,\delta)$-symmetric}) if there exists an $\RCD(n-2,n-1)$ space $(Y,\dist_Y,\haus^{n-1})$ such that 
\begin{equation}
\dist_{\rm{GH}}(B_1(p),B_1(o))\le\delta\, ,
\end{equation}
where $B_1(o)\subset C(Y)$ is the ball centered at the tip. An analogous terminology will be employed for balls of any radius $r>0$, with the understanding that the conditions hold after scaling to radius $1$.
\end{definition}

We are going to heavily rely on two by now classical consequences of \autoref{thm:almvcalmmc}. We address the reader to \cite{CheegerColding96,CheegerColding97I} and \cite{CheegerNaber13}, where the arguments originated from, and to \cite{DePhilippisGigli2,AntonelliBrueSemola} for the present setting.

\begin{lemma}\label{lemma:manyconsc}
Let $\delta>0$ and $v>0$ be fixed. There exists $C=C(n,\delta,v)\ge 1$ such that for every $\RCD(-(n-1),n)$ space and for every $p\in X$ such that $\haus^n(B_1(p))\ge v$ the following holds. For every $0<r<1$ there exists $r<r'<Cr$ such that the ball $B_{r'}(p)$ is $\delta$-conical.
\end{lemma}

\begin{lemma}\label{lemma:infveryc}
Let $(X,\dist,\haus^n)$ be an $\RCD(-(n-1),n)$ space and let $\delta>0$. For any $p\in X$ there exists $r_0=r_0(p,\delta)>0$ such that for every $0<r<r_0$ the ball $B_r(p)$ is $\delta$-conical.
\end{lemma}

\subsection{Almost splitting maps}

We provide an overview of $\delta$-splitting maps  and the almost splitting theorems within the framework of $\RCD$ spaces. They are classical tools in the study of spaces with Ricci curvature bounded below \cite{CheegerColding96,CheegerColding97I,CheegerColding00II,CheegerNaber15}. Our presentation will mainly follow \cite{BrueNaberSemolabdry}.

\begin{definition}\label{spl.def}
	Let $(X,\dist,\meas)$ be an $\mathrm{RCD}(K,N)$ space for some $K\in\R$ and some $1\le N <\infty$. Let $p\in X$ and $s>0$. A map
	$u :B_{2s}(p)\to\R^k$ is a \emph{$\delta$-splitting} map if it belongs to the domain of the local Laplacian on $B_{2s}(p)$, and
	\begin{align}\label{eq:splitting}
    \begin{aligned}
		&|\nabla u|\le C(N)\, \quad\text{on $B_s(p)$}\, ,
		\\ &\sum_{a,b=1}^k\fint_{B_s(p)}|\ang{\nabla u^a,\nabla u^b}-\delta^{ab}|\le\delta^2\, ,
		\\ &\sum_{a=1}^k s^2\fint_{B_s(p)}|\operatorname{Hess}(u^a)|^2\le\delta^2\, .
    \end{aligned}
	\end{align}
\end{definition}

\begin{remark}
    As clarified later on,
    the constant $C(N)$ appearing in the definition is a (computable) constant related to the one appearing in \autoref{lemma:Mosersharper}.
\end{remark}

\begin{remark}[Harmonic almost splitting]
In the literature, it is often assumed that $\delta$-splitting maps are harmonic. Under this assumption, the Hessian bound in equation \eqref{eq:splitting} arises as a consequence of the $L^1$ gradient estimate. However, for our purposes, we find it more convenient to drop the harmonicity assumption and directly assume that the Hessian is small in $L^2$.

\end{remark}

\begin{remark}[Sharp gradient bound]
	\label{rmk:sharp_gradient}
	If $(X,\dist,\meas)$ is an $\RCD(-\delta(N-1),N)$ and $u:B_{2}(p) \to \R^k$ is harmonic, then the gradient bound in \eqref{eq:splitting} can be sharpened to
	\begin{equation}
		\sup_{B_{1}(p)}| \nabla u| \le 1 + C(N)\delta^{1/2} \, .
	\end{equation}
	See \cite[eqs.\ (3.42)--(3.46)]{CheegerNaber15} and \cite[Remark 3.3]{BrueNaberSemolabdry}.
\end{remark}

It is now a classical result that the existence of an almost splitting function is equivalent to Gromov--Hausdorff closeness to a space that splits a Euclidean factor. The following statement corresponds to \cite[Theorem 3.8]{BrueNaberSemolabdry}. For similar statements concerning smooth manifolds and Ricci limits, we refer to \cite{CheegerColding96}, \cite[Lemma 1.21]{CheegerNaber15}, and \cite[Theorem 4.11]{CheegerJiangNaber}.

\begin{theorem}[$\delta$-splitting vs $\eps$-isometry]\label{thm:splittings}
	Let $1\le N <\infty$ be fixed. 
	\begin{itemize}
		\item[(i)] For every $\delta>0$ and $\eps<\eps_0(N,\delta)$, if $(X,\dist,\meas)$ is an $\RCD(-\eps(N-1), N)$ m.m.s.\ satisfying
		\begin{equation}
			\dist_{\rm{mGH}}(B_2(p),B_2(0,z))\le \eps \, , \quad (0^k,z)\in \R^k \times Z \, ,
			\quad p\in X
		\end{equation}
		for some m.m.s.\ $(Z,\dist_Z,\meas_Z)$, then
		there exists a harmonic $\delta$-splitting map 
		\begin{equation}
			u :B_1(p)\to \R^k \, .
		\end{equation}

		\item[(ii)] For every $\eps>0$ and $\delta < \delta_0(N,\eps)$, if  $(X,\dist,\meas)$ is an $\RCD(-\delta(N - 1), N)$ m.m.s.\ and there exists a $\delta$-splitting map $u : B_6(p)\to \R^k$ for a given $p\in X$, then
		\begin{equation}
			\dist_{\rm{GH}}(B_{1/k}(p),B_{1/k}(0^k,z))\le \eps\, , \quad
			(0^k,z)\in \R^k \times Z
		\end{equation}
		for some $\RCD(0,N-k)$ m.m.s.\ $(Z,\dist_Z, \meas_Z)$. Moreover, there exists $f:B_{1}(p)\to Z$ such that
		\begin{equation}
			(u-u(p),f):B_{1/k}(p)\to B_{C(n)\delta+1/k}(0^k,z)
			\quad \text{is an $\eps$-GH isometry.}
		\end{equation} 
	\end{itemize}
\end{theorem}

\subsection{General topological properties of $\RCD(K,n)$ spaces}

	Recall that a topological space $Y$ has covering dimension $n\ge-1$
	if each finite open cover admits a refinement for which at most $n+1$ sets intersect, and $n$ is the least integer such that this holds.
	The finiteness of the cover is a superfluous assumption for separable metric spaces: see \cite[Exercise 1.7.E and Proposition 3.2.2]{Engelking}.

The following lemma is certainly well known to experts. We sketch a proof for the sake of completeness.

	\begin{lemma}\label{lemma:dimcov}
		Let $(X,\dist,\haus^n)$ be an $\RCD(-(n-1),n)$ space. The covering dimension $\operatorname{dim}_c(U)$ of any nonempty open set $U\subseteq X$ is $n$.
	\end{lemma}

	\begin{proof}
		By the topological manifold regularity of $\RCD(K,n)$ spaces $(X,\dist,\haus^n)$ on an open dense set (see \cite{KapovitchMondino} after \cite{CheegerColding97I}), $U$ contains a compact subset homeomorphic to the closed ball $\overline D^n$.
		As $\operatorname{dim}_c$ decreases when passing to closed subsets, we have $\operatorname{dim}_c(U)\ge\operatorname{dim}_c(\overline D^n)=n$.
		On the other hand, since the Hausdorff dimension $\operatorname{dim}_{\haus}(U)\le n$, we also have $\operatorname{dim}_c(U)\le\operatorname{dim}_{\haus}(U)\le n$.
		Indeed, as shown in \cite[Section VII.2]{HurewiczWallman}, the so-called small inductive dimension $\operatorname{dim}_{si}(U)$ (used throughout the book \cite{HurewiczWallman}) satisfies $\operatorname{dim}_{si}(U)\le\operatorname{dim}_{\haus}(U)$, and moreover by \cite[Theorem V.1]{HurewiczWallman}
		we have $\operatorname{dim}_c(U)\le\operatorname{dim}_{si}(U)$ (in fact, the last inequality turns out to be an equality by Urysohn's theorem).
	\end{proof}

\subsection{Generalized manifolds and manifold recognition}\label{subsec:recogn}

In this section we gather some background material about the recognition problem for topological manifolds among generalized manifolds. Recall that a topological space $X$ is \emph{locally contractible} if, for any $p\in X$ and any open neighbourhood $U$ of $p$,
there exists an open neighbourhood $V\subseteq U$ such that the inclusion $V\hookrightarrow U$ is homotopic to a constant among maps $V\to U$.

\begin{definition}[Generalized manifold]\label{def:genmanifold}
Let $(X,\dist)$ be a metric space. We say that $X$ is a \emph{generalized $n$-manifold} if it is locally compact, locally contractible, finite-dimensional (in the sense of the covering dimension) and it has the local relative homology of $\setR^n$, i.e., the groups $H_{*}(X,X\setminus\{x\};\Z)$ are isomorphic to $H_{*}(\setR^n,\setR^n\setminus\{0\};\Z)$ for all $x\in X$. 
\end{definition}

Given a metric space $(X,\dist)$ and $n\in\setN$, we shall denote by $\mathcal{S}_{\rm{top},n}(X)$ the set of points $x\in X$ that have no neighbourhood homeomorphic to $\setR^n$. The dimension $n$ will be often suppressed when it is clear from the context. 
\medskip

The recognition of genuine manifolds among generalized manifolds has two steps. The first one is the resolution step according to the following.

\begin{definition}\label{def:resolvable}
A generalized $n$-manifold $X$ is said to be \emph{resolvable} if there exist an $n$-manifold $N$ and a proper cell-like surjective continuous map $f:N\to X$. In this case, the map $f$ is called a \emph{resolution} of $X$.
\end{definition}

We do not discuss the general definition of {\it cell-like} maps and just mention that it amounts to ask that the preimage of each point is cell-like, a slightly weaker notion compared to contractibility.

\medskip

When $n=3$, the main tool for the resolution step in our setting will be borrowed from \cite{Thickstunb} (see also the earlier work \cite{Thickstuna} where some of the ideas employed in \cite{Thickstunb} were originally introduced). In order to state the resolution criterion we need some terminology.

\begin{definition}\label{def:GPDO}
If $X$ is a generalized $3$-manifold and $A\subset X$ is a closed subset, we say that $A$ has \emph{general-position dimension one} in $X$ if any continuous map $f:\overline D^2\to X$ can be approximated arbitrarily well by maps $g:\overline D^2\to X$ such that $g(\overline D^2)\cap A$ is $0$-dimensional.
\end{definition}

In \cite[p.\ 68]{Thickstunb} it is stated that, modulo the resolution of the Poincar\'e conjecture, any compact generalized $3$-manifold whose singular set $\mathcal{S}_{\rm{top}}$ has general-position dimension one is resolvable. Hence, after Perelman's resolution of the Poincar\'e conjecture \cite{PerelmanPoincare}, we have the following.

\begin{theorem}\label{thm:recog2Thick}
Let $(X,\dist)$ be a compact generalized $3$-manifold. Assume that $\mathcal{S}_{\rm{top}}(X)$ has general-position dimension one. Then $X$ admits a resolution. 
\end{theorem}

The second step of the recognition problem in our setting (i.e., from resolvable generalized $3$-manifold to genuine $3$-manifold) will be accomplished thanks to the results in \cite{DavermanRepovs}. Again, we need to introduce some terminology. 

\begin{definition}\label{def:kcoco}
Let $(X,\dist)$ be a metric space. A subset $C\subseteq X$ is said to be \emph{locally $k$-coconnected} (abbreviated to $k$-LCC) if every neighbourhood $U\subseteq X$ of an arbitrary point $x\in X$ contains another neighbourhood $V\subseteq X$ such that all continuous maps $\partial I^{k+1}\to V\setminus C$ extend to maps $I^{k+1}\to U\setminus C$, where $I:=[0,1]$.
\end{definition}

Combining \cite[Proposition 1.2]{DavermanRepovs} with \cite[Theorem 3.4]{DavermanRepovs} we obtain the following.

\begin{theorem}\label{thm:recog2DR}
A resolvable generalized $3$-manifold $(X,\dist)$ is a $3$-manifold if any $x\in X$ is $1$-LCC and admits arbitrarily small neighbourhoods $U$ such that there exist maps $f:S^2\to U\setminus\{x\}$ with the following properties:
\begin{itemize}
\item[(i)] $f:S^2\to f(S^2)\subset X$ is a homeomorphism;
\item[(ii)] $f(S^2)$ is $1$-LCC in $X$;
\item[(iii)] $f:S^2\to U$ is homotopically trivial;
\item[(iv)] $f:S^2\to U\setminus\{x\}$ is not homotopically trivial.
\end{itemize}
\end{theorem}

\section{The Green-type distance}
\label{sec:Green}

This is the first of four sections where we study smooth complete Riemannian manifolds $(M^n,g)$ with $\Ric\ge -\delta(n-1)$ such that $B_{100}(p)\subset M^n$ is $\delta$-Gromov--Hausdorff close to a ball centered at a vertex of an $(n-3)$-symmetric cone $\setR^{n-3}\times C(Z^2)$, with $(Z^2,\dist_Z)$ an Alexandrov surface with curvature $\ge 1$. Our goal is to find a regular function $w:B_{100}(p)\to \setR^{n-2}$ such that for many of its level sets (in a measure theoretic sense): 
\begin{itemize}
\item[(i)] the function $w$ is almost splitting up to composition with a linear transformation for every point in the level set and for any sufficiently small scale;
\item[(ii)] the level set is Gromov--Hausdorff close and homeomorphic to the cross-section $Z$;
\item[(iii)] the level set is locally uniformly contractible.
\end{itemize}
All the statements above will be effective. We will also have an analogous more general statement in dimension $3$ valid for possibly non-smooth $\RCD(-\delta,3)$ spaces $(X,\dist,\haus^3)$. Both tools will be fundamental later on for studying the topological regularity and stability of noncollapsed spaces with lower Ricci bounds.

\medskip

The broad moral is similar to the slicing theorem due to Cheeger and Naber and its use in the proof of the codimension four conjecture for noncollapsed limits of manifolds with bounded Ricci curvature \cite{CheegerNaber15}. In our setting, the proof of (i) will amount to a generalization of their techniques to the case where one of the components of the map $w$ is not a harmonic almost splitting function but it is derived from the Green function of the Laplacian instead.

\medskip

In this section, we start by studying the fine properties of the \textit{local Green-type distance} within the framework of $\RCD$ spaces. It is well known that the Green function of the Laplacian enjoys effective bounds on spaces with Ricci curvature bounded below: see for instance \cite{Varopoulos,LiYau,ColdingGreen}. This makes it a powerful tool, often used to construct good regularizations of the distance function, as in \cite{ColdingGreen,ColdingMinicozziEinstein,NaberJiang}.

In the next two sections we will discuss the relevant generalization of the slicing theorem to the present setting, thus addressing the first item in the above list. The precise statements will be \autoref{trans} and \autoref{trans n=3}.

The uniform control on the topology of good level sets as in (ii) and (iii) above requires some new ideas with respect to \cite{CheegerNaber15}, where the two-sided Ricci curvature bound was heavily exploited for the analogous aim. This will be the subject of Section \ref{sec:topgoodlevels}.

\medskip

Given $n\ge3$, let us begin by considering an $\RCD(0,n)$ m.m.s.\ $(X,\dist,\haus^n)$ with Euclidean volume growth, i.e., such that 
\begin{equation}
	\lim_{r\to \infty}\frac{\haus^n(B_r(p))}{\omega_n r^n} =:\theta>0\quad \text{for some and hence for all }p\in X \, .
\end{equation}
We denote by $p_t(x,y)$ the heat kernel of the Laplacian in $X$.
If $n\ge 3$, it is known that
\begin{equation}
	G_p(x):=\int_0^{\infty}p_t(p,x)\di t\,,
\end{equation}
for $p,x\in X$,
defines a {\it positive Green function of the Laplacian}, i.e.,
\begin{equation}\label{eq:Green1}
	- \Delta G_p = \delta_p \, ,
\end{equation}
where the identity is understood in duality with test functions \cite{Gigli18}
\begin{equation}\label{eq:test}
\Test(X):=\{f\in L^\infty \cap \Lip\cap D(\Delta)\, : \, \Delta f \in W^{1,2}(X)\}\, .
\end{equation}
Moreover, $G_p$ satisfies uniform estimates for $x\neq p$:
\begin{align}
\begin{aligned}
	&\frac{C(n,\theta)^{-1}}{\dist(p,x)^{n-2}}
	\le 
	G_p(x)
	\le \frac{C(n,\theta)}{\dist(p,x)^{n-2}}\, ,\quad|\nabla G_p(x)|\le \frac{C(n,\theta)}{\dist(p,x)^{n - 1}}\, .
\end{aligned}
\end{align}
Following \cite{ColdingGreen}, the {\it Green-type distance} at $p\in X$ is then defined as
\begin{equation}\label{eq:b_pdef1}
	b_p := [n(n-2)\omega_n\theta G_p]^{-\frac{1}{n-2}} \, .
\end{equation}
The normalization constant in \eqref{eq:b_pdef1} is chosen so that on a cone it holds $b_p = \dist_p$ for the standard Green function, as in the remark below.

\begin{remark}\label{rmk:greencone}
	Let $(Z,\dist_Z,\haus^{n-1})$ be an $\RCD(n-2,n-1)$ m.m.s.\ for some $n\ge 3$. The metric cone over $Z$, $(C(Z),\dist_{C(Z)},\haus^n)$, has Euclidean volume growth with
	\begin{equation}
		\theta = \lim_{r\to 0}\frac{\haus^n(B_r(p))}{\omega_n r^n}=\frac{\haus^{n-1}(Z)}{n\omega_n}
	\end{equation}
	where $p\in C(Z)$ is the tip of the cone.
	In this case a positive Green function of the Laplacian with pole at $p$ is given by
	\begin{equation}
		G_p(x)=\frac{\dist(p,x)^{2-n}}{n(n-2)\omega_n\theta}\, .
	\end{equation}
\end{remark}

\subsection{Local Green distance}
It is well known that the construction of a positive Green function and the associated Green-type distance can be localized on balls. In this section, we sketch this construction and discuss the relevant estimates on $\RCD$ spaces.

\begin{proposition}\label{prop:Greenexistence}
	Let $(X,\dist,\haus^n)$ be an $\RCD(K,n)$ space. Assume that $\haus^n(X\setminus B_2(p))>0$ and $\haus^n(B_1(p))\ge v>0$, for some $p\in X$. For every $c > 0$ there exists a local Green function $G_p:B_2(p)\to [c,\infty)$ such that the following hold:
	\begin{itemize}
		\item[(i)] $-\Delta G_p = \delta_p$ on $B_1(p)$;
		
		\item[(ii)] $G_p = c$ on $\partial B_2(p)$;
		
		\item[(iii)] there exists $C=C(K,n,v,c)>1$ such that
		\begin{equation}
			\frac{1}{C\dist(p,x)^{n-2}}\le G_p(x) \le \frac{C}{\dist(p,x)^{n-2}}\, ,
			\quad \text{for all $x\in B_1(p)\setminus \{p\}$}\, ;
		\end{equation}
		\item[(iv)] $G_p$ is locally Lipschitz on $B_1(p)\setminus\{p\}$ and 
		\begin{equation}
			|\nabla G_p(x)|\le \frac{C(K,n,v,c)}{\dist(p,x)^{n-1}}\, ,\quad \text{for $\haus^n$-a.e.\ $x\in B_1(p)$}\, .
		\end{equation}
	\end{itemize}
\end{proposition}

\begin{proof}
	We only sketch the construction, which is standard.

	We set $G^1_p(x):= \int_0^1 p_t(p,x)\, \di t$. Let $G^2_p(x)$ be the solution of the Dirichlet problem
	\begin{equation}
		\Delta G^2_p(x) = p_1(p,x) \quad \text{in $B_1(p)$}\, , \quad
		G_p^2(x) - G_p^1(x) + c \in H_0^{1,2}(B_2(p))\, ,
	\end{equation}
	as in \cite[Theorem 2.58]{MondinoSemola23}.
	It is clear that $G_p:= G_p^1 - G_p^2$ satisfies (i) and (ii). As a consequence of the maximum principle, $G_p \ge c > 0$ on $B_2(p)$. Standard heat flow estimates and gradient comparison imply the following:

	\begin{itemize}
	    \item[(i')]  $G_p^1$ is locally Lipschitz in $X\setminus \{p\}$ and
		\begin{equation}
			\frac{C(K,n,v)^{-1}}{\dist(p,x)^{n-2}}
			\le 
			G_p^1(x)
			\le \frac{C(K,n,v)}{\dist(p,x)^{n-2}}\, ,\quad
			|\nabla G_p^1(x)|\le \frac{C(K,n,v)}{\dist(p,x)^{n - 1}}\, ;
		\end{equation}
		
		\item[(ii')] $G_p^2$ is locally Lipschitz in $B_2(p)$ and
		\begin{equation}
			|G_p^2| + |\nabla G_p^2| \le C(K,n,v,c) \, ,
			\quad \text{in $B_{4/3}(p)$}\, .
		\end{equation}
	\end{itemize}
	It is immediate to check that (iii) and (iv) follow from (i'), (ii')
	and the lower bound $G_p \ge c$ on $B_2(p)$.
\end{proof}

\begin{definition}[Local Green-type distance]\label{def:Gdistance}
	Let $(X,\dist,\haus^n)$ be an $\RCD(K,n)$ space for some $n\ge 3$. Let  $G_p : B_r(p)\to (0,\infty)$ be a local positive Green function of the Laplacian. We define the associated local \emph{Green-type distance} $b_p:B_r(p)\to [0,\infty)$ by
	\begin{equation}\label{eq:b_pdef}
		b_p:=\left[n (n-2)\frac{\haus^n(B_r(p))}{r^n}G_p\right]^{-\frac{1}{n-2}}\, .
	\end{equation}
\end{definition}

\begin{remark}
	The normalization constant in \eqref{eq:b_pdef} is chosen so that on a cone it holds $b_p=\dist_p$ if $p$ is the tip point, for a suitable $c$ (compare with \autoref{rmk:greencone}).
\end{remark}

\begin{corollary}\label{cor:boundb}
	Let $(X,\dist,\haus^n)$ be an $\RCD(-(n-1),n)$ space for some $n\ge 3$. Assume that $\haus^n(X\setminus B_2(p))>0$ and $\haus^n(B_1(p))\ge v$. Let $c > 0$ and $G_p:B_2(p) \to [c,\infty)$ be as in \autoref{prop:Greenexistence}.
	Then, for the Green-type distance $b_p:B_1(p)\to [0,\infty)$ introduced in \eqref{eq:b_pdef}, the following hold:
	\begin{itemize}
		\item[(i)] $b_p = [n(n-2)\haus^n(B_1(p))c]^{-\frac{1}{n-2}}$ on $\partial B_2(p)$;
		
		\item[(ii)] there exists $C = C(n,v,c)>1$ such that
		\begin{equation}
			C^{-1}\dist_p
			\le 
			b_p
			\le C\dist_p\, , \quad \text{on $B_{1}(p)$} 
		\end{equation}
  and 
  \begin{equation}\label{eq:Lipbp}
      |\nabla b_p|\le C\, ,\quad \text{on $B_1(p)$}\, ;
  \end{equation}
		\item[(iii)] $b_p$ belongs to the domain of the Laplacian in $B_1(p)$ and 
		\begin{equation}\label{eq:laplabp2}
			\Delta b_p^2 = 2n|\nabla b_p|^2\, ,\quad \text{on $B_{1}(p)$}\, .
		\end{equation}
	\end{itemize}
\end{corollary}

\subsection{Green-type distance on conical balls}
In this section, we study the fine properties of the local Green-type distance on balls that are $\delta$-Gromov--Hausdorff close to a ball centered at the tip of a cone. These estimates will play a central role in the proof of the \textit{annular slicing theorem} later in Section \ref{sec:slicing}.

\medskip

Given $0<s<r$, we denote by
\begin{equation}
	A_{r,s}(p) := B_r(p)\setminus \overline{B}_s(p)
\end{equation}
the {\it open annulus} centered at $p\in X$.

\begin{theorem}\label{thm:prop_Green}
	For every $\eps>0$, if $\delta<\delta_0(\eps,n,\theta)$ the following statement holds. Let $(X,\dist,\haus^n)$ be an $\RCD(-\delta(n-1),n)$ m.m.s.\ such that, for some $p\in X$
    and another $\RCD(n-2,n-1)$ space $(Z,\dist_Z,\haus^{n-1})$, it holds
	\begin{equation}
		\dist_{\rm{GH}}(B_{100}(p), B_{100}(o)) \le \delta \, ,
	\end{equation}
	where $o\in C(Z)$ is a tip point with density $\theta>0$. Then there exists a local Green-type distance $b_p:B_{20}(p) \to [0,\infty)$ satisfying the following properties:
	\begin{itemize}
	    \item[(i)] we have the first order bounds
		\begin{equation}\label{eq:gradboundsbp}
			\sup_{B_{20}(p)}|b_p - \dist_p| \le \eps \, , \quad
			\int_{B_{20}(p)} ||\nabla b_p|-1| \le \eps \, ;
		\end{equation}
		\item[(ii)] we have the higher order bounds
		\begin{equation}\label{eq:delta_grad_b}
			\int_{A_{11,6}(p)} |\nabla |\nabla b_p|| \le \eps \, ,
			\quad 
			\int_{A_{11,6}(p)} |\Delta |\nabla b_p|| \le \eps \, ;
		\end{equation}
		\item there exists $E\subset A_{11,6}(p)$ such that $\haus^n(E)\le \eps$ and 
		\begin{equation} 
			\left|\operatorname{Hess}(b_p)-\frac{1}{b_p}\left(I-\frac{\nabla b_p}{|\nabla b_p|}\otimes\frac{\nabla b_p}{|\nabla b_p|}\right)\right|^2\le \eps\, ,
			\quad \text{on $A_{11,6}(p)\setminus E$}\, .
		\end{equation}
	\end{itemize}
\end{theorem}
From now on we will often refer to a local Green-type distance simply as a \emph{Green distance}.

\begin{remark}
    Constants such as $\delta_0(\eps,n,\theta)$ in the previous statement in fact depend only on a given positive lower bound on $\theta$, rather than its specific value.
\end{remark}

\begin{remark}
	In our generality, the term $\Delta |\nabla b_p|$ is a (possibly singular) measure. So the estimate
	\begin{equation}
		\int_{A_{11,6}(p)} |\Delta |\nabla b_p|| \le \eps \, ,
	\end{equation}
	needs to be understood as a bound on the total variation of $\Delta |\nabla b_p|$.
\end{remark}

\begin{remark}[Hessian estimate]
	The Bochner inequality and the identity
	\begin{equation}\label{eq:Delta b}
		\Delta b_p = \frac{n-1}{b_p}|\nabla b_p|^2 \, , \quad \text{on $A_{20,1}(p)$}\, 
	\end{equation}
	yield the estimate
	\begin{equation}
		\int_{A_{11,6}(p)} | \Hess b_p|^2 + \int_{A_{11,6}(p)} \left\langle \nabla b_p, \nabla \left(\frac{(n-1)|\nabla b_p|^2}{b_p} \right) \right\rangle 
		\le \eps \, ,
	\end{equation}
	which ensures the $L^2$ Hessian bound
	\begin{equation}\label{eq:Hess b}
		\int_{A_{11,6}(p)} | \Hess b_p|^2 \le C(n,\theta)\, .
	\end{equation}
\end{remark}

\begin{remark}[Good Green distance]
	Let $(X,\dist, \haus^n)$ be an $\RCD(-\eps(n-1),n)$ space satisfying the assumptions of \autoref{thm:prop_Green}. We will call \emph{good Green distance} any function $b_p : B_{20}(p) \to [0,\infty)$ satisfying properties (i) and (ii).
\end{remark}

\begin{remark}
	As we will see in the proof of \autoref{thm:prop_Green}, every local Green function $G_p$ over $B_{40}(p)$ gives rise to a good Green distance, provided the following boundary condition is met:
	\begin{equation}\label{eq:boundary_b}
		\left|G_p - \frac{40^{2-n}}{n(n-2)\omega_n \theta}\right| \le \delta \, 
		\quad \text{on } \partial B_{40}(p) \, .
	\end{equation}
\end{remark}

\subsection{Proof of \autoref{thm:prop_Green}}
The proof of \autoref{thm:prop_Green} is divided into four parts. In Step 1 we shall verify \autoref{thm:prop_Green} (i) with a compactness argument. In Step 2 we obtain $L^1$ bounds for $|\Hess b_p|^2$ and $|\nabla|\nabla b_p||$. In Step 3 we obtain $L^1$ bounds for $|\Delta |\nabla b_p||$. In Step 4 we will complete the argument with the proof of \autoref{thm:prop_Green} (iii). 

\subsubsection{Step 1: Stability of the Green distance}
The first goal is to verify (i) for Green-type distances satisfying	\eqref{eq:boundary_b}, when $\delta < \delta_0(\eps,n,\theta)$. This is a soft argument based on the  classical stability for Green functions of the Laplacian under noncollapsed Gromov--Hausdorff convergence. 
We report below the simplified statement relevant to our purposes, referring the reader to \cite{AmbrosioHonda,GigliMondinoSavare15}, \cite[Section 2]{BrueDengSemola} for the relevant background and terminology.

\begin{proposition}\label{prop:Greenstability}\label{prop:Gstability}
	Let $(Y,\dist_Y,\haus^{n-1})$ be an $\RCD(n-2,n-1)$ m.m.s.\ for some $n\ge 3$. Let $(C(Y),\dist_{C(Y)},\haus^n)$ be the metric measure cone over $Y$, with tip $o$ and density $\theta>0$. If $(X_i,\dist_i,\haus^n,p_i)$ are $\RCD(-\delta_i(n-1),n)$ spaces such that
	\begin{equation}
		\dist_{{\rm GH}}\left(B_{100}(p_i),B_{100}(o)\right)\le \delta_i \to 0\, , \quad \text{as $i\to \infty$}\, ,
	\end{equation} 
	then any sequence of Green-type distances $b_{p_i}:B_{40}(p_i)\to [0,\infty)$ satisfying \eqref{eq:boundary_b} with $\delta = \delta_i$ converges uniformly and in $H^{1,2}$ to $\dist_o:B_{40}(o)\to [0,\infty)$.
\end{proposition}

The uniform convergence $b_{p_i} \to \dist_o$ immediately implies 
\begin{equation}
	\sup_{B_{20}(p)} |b_{p_i} - \dist_{p_i}| \to 0 \quad \text{as $i\to \infty$}\, .
\end{equation}
By relying on the strong $H^{1,2}$ convergence of $b_{p_i} - \dist_{p_i} \to 0$, we deduce
\begin{equation}
	\int_{B_{20}(p_i)}| |\nabla b_{p_i}| - 1| 
	\le
	\int_{B_{20}(p_i)}|\nabla (b_{p_i} - \dist_{p_i})| \to 0 \, ,
	\quad \text{as $i\to \infty$}\, .
\end{equation}
We refer the reader unfamiliar with the notion of Sobolev and uniform convergence on varying spaces to \cite{AmbrosioHonda,GigliMondinoSavare15}.

\subsubsection{Step 2: Hessian bounds}\label{subsec:Hess bound}
At this point it is clear that any Green-type distance $b_p : B_{40}(p) \to [0,\infty)$ enjoying the boundary condition \eqref{eq:boundary_b} satisfies \autoref{thm:prop_Green} (i). The second step of the proof consists in showing that any such a Green-type distance satisfies
\begin{equation}\label{eq:hess_b}
	\int_{A_{11,6}(p)} | \Hess b_p|^2 \le C(n,\theta)
	\, , \quad
	\int_{A_{11,6}} |\nabla |\nabla b_p|| \le \eps\, ,
\end{equation}
provided $\delta < \delta_0(\eps,n,\theta)$.

\medskip

Let $0\le\varphi\le1$ be a cut-off function, with $|\nabla\varphi|,|\Delta\varphi|\le C(n,\theta)$, such that $\varphi=1$ on $A_{11,6}(p)$ and $\varphi=0$ on the complement of $A_{12,5}(p)$. See \cite{CheegerColding96,AmbrosioMondinoSavare} and \cite[Lemma 3.1]{MondinoNaber19} for the detailed construction of good cut-off functions on spaces with Ricci bounded below.

\medskip

By Bochner's inequality, the estimates $4\le b_p\le 13$ (as $|b_p - d_p|\le 1$) on $A_{12,5}(p)$, and \eqref{eq:Delta b} we then have that $\Delta |\nabla b_p|^2$ is a measure satisfying the lower bound
\begin{align}\label{boch.sq}\begin{aligned}
		\Delta\frac{|\nabla b_p|^2}{2} &\ge|\operatorname{Hess}(b_p)|^2+\ang{\nabla\Delta b_p,\nabla b_p}-\delta(n-1)|\nabla b_p|^2 \\
		&=|\operatorname{Hess}(b_p)|^2-\frac{n-1}{b_p^2}|\nabla b_p|^4+2\frac{n-1}{b_p}\operatorname{Hess}(b_p)[\nabla b_p,\nabla b_p]-\delta(n-1)|\nabla b_p|^2 \\
		&\ge|\operatorname{Hess}(b_p)|^2-C(n,\theta)(1+|\operatorname{Hess}(b_p)|)\, .
\end{aligned}\end{align}
Integrating by parts against the cut-off function $\varphi$ and using Young's inequality, we deduce 
\begin{equation}\label{eq:hessbd}
	\begin{split}
		\int_{A_{11,6}(p)} | \Hess b_p|^2 
		&\le 
		\int\varphi|\operatorname{Hess} \, b_p|^2
		\\&
		\le C(n,\theta)+\mz\int|\Delta\varphi||\nabla b_p|^2
		\\&
		\le C(n,\theta)\, ,
	\end{split} 
\end{equation}
thanks to the Lipschitz bound on $b_p$ from \eqref{eq:Lipbp}.

\medskip

By the Bochner inequality and the chain rule \cite[Proposition 3.3]{GigliViolo} (see also \cite[Theorem 3.4]{GigliViolo}), $\Delta \sqrt{|\nabla b_p|^2 + \eps}$ is a measure on $A_{12,5}(p)$ for every $\eps>0$. Moreover
\begin{equation*}
	\begin{split}
		\Delta \sqrt{|\nabla b_p|^2 + \eps} &
		\ge 
		\frac{|\Hess b_p|^2 - |\nabla |\nabla b_p||^2 \frac{|\nabla b_p|^2}{|\nabla b_p|^2 + \eps}}{\sqrt{|\nabla b_p|^2 + \eps}} - \frac{n-1}{b_p^2}\frac{|\nabla b_p|^4}{\sqrt{|\nabla b_p|^2 + \eps}}
		\\ & \quad + 2 \frac{n-1}{b_p} |\nabla b_p| \Hess (b_p) \left[\frac{\nabla b_p}{|\nabla b_p|}, \frac{\nabla b_p}{\sqrt{|\nabla b_p|^2 + \eps}}\right] - \delta(n-1) \frac{|\nabla b_p|^2}{\sqrt{|\nabla b_p|^2 + \eps}} \, .
	\end{split}
\end{equation*}
By sending $\eps \to 0$, we have that $\Delta \sqrt{|\nabla b_p| + \eps} \to \Delta |\nabla b_p|$ in duality with test functions. Moreover, the following lower bound on $A_{12,5}(p)$ holds in the sense of distributions:
\begin{equation}\label{eq:Delta|nabla b|}
	\begin{split}
		\Delta |\nabla b_p| &\ge 
		\frac{|\Hess b_p|^2 - |\nabla |\nabla b_p||^2}{|\nabla b_p|} - \frac{n-1}{b_p^2}|\nabla b_p|^3
		\\& \quad + 2 \frac{n-1}{b_p} |\nabla b_p| \Hess (b_p) \left[\frac{\nabla b_p}{|\nabla b_p|}, \frac{\nabla b_p}{|\nabla b_p|} \right] - \delta(n-1) |\nabla b_p| 
		\\& \ge -C(n,\theta)(1+|\operatorname{Hess}\, b_p|)\, ,
	\end{split}
\end{equation}
where the previous expression is understood to be zero on $\{\nabla b_p=0\}$.

Therefore, $\Delta |\nabla b_p| + C(n,\theta)(1+|\Hess(b_p)|)$ is a nonnegative distribution in $A_{12,5}(p)$,
and by the previous bounds $\Delta |\nabla b_p|$ is a measure in $A_{11,6}(p)$.

\medskip

Taking $s:=\frac{1}{100}$, for any $q\in A_{11,6}(p)$ it holds $B_{20s}(q)\subset\{\varphi=1\}$, up to slightly modifying $\varphi$.
Using \autoref{dir} below with $f:=|\nabla b|$ and \eqref{eq:hessbd}, we can find a set $E_b(q)$ of volume $\haus^n(E_b(q))\le\tau$ such that
\begin{align}
	\nonumber s^2\int_{B_s(q)\setminus E_b(q)}|\nabla|\nabla b||^2
	&\le\tau^{-1}\||\nabla b|-1\|_{L^2}(\||\nabla b|-1\|_{L^2}+C(n,\theta)(1+\|\operatorname{Hess}(b)\|_{L^2})) \\
	&\le C(n,\theta)\tau^{-1}\||\nabla b|-1\|_{L^2}\, ,\label{eq:rhsbd}
\end{align}
where the $L^2$ norms are taken on $B_{2s}(q)$. By \autoref{thm:prop_Green} (i), the right-hand side in \eqref{eq:rhsbd} becomes arbitrarily small as $\delta\to 0$.

By the doubling property of the volume, taking a maximal collection of points $q_i\in A_{11,6}(p)$ with pairwise distance at least $s$,
the balls $B_s(q_i)$ cover $A_{11,6}(p)$ and this collection has cardinality bounded by $C(n,\theta)$ (since the balls $B_{s/2}(q_i)$ are disjoint).
Hence, setting $E_b:=\bigcup_i E_b(q_i)$, it holds
\begin{equation}\label{eqzz}
	\int_{A_{11,6}(p)\setminus E_b}|\nabla|\nabla b||^2\le\tau
\end{equation}
for $\delta$ small. Clearly, \eqref{eqzz} and \eqref{eq:hessbd} imply \eqref{eq:hess_b}.

\medskip

\begin{lemma}
	\label{dir}
	Let $(X,\dist, \haus^n)$ be an $\RCD(-(n-1),n)$ space for some $n\ge 2$ and fix $p \in X$. For any $f\in H^{1,2}_{\rm loc}(B_{20}(p))$ such that $\Delta f$ is a measure satisfying the lower bound
	\begin{equation}
		\Delta f \ge - g \, , \quad \text{on $B_{20}(p)$}
	\end{equation}
	for some $g \ge 0$, $g\in L^2(B_{20}(p))$, the following holds. For every $\tau\in (0,1)$ and $c\in \R$ there exists $E\subset B_{20}(p)$ such that $\haus^n(E)\le \tau$ and 
	\begin{equation}
		\int_{B_1(p)\setminus E} |\nabla f|^2\le C(n) \tau^{-1}\|f-c\|_{L^2(B_{20}(p))}(\|f-c\|_{L^2(B_{20}(p))}+\|g\|_{L^2(B_{20}(p))})\, .
	\end{equation}
\end{lemma}

\begin{proof}
	Let $\varphi_{10}$ be a good cut-off on the ball $B_{10}(p)$, i.e., $\varphi_{10}$ is a test function satisfying $\varphi = 1$ on $B_{10}(p)$ and $\varphi = 0$ on the complement of $B_{15}(p)$. Similarly, $\varphi_2$ will be a test function relative to $B_2(p)$. For every $\eps,\lambda>0$ we define
	\begin{equation}
		h_\eps := (P_\eps(\varphi_{10}f) - c + \lambda)^+ \in \Lip(X)\, ,
	\end{equation}
	where $P_\eps$ is the heat semigroup. It is easy to check that
	\begin{equation}
		\Delta h_\eps \ge - g_\eps \, 
		\quad \text{on $B_5(p)$}
	\end{equation}
	for some functions $g_\eps \ge 0$ converging to $g$ in $L^2(B_5(p))$ as $\eps \to 0$. We apply the standard Caccioppoli estimate to $h$ obtaining
	\begin{equation}
		\begin{split}
			\int \varphi_2^2 |\nabla h_\eps|^2
			&= - \int h_\eps \nabla \varphi_2^2 \cdot \nabla h_\eps - \int \varphi_2^2 h_\eps \Delta h_\eps
			\\&\le 
			C(n)\int_{B_3(p)} h_\eps^2 + \frac{1}{2} \int \varphi^2_2 |\nabla h_\eps|^2 + \int \varphi_2^2 h_\eps g_\eps \, \, .
		\end{split}
	\end{equation}
	Therefore
	\begin{equation}
		\int_{B_2(p)} |\nabla h_\eps|^2 \le C(n) \| h_\eps \|_{L^2(B_3(p))}(\|h_\eps\|_{L^2(B_3(p))} + \| g_\eps \|_{L^2(B_3(p))}) \, .
	\end{equation}
	We send $\eps \to 0$ and observe that $h_\eps \to h$ in $H^{1,2}$. Moreover, $|\nabla h| = |\nabla f|$ in the complement of $E:=\{f \le c - \lambda\}$, and $h \le |f-c| + \lambda$ in $B_3(p)$, so
	\begin{equation}
		\int_{B_2(p)\setminus E} |\nabla f|^2 \le C(n) \| |f-c| + \lambda \|_{L^2(B_3(p))}(\||f-c| + \lambda\|_{L^2(B_3(p))} + \| g \|_{L^2(B_3(p))}) \, .
	\end{equation}
	The conclusion follows by choosing $\lambda:= \tau^{-1/2}\| f - c \|_{L^2(B_3(p))}$.

\end{proof}

\subsubsection{Step 3:}
We claim that
\begin{equation}
	\int_{A_{11,6}(p)} |\Delta |\nabla b_p|| \le \eps \, ,
\end{equation}
provided $\delta \le \delta_0(\eps, n, \theta)$.

\medskip

Let $\varphi$ be a good cut-off as above.
In order to conclude, it suffices to show that for a set $E$ of arbitrarily small measure it holds
\begin{equation}
	\label{claim.compl}
	\int_{E^c}\varphi(\Delta|\nabla b_p|)^-\le \frac{\eps}{8}\, ,
\end{equation}
where $(\Delta|\nabla b_p|)^-$ denotes the negative part of the measure $\Delta |\nabla b_p|$.
Indeed, in view of \eqref{eq:Delta|nabla b|}, using also Cauchy--Schwarz and the $L^2$ bound on $\operatorname{Hess}(b_p)$, \eqref{claim.compl} would give
\begin{equation} 
	\int\varphi(\Delta|\nabla b_p|)^-
	\le
	\frac{\eps}{8} + C(n,\theta)\int_E\varphi(1+|\operatorname{Hess}(b_p)|)
	\le
	\frac{\eps}{8} + C(n,\theta)\haus^n(E)^{1/2}\le \frac{\eps}{4}\, ,
\end{equation}
provided that $\haus^n(E)$ is small enough.
The sought conclusion would follow since
\begin{align*}
	\int\varphi(\Delta|\nabla b_p|)^+=&\int\varphi(\Delta|\nabla b_p|)^-+\int\varphi\Delta|\nabla b_p|\\
	\le&\int\varphi(\Delta|\nabla b_p|)^-+\int|\Delta\varphi|\cdot ||\nabla b_p|-1|
	\le
	\frac{\eps}{2}\, ,
\end{align*}
thanks to the pointwise bound $|\Delta\varphi|\le C(n,\theta)$, and \autoref{thm:prop_Green} (i).

\medskip

In order to check \eqref{claim.compl},
we fix $\tau\in(0,\mz)$ and note that, since $|\nabla b_p| - 1$ and $|\nabla |\nabla b_p||$ are small in $L^1$, there exists a set $E$ of volume $\haus^n(E)\le\tau$ such that
\begin{equation}
	|\nabla|\nabla b_p||\le\tau\, ,\ \frac{1}{1+\tau}<|\nabla b_p|<1+\tau\quad\text{on }A_{11,6}(p)\setminus E\, .
\end{equation}
By \eqref{eq:Delta|nabla b|} we then have
\begin{align}
	\nonumber \Delta|\nabla b_p|
	&\ge (1-\tau)|\operatorname{Hess}(b_p)|^2-(1+\tau)\tau -(1+\tau)\frac{n-1}{b_p^2}|\nabla b_p|^4 \\
	&\quad-C(n)\left|\operatorname{Hess}(b_p)\left[\frac{\nabla b_p}{|\nabla b_p|},\frac{\nabla b_p}{|\nabla b_p|}\right]\right|-\delta(n-1)|\nabla b_p|\, ,\label{eq:firstuse}
\end{align}
on $A_{11,6}(p)\setminus E$. Since $\nabla|\nabla b_p|=\operatorname{Hess}(b_p)[\frac{\nabla b_p}{|\nabla b_p|},\cdot]$, using the bound $|\nabla|\nabla b_p||\le\tau$ and \eqref{eq:Delta b}, the inequality \eqref{eq:firstuse} rearranges to
\begin{equation} 
	\Delta|\nabla b_p|
	\ge(1-\tau)|\operatorname{Hess}(b_p)|^2-(1+\tau)\frac{(\Delta b_p)^2}{n-1}-C(n)(\tau+\delta)\, , \quad \text{on $A_{11,6}(p)\setminus E$}\, .
\end{equation}
Applying \autoref{trace} below with $A:=\operatorname{Hess}(b_p)$ and $v:=\frac{\nabla b_p}{|\nabla b_p|}$, and recalling that the Laplacian is the trace of the Hessian on noncollapsed $\RCD$ spaces (see for instance \cite{BrenaGigliHondaZhu}), we get
\begin{align*}
	\Delta|\nabla b_p|
	&\ge (1-\tau)\left[(1-\tau)\frac{(\Delta b_p)^2}{n-1}-\frac{1}{\tau}|\nabla|\nabla b_p||^2\right]-(1+\tau)\frac{(\Delta b_p)^2}{n-1}-C(n,\theta)(\tau+\delta) \\
	&\ge -3\tau (\Delta b_p)^2-C(n,\theta)(\tau+\delta)\, ,\quad \text{on $A_{11,6}(p)\setminus E$}\, .
\end{align*}
Recalling that $\|\operatorname{Hess}(b_p)\|_{L^2}\le C(n,\theta)$ on $A_{11,6}(p)$ (see \eqref{eq:hess_b}) and choosing $\delta\le \delta_0(\eps, \tau, n , \theta)$ we conclude the proof.

\subsubsection{Step 4: Proof of (iii)}

At this point we know that $\int_{A_{11,6}}|\Delta|\nabla b_p||$ is arbitrarily small
(up to decreasing $\delta$). Hence, outside of a set $E$ with $\haus^n(E)\le \tau$, it holds
\begin{equation}\label{eq:zz}
	(1-\tau)|\operatorname{Hess}(b_p)|^2-(1+\tau)\frac{(\Delta b_p)^2}{n-1}-C(n)\tau\le|\Delta|\nabla b_p||\le\tau\, ,
\end{equation}
if $\delta \le \delta_0(\tau, n ,\theta)$, where the inequality \eqref{eq:zz} is understood in the sense of measures. 

Applying \autoref{lemma:linalgqua} below with $A:=\operatorname{Hess}(b_p)$ and $v:=\frac{\nabla b_p}{|\nabla b_p|}$, we obtain
\begin{align*}
	&\left|\operatorname{Hess}(b_p)-\frac{\Delta b_p}{n-1}\left(I-\frac{\nabla b_p}{|\nabla b_p|}\otimes\frac{\nabla b_p}{|\nabla b_p|}\right)\right|^2\\
	&\le|\operatorname{Hess}(b_p)|^2-(1-\tau)\frac{(\Delta b_p)^2}{n-1}+\frac{1}{\tau}|\nabla|\nabla b_p||^2 \\
	&\le|\Delta|\nabla b_p||+C(n)\tau(1+|\operatorname{Hess}(b_p)|^2)+\frac{1}{\tau}|\nabla|\nabla b_p||^2 \\
	&\le C(n)\tau(1+|\operatorname{Hess}(b_p)|^2)\, ,
\end{align*}
on $A_{11,6}(p)\setminus E$.
Thanks to the $L^2$ bound on $\Hess(b_p)$ (see \eqref{eq:Hess b}), up to slightly enlarging $E$ we can assume that $|\Hess (b_p)|\le C(n,\theta)\tau^{-1/4}$ in $A_{11,6}(p)\setminus E$. 

Recalling that $\Delta b_p =\frac{n-1}{b_p}|\nabla b_p|^2$, on $A_{11,6}(p)\setminus E$ we then have
\begin{equation} 
	\left|\operatorname{Hess}(b_p)-\frac{|\nabla b_p|^2}{b_p}\left(I-\frac{\nabla b_p}{|\nabla b_p|}\otimes\frac{\nabla b_p}{|\nabla b_p|}\right)\right|^2\le C(n)\sqrt{\tau}\, . 
\end{equation}
To conclude the proof we observe that, up to slightly enlarging $E$, we can assume $1-\tau \le |\nabla b_p| \le 1+\tau$ by the second inequality in \eqref{eq:gradboundsbp}.

\subsubsection{Elementary lemmas}

In this section we present two elementary lemmas from linear algebra that were useful in the proof of \autoref{thm:prop_Green}.

\begin{lemma}\label{trace}
	For any symmetric $n\times n$ matrix $A$, $\tau\in(0,1)$ and $v\in\R^n$ with $|v|=1$, it holds
	\begin{equation} 
		|A|^2\ge(1-\tau)\frac{\operatorname{tr}(A)^2}{n-1}-\frac{1}{\tau}|Av|^2\, .
	\end{equation}
\end{lemma}

\begin{proof}
	Let $\lambda_1,\dots,\lambda_n$ be the eigenvalues of $A$, counted with multiplicity. Up to reordering, we can assume that $|\lambda_1|\le\dots\le|\lambda_n|$, so that $|Av|^2\ge\lambda_1^2$. Since $\operatorname{tr}(A)=\sum_i\lambda_i$, it holds
	\begin{equation}
		\begin{split}
			\operatorname{tr}(A)^2
			&\le
			\left(1+\frac{1}{\tau}\right)\lambda_1^2+(1+\tau)\left(\sum_{i=2}^{n}\lambda_i\right)^2
			\\&\le
			\left(1+\frac{1}{\tau}\right)|Av|^2+(1+\tau)(n-1)\sum_{i=2}^{n}\lambda_i^2\, ,
		\end{split}   		
	\end{equation}
	where we used Young's inequality and Cauchy--Schwarz. Recalling that $|A|^2=\sum_{i=1}^n\lambda_i^2$, we obtain
	\begin{equation}
		(1+\tau)(n-1)|A|^2\ge\operatorname{tr}(A)^2-\frac{1+\tau}{\tau}|Av|^2\, ,
	\end{equation}
	which gives the claim.
\end{proof}

Looking at the model case where $Av=0$, the argument in \autoref{trace} gives in fact 
\begin{equation}\label{eq:linal0}
|A|^2-\frac{\operatorname{tr}(A)^2}{n-1}\ge 0\, ,
\end{equation}
and equality occurs precisely when $A$ is a multiple of $I-v\otimes v$, i.e., when 
\begin{equation}
A=\frac{\operatorname{tr}(A)}{n-1}(I-v\otimes v)\, .
\end{equation}
Thus, the nonnegative quantity in \eqref{eq:linal0} (and $|Av|^2$) should bound the distance from this rigid case. This is exactly the content of the next lemma
(which, in fact, implies the previous one).

\begin{lemma}\label{lemma:linalgqua}
	For any symmetric $n\times n$ matrix $A$, $\tau\in(0,1)$ and $v\in\R^n$ with $|v|=1$, it holds
	\begin{equation}
		\label{rigid.mat}
		\left|A-\frac{\operatorname{tr}(A)}{n-1}(I-v\otimes v)\right|^2\le|A|^2-(1-\tau)\frac{\operatorname{tr}(A)^2}{n-1}+\frac{1}{\tau}|Av|^2\, .
	\end{equation}
\end{lemma}

\begin{proof}
	We compute
	\begin{equation}
		\left|A-\frac{\operatorname{tr}(A)}{n-1}(I-v\otimes v)\right|^2
		=|A|^2-\frac{\operatorname{tr}(A)^2}{n-1}+2\frac{\operatorname{tr}(A)\ang{v,Av}}{n-1}\, ,
	\end{equation}
	and apply Young's inequality.
\end{proof}

\section{Slicing Theorem}\label{sec:slicing}

The goal of this section is to establish an annular version of the slicing theorem by Cheeger and Naber \cite{CheegerNaber15} for smooth manifolds with Ricci curvature bounded below: see \autoref{trans} for the precise statement. In the special case of dimension $n=3$, we extend our result to the broader class of noncollapsed $\RCD$ spaces, in \autoref{trans n=3} below. This extension will be important later on to study the topology of $3$-dimensional $\RCD$ spaces. The overall strategy of the proof will be similar to \cite{CheegerNaber15}, although in the present setting there will be some additional technical challenges.

It seems likely that also \autoref{trans} holds for $\RCD$ spaces. However, generalizing the present argument would require some nontrivial technical work, in particular concerning the tensor calculus for $k$-forms for $k\ge 2$ that is heavily exploited in the proof of \autoref{step1}. For this reason, we leave the generalization to the $\RCD$ setting to the future investigation.

\medskip

We consider a smooth and complete manifold $(M^n,g)$ with $n\ge 3$ and $\Ric_g \ge -\delta (n-1)$. Let $p\in M$ be a  reference point. Our setup will be that the ball $B_{100}(p)$ is $\delta$-GH close to the ball in an $(n-3)$-symmetric cone, i.e.,
\begin{equation}\label{eq:close_cone}
	\dist_{{\rm GH}}\left(B_{100}(p),B_{100}(0^{n-3},o)\right)<\delta \, ,
	\quad
	(0^{n-3},o) \in \R^{n-3}\times C(Z^2) \, ,
\end{equation}
where $(Z^2,\dist_Z)$ is a two-dimensional Alexandrov space with curvature $\ge 1$. We denote by $\theta$ the density of the cone $\R^{n-3} \times C(Z)$ at $o$.

\medskip

We let $b:= b_p$ be a good Green distance with center $p$, as in \autoref{thm:prop_Green}.
Also, given a harmonic $\delta$-splitting map $v:B_{100}(p)\to\R^{n-3}$ (see \autoref{spl.def} and \autoref{thm:splittings}) such that $v(p)=0$, on the set $\{b>|v|\}$ we define the function
\begin{equation}\label{eq:u}
	u:=\sqrt{b^2-|v|^2} \, .
\end{equation}

\begin{remark}
	In the model case $M=\R^{n-3}\times C(Z^2)$ where the second factor is a three-dimensional cone
	and $p=(0, o)$ (with $o$ the tip of the cone),
	the map $v$ will simply be the projection onto the first factor, and for a point $q=(v,y)$ (with $y\in C(Z^2)$)
	we will have $b(q)=\dist_p(q)=\sqrt{|v|^2+\dist_{o}(y)^2}$. Hence, in this case $u$ coincides with the distance function from the set of vertices $\R^{n-3}\times \{o\}\subset \setR^{n-3}\times C(Z^2)$.
\end{remark}

\medskip

We define annular regions
\begin{equation}
	A:= [B_{9}(p)\setminus \overline{B}_{8}(p)]\cap\{|v|< 1\}\, ,
	\quad 
	A':=[B_{10}(p)\setminus \overline{B}_{7}(p)]\cap\{|v|<2\}\, .
\end{equation}
Notice that, for $\delta$ sufficiently small, $u$ will be defined on $A$ and $A'$, thanks to \autoref{thm:prop_Green}.
Moreover, since $|v|$ is $2$-Lipschitz (for $\delta$ small enough), we have $B_s(q)\subseteq A'$ whenever $q\in A$ and $s<\frac{1}{2}$.

\begin{theorem}[Slicing Theorem]\label{trans}
	For any $\eps>0$ there exists $\delta(\eps,n,\theta)>0$ with the following property.
	Let $(M^n,g)$ be a smooth complete manifold $(M^n,g)$ with $n\ge 3$, $\Ric_g \ge -\delta (n-1)$ such that \eqref{eq:close_cone} holds for some $p\in M$. Let $v:B_{100}(p)\to\R^{n-3}$ be a harmonic $\delta$-splitting map with $v(p)=0\in \R^{n-3}$, and $u$ be defined as in \eqref{eq:u}. Then there exists a Borel set $\mathcal{B}\, \subset B_1(0^{n-3})\times [0,10]$
	such that
	\begin{itemize}
		\item[(i)] 	$\Leb^{n-2}(\mathcal{B})\le\eps $;

		\item[(ii)] for every $(x,y) \in  B_1(0^{n-3}) \times [0,10]$ such that 
		\begin{equation}
			8 \le \sqrt{|x|^2 + y^2}\le 9 \, , \quad
			(x,y)\notin \mathcal{B} \, ,
		\end{equation}
		the level set $\{(v,u)=(x,y)\}$ is not empty;

		\item[(iii)] for every $s\in(0,c(\eps,n,\theta))$ and every $q\in A$ with $(v,u)(q)\notin\mathcal{B}$,
		there exists a lower triangular $(n-2)\times(n-2)$ matrix $L_{q,s}$ with positive diagonal entries such that
		\begin{equation}
			L_{q,s}\circ (v,u):B_s(q)\to\R^{n-2} 
		\end{equation}
		is an $\eps$-splitting map.
	\end{itemize}
\end{theorem}

In a nutshell, \autoref{trans} tells us that there are many good points $(x,y)\in B_1(0^{n-2}) \times [0,10]$ (with a quantitative volume estimate) in the image of $(v,u): A \to \R^{n-2}$ such that, for each point $q$ in the level set $\{(v,u)=(x,y)\}$, the map $(v,u)$ becomes $\eps$-splitting at every scale $s\le c(\eps,n,\theta)$, up to composing with a suitable transformation matrix.

In view of \autoref{thm:splittings}, if $(u,v)(q) = (x,y)$ is a good level set and $s<c(\eps,n,\theta)$, the ball $B_s(q)$ is $\eps'$-GH close to $B_s(0,w)\subset\R^{n-2}\times W$, where $(W,\dist_W)$ is a two-dimensional Alexandrov space with curvature $\kappa \ge 0$, and $\eps \le \eps_0(n, \eps')$.

\begin{remark}[Sharp gradient bound]
	Let $q\in A$ be as in \autoref{trans} (iii). According to our definition of $\eps$-splitting map, $L_{q,s}\circ (v,u)$ is a $C(n)$-Lipschitz function. However, a variant of the observation in \autoref{rmk:sharp_gradient}, based on the fact that
	\begin{equation}
		F(v,u):=\left(v,(u^2+|v|^2)^{(2-n)/2}\right)\, 
	\end{equation}
	is harmonic (despite $(v,u)$ is not), ensures the sharp gradient bound:
	\begin{equation}
		\sup_{B_{s/2}(q)}
		| \nabla L_{q,s}\circ (v,u) | \le 1 + C(n)\eps^{1/2}\, .
	\end{equation}
\end{remark}

As anticipated, in dimension $n=3$ it will be important to have a version of the annular transformation theorem valid in the broader class of noncollapsed $\RCD(K,3)$ spaces $(X,\dist,\haus^3)$. Below we report the relevant statement.

\begin{theorem}[Slicing Theorem, $n=3$]\label{trans n=3}
	For any $\eps>0$, there exists $\delta(\eps,\theta)>0$ with the following property.
	If $(X,\dist,\haus^3)$ is an $\RCD(-\delta,3)$ space and 
	\begin{equation}
		\dist_{{\rm GH}}(B_{100}(p), B_{100}(o)) \le \delta \, , 
	\end{equation}
	where $p\in X$ and $o\in C(Z)$ is a tip of an $\RCD(0,3)$ cone with density $\theta>0$, then the following holds. For every good Green distance $b_p : B_{40}(p) \to (0,\infty)$ there exists a Borel set $\mathcal{B} \subset [8,9]$
	such that
	\begin{itemize}
		\item[(i)] $\label{meas} \Leb^{1}(\mathcal{B})\le\eps$;

		\item[(ii)] for every $y\in [8,9]\setminus \mathcal{B}$,
		the level set $\{b_p = y\}$ is not empty;

		\item[(iii)] for every $s\in(0,c(\eps,\theta))$ and $q\in b_p^{-1}([8,9]\setminus \mathcal{B})$, the normalized Green distance 
		\begin{equation}
			\hat b_p := \frac{b_p}{\fint_{B_s(q)}|\nabla b_p|} : B_s(q) \to \R\, 
		\end{equation}
		is an $\eps$-splitting map.
	\end{itemize}
\end{theorem}

In other words, when centered at all points on most level sets, the Green distance $b_p$ induces a geometric splitting. Therefore, if $\{b_p = y\}$ is a good level set, for every point $q\in \{b_p = y\}$ and sufficiently small scale $s<c(\eps,n,\theta)$, it holds
\begin{equation}
	\dist_{{\rm GH}}(B_s(q), B_s(0,p_0)) \le \eps' \, , \quad p_0 \in W
\end{equation}
where $(W^2,\dist_W)$ is an Alexandrov space with curvature $\kappa \ge 0$, provided $\eps \le \eps_0(\eps')$.

\subsection{Outline of the proof}
\label{sec:outline trans}
The proof of \autoref{trans} is based on two distinct steps, namely \autoref{step1} and \autoref{step2} below. After briefly discussing the two steps, in this section we show how to complete the proof using them.

The proof of \autoref{trans n=3} follows an analogous strategy, albeit with the necessity to formulate the key steps in the broader context of $\RCD$ spaces. Specifically, \autoref{rmk:step1 n=3} serves as a substitute for \autoref{step1}, while \autoref{step2 n=3} takes the place of \autoref{step2}.

\medskip

Following \cite{CheegerNaber15} we introduce
\begin{equation}
	\begin{split}
		&\omega^\ell:= dv^{1}\wedge\dots\wedge dv^\ell \, ,
		\quad \text{for $\ell=1,\dots,n-3$}\, ,
		\\
		&\omega^{n-2}:= dv^{1}\wedge\dots\wedge dv^{n-3} \wedge du\, .
	\end{split}
\end{equation}
When $n=3$, we have  $\omega^1=du=db$. 

Notice that $\omega^\ell$ is well-defined and smooth in $A'$ and, as detailed later on, the Bochner formula gives
\begin{equation}\label{eq:|omega| Bochner}
	\begin{split}
		\Delta\frac{|\omega^\ell|^2}{2}
		& =|\nabla\omega^\ell|^2+\ang{\Delta\omega^\ell,\omega^\ell}
		\\&
		\ge|\nabla\omega^\ell|^2+\ang{dv^1\wedge\dots\wedge dv^{\ell-1}\wedge d\Delta u,\omega^\ell}-C(n)\delta|\omega^\ell|^2\, ,
	\end{split}
\end{equation}
where the middle term appears only when $\ell=n-2$, and
where we used the fact that $\Delta v=0$ and $\operatorname{Ric}_g\ge-\delta$. As in 
\cite[Lemma 3.7]{CheegerNaber15}, we will need to consider the Laplacian of $|\omega_\ell|$, which is not smooth. The following lemma corresponds to \cite[Lemma 3.7]{CheegerNaber15}.

\begin{lemma}
	The distributional Laplacians $\Delta|\omega^\ell|$ are measures with nonnegative singular part on $A'$. Moreover, they satisfy
	\begin{equation}
		\label{delta.uno}
		\Delta|\omega^\ell|\ge\frac{\ang{\Delta\omega^\ell,\omega^\ell}+|\nabla\omega^\ell|^2-|\nabla|\omega^\ell||^2}{|\omega^\ell|}\, ,
	\end{equation}
	where the right-hand side should be interpreted as zero on the closed set $\{\omega^\ell=0\}$ and $\Delta \omega^\ell$ is the connection Laplacian of $\omega^\ell$.	
\end{lemma}

\begin{remark}
	When $n=3$, the previous lemma amounts to the statement that $\Delta |\nabla b|$ is a measure in the annulus $A_{9,8}(p)$ satisfying the lower bound
	\begin{equation}\label{eq:Delta|nabla b|2}
		\begin{split}
			\Delta |\nabla b| &\ge 
			\frac{|\Hess b|^2 - |\nabla |\nabla b||^2}{|\nabla b|} - \frac{n-1}{b^2}|\nabla b|^3
			\\& \quad \quad+ 2 \frac{n-1}{b} |\nabla b| \Hess (b) \left[\frac{\nabla b}{|\nabla b|}, \frac{\nabla b}{|\nabla b|} \right] - \delta |\nabla b| \, .
		\end{split}
	\end{equation}
	The latter has been verified in the previous section in the generality of $\RCD$ spaces.
\end{remark}

\autoref{trans} will be accomplished in two steps, corresponding to the following two propositions.

\begin{proposition}\label{step1}
	Given $\eta>0$, there exists $\delta(\eta,n,\theta)>0$ such that under the same assumptions of \autoref{trans},
	it holds
	\begin{equation}
		\int_{A'}|\Delta|\omega^\ell||\le\eta\, ,\quad\int_{A'}||\omega^\ell|-1|\le\eta\, ,
	\end{equation}
	for all $\ell = 1,\dots,n-2$, where the first integral denotes the total variation of the measure $\Delta|\omega^\ell|$ on $A'$.
\end{proposition}

\begin{remark}\label{rmk:step1 n=3}
	When $n=3$, \autoref{step1} amounts to say that if $\delta \le \delta_0(\eta,\theta)$, then
	\begin{equation}
		\int_{A_{10,7}(p)} |\Delta |\nabla b|| \le \eta \, ,
		\quad
		\int_{A_{10,7}(p)} ||\nabla b|-1| \le \eta \, .
	\end{equation}
	This statement holds true in the class of noncollapsed $\RCD(-\delta(n-1),n)$ spaces as proven in \autoref{thm:prop_Green}.
\end{remark}

\begin{definition}[Singular scale]
	Given $\eta>0$, for any $q\in A$ we define the \emph{singular scale} $0\le s_q\le\mz$ to be the smallest number
	such that for all $s_q\le s<\mz$ it holds
	\begin{equation}
		s^2\int_{B_s(q)}|\Delta|\omega^\ell||\le\eta\int_{B_s(q)}|\omega^\ell|\, , 
	\end{equation}
	for all $\ell\in\{1,\dots,n-2\}$ (hence, $s_q=\mz$ if this fails already for $s=\mz$).
\end{definition}

\begin{proposition}\label{step2}
	Given $\eps>0$, if $\eta \le \eta_0(\eps,n,\theta)$ the following holds.
	Let $(M^n,g)$ be a complete smooth manifold with $n\ge 3$, $\Ric_g \ge -\eta (n-1)$ such that \eqref{eq:close_cone} holds for some $p\in X$ and $\delta=\mu$, where  $v:B_{100}(p)\to\R^{n-3}$ is a harmonic $\eta$-splitting map, and $u$ is defined as in \eqref{eq:u}.
	Then for every $q\in A$ and $s_q\le s\le c(\eps,n,\theta)$ there exists a lower triangular $(n-2)\times(n-2)$ matrix $L_{q,s}$ with positive diagonal entries such that
	\begin{equation}
		L_{q,s}\circ (v,u):B_s(q)\to\R^{n-2} 
	\end{equation}
	is an $\eps$-splitting map.

\end{proposition}

When $n=3$, as in the statement of \autoref{trans n=3}, \autoref{step2} involves only the Green distance $b_p$. The splitting result in terms of the singular scale can be expressed in the broader context of noncollapsed $\RCD$ spaces of any dimension, as follows.

\begin{proposition}\label{step2 n=3}
	For any $\eps>0$, if $\eta \le \eta_0(\eps,n,\theta)$ the following holds. Let
	$(X,\dist,\haus^n)$ be an $\RCD(-\eta(n-1),n)$ space such that
	\begin{equation}
		\dist_{{\rm GH}}(B_{100}(p), B_{100}(o)) \le \eta \, , 
	\end{equation}
	where $p\in X$ and $o\in C(Z)$ is a tip of an $\RCD(0,n)$ cone with density $\theta>0$. Let $b_p : B_{40}(p) \to (0,\infty)$ be a good Green distance. Let $q\in A_{8,9}(p)$ such that  
	\begin{equation}
		s^2 \int_{B_s(q)} |\Delta |\nabla b|| \le \eta \int_{B_s(q)} |\nabla b| \, 
		\quad \text{for every $s_q\le s <\mz$}\, .
	\end{equation}
	Then for every $s_q\le s \le c(\eps,n,\theta)$
	the normalized Green distance 
	\begin{equation}
		\hat b_p := \frac{b_p}{\fint_{B_s(q)}|\nabla b_p|} : B_s(q) \to \R\, 
	\end{equation}
	is an $\eps$-splitting map.
\end{proposition}

\begin{proof}
	[Proof of \autoref{trans} and \autoref{trans n=3}]

 As in \cite{CheegerNaber15}, \autoref{trans} follows from \autoref{step1} and \autoref{step2}.
Similarly, \autoref{trans n=3} follows from \autoref{rmk:step1 n=3} and \autoref{step2 n=3} when $n=3$. 

	For the sake of illustration, we outline only the argument for \autoref{trans}, pointing out the main changes needed for the proof of  \autoref{trans n=3}.

	\medskip

	Let $\eta$ be given by \autoref{step2} (depending on $\eps$), and let $\delta$ be given by \autoref{step1}, applied with $\eta$ replaced by a smaller threshold $\eta'<\eta$ to be chosen momentarily.
	Since
	\begin{equation}
		\int_{A'}|\Delta|\omega^\ell||\le\eta',\quad\int_{A'}||\omega^\ell|-1|\le\eta'\, ,
	\end{equation}
	we necessarily have $s_q\le\frac{c}{10}$ for all $q\in A$, provided that $\eta'$ is small enough (here $c$ is the constant appearing in \autoref{step2}).
	Indeed, for any $s\in[\frac{c}{10},c]$, we have
	\begin{equation}
		\int_{B_s(q)}|\omega^\ell|\ge\haus^n(B_s(q))-\int_{B_s(q)}||\omega^\ell|-1|
		\ge \haus^n(B_s(q))-\eta'\ge\frac{1}{C(\eps,n,\theta)}\, , 
	\end{equation}
	thanks to $s\ge\frac{c}{10}$ and the noncollapsing condition (provided $\eta'$ is taken small enough).
	Hence
	\begin{equation}
		\int_{B_s(q)}|\Delta|\omega^\ell||\le\eta'\le C(\eps,n,\theta)\eta'\int_{B_s(q)}|\omega^\ell|\le\eta\int_{B_s(q)}|\omega^\ell|\, , 
	\end{equation}
	provided that $\eta'$ is so small that $C(\eps,n,\theta)\eta'\le\eta$,
	as desired. 
	
	\medskip

	Setting $w:=(v,u)$ and
	\begin{equation}
		\tilde{\mathcal{B}}:=\bigcup\{B_{s_q}(q)\mid q\in A\text{ such that }s_q>0\}\, , 
	\end{equation}
	we can apply Vitali's covering lemma to obtain that
	$\tilde{\mathcal{B}}\subseteq\bigcup_j B_{5s_{q_j}}(q_j)$ for a suitable collection of disjoint balls $B_{5s_{q_j}}(q_j)$. Then we can bound
	\begin{equation}
		\Leb^{n-2}(w(\tilde{\mathcal{B}}))\le \sum_j\Leb^{n-2}(w(B_{5s_{q_j}}(q_j)))\le C(n)\sum_j s_{q_j}^{-2}\int_{B_{s_{q_j}}}|\omega^{n-2}|\, , 
	\end{equation}
	where the last inequality follows easily from the fact that $w$ is $\eps$-splitting for the radii $s=s_{q_j},5s_{q_j}\in[s_{q_j},c)$, up to composing with a linear transformation, by \autoref{step2} (see \cite[Lemma 4.1 and Lemma 4.10]{CheegerNaber15} for the details).
	Moreover, by definition of $s_q$, since $s_{q_j}>0$ we can bound
	\begin{equation}
		\int_{B_{s_{q_j}}(q_j)}|\omega^{n-2}|\le\eta^{-1}s_{q_j}^2\int_{B_{s_{q_j}}(q_j)}|\Delta|\omega^{n-2}||\, .
	\end{equation}
	Hence
	\begin{equation}\label{eq:intbound3}
		\Leb^{n-2}(w(\tilde{\mathcal{B}}))
		\le C(n)\eta^{-1}\int_{A'}|\Delta|\omega^{n-2}||
		\le C(n)\eta^{-1}\eta' \, ,
	\end{equation}
	since the balls are disjoint and included in $A'$. The right-hand side in \eqref{eq:intbound3} is less than $\eps$ once we choose $\eta'$ small enough.
 
	We define
	\begin{equation}
		\mathcal{B}:=w(\tilde{\mathcal{B}} \cap A) \cup (\{(x,y)\in B_1(0^{n-3})\times [0,10]\, : \, 8 \le \sqrt{|x|^2 + y^2} \le 9\}\setminus w(A))\, .	
	\end{equation}
	In other words, $\mathcal{B}$ consists of those points that either belong to the image of the bad set $\tilde{\mathcal{B}}$, or such that the preimage through $w$ does not intersect $A$.
	
	\medskip
	
	To complete the proof, it is sufficient to check that 
	\begin{equation}\label{eq:annoying}
		\Leb^{n-2}(\{(x,y)\in B_1(0^{n-3})\times [0,10]\, : \, 8 \le \sqrt{|x|^2 + y^2} \le 9\}\setminus w(A))\le \eps \, ,
	\end{equation}
	if $\delta\le \delta_0(\eps,n,\theta)$. 
 The analogous estimate in \cite{CheegerNaber15} follows from \cite{CheegerColdingTian}. In the present setting we need to slightly modify the argument due to the fact that the last component of the map is not a harmonic splitting function.
 
 We begin with the case $n=3$, which is immediate. In that case, \eqref{eq:annoying} amounts to showing that 
	\begin{equation}
		\Leb^1([8,9]\setminus b(A_{9,8}(p))) \le \eps 
		\quad \text{provided $\delta \le \delta_0(\eps,n)$}\, .
	\end{equation}
	Let $q$ be a point such that $\dist(p,q)=9$. We consider a minimizing geodesic $\gamma:[0,9]\to X$ connecting $p$ and $q$, parametrized by arclength.
	From \autoref{thm:prop_Green}, we know that $b$ is $\eps$-close to $\dist_p$, giving $|b(\gamma(8)) - 8| \le \eps$ and $|b(\gamma(9)) - 9| \le \eps$. Hence, by continuity,
	\begin{equation}
		(8 + \eps, 9-\eps) \subseteq b \circ \gamma((8,9)) \subseteq b(A_{9,8}(p)) \, .
	\end{equation}
	
	\medskip
	Let us consider the general case $n>3$. 
	Recall that $B_{100}(p)$ is $\delta$-GH close to the corresponding ball $B_{100}(0^{n-3},o)$ in $\R^{n-3} \times C(Z^2)$. Since $Z^2$ is a noncollapsed $\RCD(1,2)$ space, there exist $z_0\in Z$ and $\rho = \rho(\eta,n,\theta)>0$ such that $B_{10\rho}(z_0)$ is $\eta \rho$-GH close to the Euclidean ball $B_{10\rho}(0^2)$, and there is a harmonic $\eta$-splitting map $f=(f_1,f_2):B_{\rho}(z_0)\to \R^2$ which is also an $\eta \rho$-GH isometry. We assume that $f(z_0) = 0$ and extend $f$ to the open domain
	\begin{equation}
		\Omega :=\{(x,r,z)\in \R^{n-3}\times C(Z)\, : \, |x|< 1\, , \   r\in (7,10)\, , \  z\in B_\rho(z_0) \} \, ,
	\end{equation}
	by setting $f(x,r,z) := f(z)$. Note that $f: \Omega \to \R^2$ is harmonic, and the domain $\Omega$ is included in a finite union of $C(n)\eta\rho$-Euclidean balls.
	
	\medskip

    Let $\eps' \le \eta$.
	If $\delta \le \delta_0(\eps',n,\theta)$, we can find an open domain $\tilde \Omega\subset M^n$ that is $\eps'$-GH close to $\Omega$, and a harmonic function $\tilde f = (\tilde f_1, \tilde f_2): \tilde \Omega \to \R^2$ such that $f$ and $\tilde f$ are $\eps'$-GH close in uniform norm (i.e., there is an $\eps'$-GH isometry $\psi:\Omega \to \tilde \Omega$ such that $\|f -\tilde f \circ \psi\|_{\infty} \le \eps'$). 
	
	We consider the mapping
	\begin{equation}
		\tilde \Phi:= (v,u,\tilde f) : 
		A' \cap \tilde \Omega \to \R^n \, .
	\end{equation}
	Observe that, if $\delta \le \delta_0(\eps',n,\theta)$ is small enough, $\tilde \Phi$ is $C(n)\eps'$-close in $W^{1,2}$ to $\Phi(x,r,z) := (x,r,f(z))$, which is defined on $\Omega \subset \R^{n-3}\times C(Z)$. By relying on the $W^{1,2}$-closeness of $\tilde \Phi$ and $\Phi$ we get
	\begin{equation}
		\begin{split}
			&\fint_{A' \cap \tilde \Omega} |  \langle \nabla \tilde \Phi^a , \nabla \tilde \Phi^b \rangle| \le \eta \, , \quad \text{for $a\neq b$,}
			\\
			&\fint_{A' \cap \tilde \Omega} ||\nabla v |^2 - 1| \le \eta \, , 
			\quad
			\fint_{A' \cap \tilde \Omega} ||\nabla u|^2 - 1| \le \eta \, , 
			\quad 
			\fint_{A' \cap \tilde \Omega} |\dist_p |\nabla \tilde f|-1| \le \eta\, .
		\end{split}		
	\end{equation}
	By a standard maximal function argument, we can find $E\subset A\cap \tilde \Omega$ such that 
 \begin{equation}\label{E.small.complement}
     \haus^n((A\cap \tilde \Omega)\setminus E) \le C(n,\theta)\eta^{1/2} \haus^n(A\cap \tilde \Omega) \le C(n,\theta) \eta^{1/2} \rho^2 \, ,
 \end{equation}
 and for every $q\in E$, $r\in(0,\mz)$ it holds
	\begin{equation}\label{eqz90}
		\begin{split}
			&\fint_{B_r(q)} |  \langle \nabla \tilde \Phi^a , \nabla \tilde \Phi^b \rangle| \le \eta^{1/2}
			\\
			&\fint_{B_r(q)}||\nabla v|^2-1| 
			+
			\fint_{B_r(q)}||\nabla u|^2-1|
			+
			\fint_{B_r(q)}|\dist_p|\nabla \tilde f| - 1| \le \eta^{1/2}\, ,
		\end{split}
	\end{equation}
    where we used the inclusion $B_r(q)\subset A'\cap \tilde \Omega$.
    Let $q\in E$ and $r>0$. An easy contradiction argument based on \eqref{eqz90} shows that if $q'\in B_r(q)\cap A \cap \tilde \Omega$ and $\dist(q,q') \ge r/2$, then $|\tilde \Phi(q) - \tilde \Phi(q')| \ge r/4$. This is enough to show that
    for every $q\in E$, it holds
    \begin{equation}\label{eqz91}
    	\tilde \Phi^{-1}(\tilde \Phi(q))\cap A \cap \tilde \Omega = \{q\} \, ,
    	\quad \text{provided $\delta \le \delta_0(\eta, \eps ,n)$}\, .
    \end{equation} 
    Indeed, if we assume by contradiction the existence of $q\in E$ and $q'\in A \cap \tilde \Omega$ such that $\tilde \Phi(q) = \tilde \Phi(q')$, we can pick $r:= 2\dist(q,q')$ and apply the observation above.

	\medskip
	
	We claim that $\tilde \Phi(E)$ covers the set
	\begin{equation*}
		\Lambda:=\lbrace(x,y,t)\in B_{1-\eta}(0^{n-3})\times \R \times \R^2\, : \,
		8 + \eta \le \sqrt{|x|^2 + y^2} \le 9-\eta\, , \  |t|\le \rho (1 - \eta)\rbrace
	\end{equation*}
    up to a set of Lebesgue measure $C(n,\theta)\eta^{1/2}\rho^2$, provided $\delta \le \delta_0(\eta, \eps, n)$. It is not hard to check that this claim implies in turn the sought \eqref{eq:annoying}.

    \medskip
    
    The first observation is that $\tilde \Phi(E)$ is included in
    $$\lbrace(x,y,t)\in B_{1+\eta}(0^{n-3})\times \R \times \R^2\, : \,
		8 - \eta \le \sqrt{|x|^2 + y^2} \le 9+\eta\, , \  |t|\le \rho (1 + \eta)\rbrace$$
    provided that $\delta \le \delta_0(\eps',n,\theta)$. So it is enough to check that $\Leb^{n}(\tilde \Phi(E)) \ge \Leb^n(\Lambda) - C(n,\theta)[\eta^{1/2}\rho^2+\eps']$.
    
    By the area formula and \eqref{E.small.complement}--\eqref{eqz91}, we obtain
    \begin{equation}
    	\Leb^n(\tilde \Phi(E))
    	=
    	\int_E |J\tilde \Phi|
    	\ge \int_{A\cap \tilde \Omega} |J\tilde \Phi| - C(n,\theta)\eta^{1/2}\rho^2\, .
    \end{equation}
    Moreover, if $\delta \le \delta_0(\eps',n,\theta)$, by $W^{1,2}$-closeness of $\Phi$ and $\tilde \Phi$ we have
    \begin{align*}
    	\int_{A\cap \tilde \Omega} |J\tilde \Phi|
    	&\ge 
    	\int_{\psi^{-1}(A)\cap\Omega} |J \Phi| - C(n)\eps'\\
    	&\ge \Leb^n(\Phi(\psi^{-1}(A)\cap\Omega)) - C(n)\eps'\\
        &\ge \Leb^n(\Phi(\Omega)) - C(n)\epsilon'\, ,
    \end{align*}
    where we applied the area formula on $\R^{n-3} \times C(Z^2)$ and assumed
    that $\Omega\setminus\psi^{-1}(A)$ has measure at most $\eps'$, which holds
    for $\delta \le \delta_0(\eps',n,\theta)$.
    Finally, we use that $f:B_\rho(z_0) \to \R^2$ is an $\eta \rho$-GH isometry to deduce that $\Phi(\Omega)\supseteq\Lambda$, giving the claim once we choose $\eta,\eps'$ so small that $C(n,\theta)\eta^{1/2}\rho^2+C(n)\eps'\le\eps$.
\end{proof}

\subsection{Proof of \autoref{step2 n=3}}

Considering \autoref{rmk:step1 n=3} and the discussion in Section \ref{sec:outline trans}, to complete the proof of \autoref{trans n=3} we only need to prove \autoref{step2 n=3}. The strategy in this part will be to argue as in the proof of \cite[Lemma 3.3]{CheegerNaber15}. To this aim, we begin with an elementary fact, similar to \autoref{trace} in spirit. It will be used as a replacement of {\it Kato's inequality}.

\begin{lemma}\label{trace.bis}
	For any symmetric $n\times n$ matrix $A$ and $v\in\R^n$ with $|v|=1$, it holds
	\begin{equation} 
		|A|^2-|Av|^2\ge\frac{1}{2n-1}|A|^2-\frac{2}{2n-1}\operatorname{tr}(A)^2\, . 
	\end{equation}
\end{lemma}

\begin{proof}
	With the same notation used in the proof of \autoref{trace}, it holds
	\begin{equation}
		|A|^2-|Av|^2\ge\sum_{i=1}^{n-1}\lambda_i^2 \,. 
	\end{equation}
	Setting $s:=\sum_{i=1}^{n-1}\lambda_i$, for any $\tau>0$ it holds 
	\begin{align*}
		\lambda_n^2=&(\operatorname{tr}(A)-s)^2\le\left(1+\frac{1}{\tau}\right)\operatorname{tr}(A)^2+(1+\tau)s^2\\
		\le&\left(1+\frac{1}{\tau}\right)\operatorname{tr}(A)^2+(1+\tau)(n-1)\sum_{i=1}^{n-1}\lambda_i^2\, .
	\end{align*}
	Adding $\sum_{i=1}^{n-1}\lambda_i^2$ to both sides, we get
	$$ |A|^2\le\left(1+\frac{1}{\tau}\right)\operatorname{tr}(A)^2+[1+(1+\tau)(n-1)](|A|^2-|Av|^2)\, , $$
	and the claim follows by taking $\tau:=1$.
\end{proof}

\medskip

Next we state a technical lemma stemming from the Moser iteration technique. Its proof is standard, as it follows from the very same argument as in the Euclidean case, and therefore we omit it.

\begin{lemma}\label{lemma:Mosersharper}
	Let $(X,\dist,\meas)$ be an $\RCD(-(N-1),N)$ m.m.s.\ and $f\in H^{1,2}_{\rm loc}$ nonnegative such that
	\begin{equation}
		\Delta f\ge -\delta f\, ,\quad \text{on $B_1(p)\subset X$}\, ,
	\end{equation}
	in duality with test functions, for some $\delta>0$. Then
	\begin{equation}
		\sup_{B_r(x)}f\le C(N)\fint_{B_{2r}(x)}f\, ,
	\end{equation}
	for any $B_{2r}(x)\subset B_1(p)$.
\end{lemma}

\medskip

We are now ready to prove \autoref{step2 n=3}.
Let $b$ be as in the assumption of \autoref{step2 n=3}. Applying \autoref{trace.bis} above in \eqref{eq:Delta|nabla b|}, choosing $A:=\operatorname{Hess}(b)$ and $v:=\frac{\nabla b}{|\nabla b|}$, on $A_{12,5}(p)$ we have
\begin{equation}\label{eq:bochproof3}
	\begin{split}
		\Delta|\nabla b|&\ge
		\frac{|\operatorname{Hess}(b)|^2}{(2n-1)|\nabla b|}-C(n)|\nabla b|(|\nabla b|^2+|\operatorname{Hess}(b)|+\delta)
		\\& \ge
		C(n)^{-1} \frac{|\Hess (b)|^2}{|\nabla b|} - C(n,\theta)|\nabla b|\, ,
	\end{split}
\end{equation}
since the trace of $\operatorname{Hess}(b)$ is $\Delta b =\frac{n-1}{b}|\nabla b|^2$ (see \eqref{eq:laplabp2}), as already mentioned. Above in the second inequality we used Young's inequality and the Lipschitz bound $|\nabla b| \le C(n,\theta)$ on $A_{12,5}(p)$ (see \eqref{eq:Lipbp}).

\medskip

Recall our assumption:
\begin{equation}
	s^2\int_{B_s(q)}|\Delta|\nabla b||\le\eta\int_{B_s(q)}|\nabla b|\, ,
\end{equation}
for all $s_q\le s<\frac{1}{2}$. Fix a radius $s\in[\frac{s_q}{2}, \frac{1}{4})$ and set
\begin{equation}
	\hat b :=\frac{b}{\fint_{B_s(q)}|\nabla b|}\, . 
\end{equation}
From \eqref{eq:bochproof3} and the fact that $2s\ge s_q$, we deduce
\begin{equation}\label{eq:intbound15}
	\begin{split}
		s^2\fint_{B_{2s}(q)} \frac{|\Hess (\hat b)|^2}{|\nabla \hat b|}
		&\le 
		C(n)s^2\fint_{B_{2s}(q)} |\Delta |\nabla \hat b|| + C(n, \theta)s^2 \fint_{B_{2s}(q)} |\nabla  \hat b|
		\\&
		\le C(n,\theta) (\eta + s^2) \fint_{B_{2s}(q)} |\nabla  \hat b| \, .
	\end{split}
\end{equation}
Together with \autoref{lemma:Mosersharper} applied to $f:=|\nabla \hat b|$, \eqref{eq:intbound15} implies
\begin{equation}
	\begin{split}
		s^2 \fint_{B_{2s}(q)} |\Hess \hat b|^2
		&\le 
		\sup_{B_{2s}(q)} |\nabla \hat b| \left( s^2\fint_{B_{2s}(q)} \frac{|\Hess (\hat b)|^2}{|\nabla \hat b|} \right)
		\\&\le C(n,\theta)(\eta + s^2) \left(\sup_{B_{2s}(q)} |\nabla \hat b|\right)^2
		\\& \le
		C(n,\theta) (\eta + s^2) \left(\fint_{B_{4s}(q)} |\nabla  \hat b|\right)^2\, .
	\end{split}
\end{equation}
Moreover, since $|\nabla|\nabla\hat b||\le|\operatorname{Hess}(\hat b)|$ and $\fint_{B_s(q)}|\nabla\hat b|=1$, by the Poincar\'e inequality
we have
\begin{equation} 
	\fint_{B_s(q)}||\nabla\hat b|-1|^2\le C(n)s^2\fint_{B_{2s}(q)}|\operatorname{Hess}(\hat b)|^2\le C(n,\theta)(\eta+s^2)\left(\fint_{B_{4s}(q)}|\nabla \hat b|\right)^2 \, .
\end{equation}
To conclude the proof it suffices to pick $\eta \le \eta_0(\eps,n, \theta)$ and $s\le c(\eps,n,\theta)$ so that $C(n,\theta)(\eta + c(\eps,n,\theta)^2)\le \eps^3$ and prove the following gradient bound:
\begin{equation}
	\label{lip.tilde.u} 
	\fint_{B_{4s}(q)}
	|\nabla b|
	\le 10
	\fint_{B_{s}(q)}
	|\nabla b| \, , \quad \text{for every $s\in (s_q, c(\eps,n,\theta))$}\, 
\end{equation}
provided $\delta \le \delta_0(\eps,n,\theta)$.

When $\frac{1}{10}c(\eps,n,\theta)\le s\le c(\eps,n,\theta)$, \eqref{lip.tilde.u} is satisfied provided $\delta \le \delta_0(\eps,n,\theta)$, by \autoref{thm:prop_Green}
(which also guarantees that $s_q\le\frac{1}{10}c(\eps,n,\theta)$, as seen above).
Let $\overline s$ be the smallest number in $[s_q,c(\eps,n,\theta)]$ such that \eqref{lip.tilde.u} holds, for all $s\in[\overline s,c(\eps,n,\theta)]$. Assume by contradiction that $\overline s>s_q$.
Since \eqref{lip.tilde.u} holds for $4\overline s$, we know that
\begin{equation}
	\tilde b := \frac{b}{\fint_{B_{4\overline s}(q)}|\nabla b|} : B_{4\overline s}(q) \to \R \, 
\end{equation}
is an $\eps$-splitting map. Hence,
\begin{equation}
	\begin{split}
		\left| \fint_{B_{\overline s}(q)} |\nabla b| - \fint_{B_{4\overline s}(q)} |\nabla b| \right| 
		& \le
		\fint_{B_{\overline s}(q)} \left| |\nabla b| - \fint_{B_{4\overline s}(q)}|\nabla b| \right|
		\\& \le 
		C(n) \fint_{B_{4 \overline s}(q)} \left| |\nabla b| - \fint_{B_{4\overline s}(q)}|\nabla b| \right|
		\\& = C(n) \left(\fint_{B_{4\overline s}(q)} |\nabla b|\right) \fint_{B_{4\overline s}(q)} ||\nabla \tilde b|-1|
		\\&\le  \eps C(n) \fint_{B_{4 \overline s}(q)} |\nabla b| \, ,
	\end{split}
\end{equation}
which implies
\begin{equation}
	\fint_{B_{4\overline s}(q)} |\nabla b|
	\le (1 + C(n)\eps) \fint_{B_{\overline s}(q)} |\nabla b| \, ,
\end{equation}
a contradiction for $\eps \le \eps_0(n)$.

\begin{remark}
    The same argument used in this proof shows that we cannot have $\nabla b=0$ on a ball $B_s(q)$ with $s\ge s_q$.
\end{remark}

\section{Proofs of \autoref{step1} and \autoref{step2}}\label{sec:proofslicing}

As in \cite{CheegerNaber15}, the proof of \autoref{step2} in the general case requires an argument by contradiction, and we will just recall and sketch the main steps here, discussing more carefully only the parts which are different in our setting.

\subsection{Preliminary bounds}

Before showing \autoref{step1} and \autoref{step2}, we obtain some preliminary bounds on $w:=(v,u)$.
We will need the larger domains
\begin{equation}
	A'' 
	:= [B_{11}(p)\setminus \overline{B_6}(p)]\cap\{|v|<3\}\, ,\quad A''' := [B_{12}(p)\setminus \overline{B_5}(p)]\cap\{|v|< 4\}\, .
\end{equation}
Notice that, if $\delta\le \delta_0(n,\theta)$, then
$u$ is defined on $A'''$ and
\begin{equation}\label{prel2}
	u \ge 1 \, ,
	\quad 
	|\nabla u|\le C(n,\theta)
	\quad \text{on $A'''$}\, ,
\end{equation} 
by \autoref{thm:prop_Green}. 
Setting $w:=(v,u)$, an easy compactness argument based on \autoref{prop:Gstability} ensures that
\begin{equation}\label{eq:prel}
	\int_{A'''}|\ang{\nabla w^a,\nabla w^b}-\delta^{ab}|\le\eps\, ,
	\quad \text{for any $a,b=1,\dots,n-2$}\, ,
\end{equation}
provided $\delta\le \delta_0(\eps,n,\theta)$.

\begin{lemma}\label{abs.sobolev}
	Given $\tau>0$, if $\delta\le \delta_0(\tau,n,\theta)$, then 
	\begin{equation} 
		\int_{A''} |\nabla|\nabla u|| \le \tau \,,\quad
		\int_{A''}|\nabla|\nabla v||\le \tau \, ,\quad
		\int_{A''}|\nabla|\omega^\ell||\le\tau \, .
	\end{equation}
\end{lemma}

\begin{proof}
	First of all, note that we can construct a cut-off function $0\le\varphi\le1$, with $|\nabla\varphi|,|\Delta\varphi|\le C(n,\theta)$,
	such that $\varphi=1$ on $A''$ and $\varphi=0$ on the complement of $A'''$.\\
	Indeed, let $\psi:\R^2\to\R$ smooth such that $\psi(s,t)=1$ if $(s,t)\in[6-\frac14,11+\frac14]\times[0,3+\frac14]$,
	while $\psi(s,t)=0$ if $(s,t)\notin[5+\frac14,12-\frac14]\times[0,4-\frac14]$.
	For $\delta$ small enough, $\varphi:=\psi(b, |v|)$ will have the desired properties.

	\medskip
	
	The $L^1$ bound on $|\nabla |\nabla v||$ is obvious since $|\nabla|\nabla v^i||\le|\operatorname{Hess}(v^i)|$ is small in $L^2$. For $u$ we argue as follows. We compute that
	\begin{equation}
		u \nabla u = b\nabla b-v\nabla v
	\end{equation}
	(abbreviating $v\nabla v:=\sum_i v^i\nabla v^i$) and, since $\Delta v=0$,
	\begin{align}
		\begin{aligned}
			\label{lapl.u}
			\Delta u&=-\frac{|\nabla u|^2}{u}+\frac{b\Delta b+|\nabla b|^2-|\nabla v|^2}{u} \\
			&=\frac{n|\nabla b|^2-|\nabla u|^2-|\nabla v|^2}{u}\, .
		\end{aligned}
	\end{align}
	By \eqref{eq:Hess b} and the defining conditions of splitting maps,  $\int\varphi(|\operatorname{Hess}(b)|^2+|\operatorname{Hess}(v)|^2)\le C(n,\theta)$.
    The analogous bound for $u$ follows from the chain rule, as $u$ is a function of $b$ and $v$.

      In particular, we conclude that
      \begin{equation}\label{eqz11}
        \int\varphi|\operatorname{Hess}(u)|^2
        \le C(n,\theta)\, .
      \end{equation}
	Moreover, arguing as in Section \ref{subsec:Hess bound}, from the Bochner inequality again we deduce that $\Delta |\nabla u|$ is a measure satisfying the lower bound
	\begin{equation}
		\Delta |\nabla u| \ge -C(n,\theta)(1 + |\Hess u| + |\Hess v| + |\Hess v|)\, .
	\end{equation}
	The $L^1$ bound on $|\nabla |\nabla u||$ follows by \autoref{dir} with $f:=|\nabla u|$.

	\medskip

	Finally, for $\omega^\ell$ we employ \eqref{delta.uno}.
	From \eqref{lapl.u} we see that $|d\Delta u|\le C(n,\theta)(|\operatorname{Hess}(b)|+|\operatorname{Hess}(u)|+|\operatorname{Hess}(v)|)$,
	thanks to the Lipschitz bounds on $b$, $u$, and $v$. Thus, integrating by parts against $\varphi$, we get
	\begin{equation}
		\int\varphi|\nabla\omega^\ell|^2\le C(n,\theta)\, ,
	\end{equation}
	and we can apply \autoref{dir} with $f:=|\omega^\ell|$. Since $|\omega^\ell|$ is arbitrarily close to $1$ in $L^2$ by \eqref{eq:prel}, the conclusion follows.
\end{proof}

\subsection{Proof of \autoref{step1}}

The estimate on $\Delta|\omega^\ell|$ when $\ell < n-2$ corresponds to \cite[Theorem 1.6]{CheegerNaber15}, since $\omega^\ell = dv^1 \wedge \ldots \wedge dv^\ell$ is a wedge product of differentials of harmonic splitting maps. The main novelty in our proof concerns the bound on $\Delta | \omega^{n-2}|$, which requires the estimates on the Green distance $b$. We sketch also the proof in the case $\ell < n-2$, as the argument for $\ell=n-2$ will be similar.

\medskip

Since $|\omega^\ell|$ is arbitrarily close to $1$ by \eqref{eq:prel}, it suffices to show that
\begin{equation}\label{eq:bdnegla}
	\int\varphi(\Delta|\omega^\ell|)^-\le C(n,\theta)\eta\, ,
\end{equation}
where $\varphi$ is a good cut-off function interpolating between $A'$ and $A''$, provided $\delta \le \delta_0(\eta,n,\theta)$.
From \eqref{eq:bdnegla} it will follow that
\begin{equation}
	\begin{split}
		\int\varphi|\Delta|\omega^\ell|| &=
		\int\varphi \Delta|\omega^\ell| + 2 \int\varphi (\Delta|\omega^\ell|)^{-}
		\\& 
        = \int \Delta \varphi (|\omega^\ell| - 1) + 2 \int\varphi(\Delta|\omega^\ell|)^{-}
		\\& \le C(n,\theta) \eta \, ,
	\end{split}
\end{equation}
as desired. The remaining part of this proof is devoted to showing \eqref{eq:bdnegla}.

\medskip

Recall that
a linear endomorphism $L:V\to V$ on a vector space $V$ gives rise to a well-defined endomorphism $\hat{L}:\Lambda^{k}V\to\Lambda^k V$, given by
\begin{equation} 
	\hat{L}(z_1\wedge\dots\wedge z_k):=\sum_{i=1}^k z_1\wedge\dots\wedge z_{i-1}\wedge Lz_i\wedge z_{i+1}\wedge\dots\wedge z_k\, .
\end{equation}
Also, if $L\ge-\delta I$ with respect to a given positive definite scalar product, then
\begin{equation}
	\label{pos.def}
	\ang{\hat{L}(z_1\wedge\dots\wedge z_k),z_1\wedge\dots\wedge z_k}\ge -k\delta|z_1\wedge\dots\wedge z_k|^2\, .
\end{equation}
Indeed, by homogeneity we can assume that the vectors $z_1,\dots,z_k$ are orthonormal, so that
$$ \ang{z_1\wedge\dots\wedge z_{i-1}\wedge Lz_i\wedge z_{i+1}\wedge\dots\wedge z_k,z_1\wedge\dots\wedge z_k}=\ang{Lz_i,z_i}\ge-\delta\, , $$
from which the claim follows by summing over $i=1,\dots,k$.

\medskip

We can apply the previous observation to the Ricci tensor $\operatorname{Ric}=\operatorname{Ric}_g$, viewed pointwise as a symmetric endomorphism on the space of covectors $T_x^*M$, for every $x\in M$.\\ 
Using Bochner's formula and denoting $w:=(v,u)$, since $\Delta v=0$ we get
\begin{align}
	\Delta\omega^\ell= \hat{\operatorname{Ric}}(\omega^\ell)
	+2\sum_{1\le a<b\le\ell}\sum_{j=1}^{\ell}\dots\wedge\nabla_j(dw^a)\wedge\dots\wedge\nabla_j(dw^b)\wedge\dots\, , \label{eq:Bochomega}
\end{align}
when $\ell < n-2$, and
\begin{align}\label{eq:Bochomega2}
\begin{aligned}
	\Delta\omega^{n-2}&=dv^1\wedge\dots\wedge dv^{n-3}\wedge d\Delta u+\hat{\operatorname{Ric}}(\omega^{n-2})\\
	&\quad+2\sum_{1\le a<b\le n-2}\sum_{j=1}^{n-2}\dots\wedge\nabla_j(dw^a)\wedge\dots\wedge\nabla_j(dw^b)\wedge\dots\, .
\end{aligned}
\end{align}
In \eqref{eq:Bochomega} and \eqref{eq:Bochomega2}, $\nabla_j$ denotes the covariant derivative in a local orthonormal frame, and
the omitted factors are just $dw^i$, for $i\notin\{a,b\}$, without the covariant derivative.

\medskip

Since $|\nabla w|\le C(n,\theta)$ on $A'''$, the last sum in \eqref{eq:Bochomega} and \eqref{eq:Bochomega2} is bounded by 
$$C(n,\theta)\sum_{a<b}| \omega^\ell| |\operatorname{Hess}(w^a)||\operatorname{Hess}(w^b)|$$
in absolute value.
Hence, using also \eqref{pos.def}, we obtain
$$\ang{\Delta\omega^\ell,\omega^\ell}
\ge
-C(n)\delta|\omega^\ell|^2
-C(n,\theta)| \omega^\ell|\sum_{a<b}|\operatorname{Hess}(w^a)||\operatorname{Hess}(w^b)|$$
for $\ell<n-2$, while
\begin{equation}\label{eq:laplaomegaomega}
	\begin{split}
		\ang{\Delta\omega^{n-2},\omega^{n-2}}
		&\ge\ang{dv^1\wedge\dots\wedge dv^{n-3}\wedge d\Delta u,\omega^{n-2}}
		\\& \quad
		-C(n)\delta|\omega^{n-2}|^2
		-C(n,\theta)| \omega^{n-2}|\sum_{a<b}|\operatorname{Hess}(w^a)||\operatorname{Hess}(w^b)|\, .
	\end{split}
\end{equation}
Recalling \eqref{lapl.u}, namely
\begin{equation}
\Delta u=\frac{n|\nabla b|^2-|\nabla u|^2-|\nabla v|^2}{u}=\frac{n|\nabla b|^2-|\nabla w|^2}{u}\, , 
\end{equation}
we get
\begin{equation} 
	d\Delta u=-\frac{\Delta u}{u}du+\frac{2}{u}(n\operatorname{Hess}(b)[\nabla b,\cdot]-\operatorname{Hess}(w)[\nabla w,\cdot])\, ,
\end{equation}
where we abbreviate $\operatorname{Hess}(w)[\nabla w,\cdot]:=\sum_i\operatorname{Hess}(w^i)[\nabla w^i,\cdot]$.
Hence,
\begin{equation*} 
	\ang{dv^1\wedge\dots\wedge dv^{n-3}\wedge d\Delta u,\omega^{n-2}}
	\ge -\frac{\Delta u}{u}|\omega^{n-2}|^2-C(n,\theta)\left(|\nabla|\nabla b||+\sum_{i=1}^{n-2}|\nabla|\nabla w^i||\right)|\omega^{n-2}|\, . 
\end{equation*}
Recall now \eqref{delta.uno}, namely
\begin{equation} 
	\Delta|\omega^\ell|\ge\frac{\ang{\Delta\omega^\ell,\omega^\ell}+|\nabla\omega^\ell|^2-|\nabla|\omega^\ell||^2}{|\omega^\ell|}\, . 
\end{equation}
Since $|\nabla|\omega^\ell||\le|\nabla\omega^\ell|$, using \eqref{eq:laplaomegaomega} to lower bound $\ang{\Delta\omega^\ell,\omega^\ell}$
we immediately deduce that
\begin{equation*}
    \Delta|\omega^\ell|
	\ge C(n)\delta |\omega^\ell|
		-C(n,\theta)\sum_{a<b}|\operatorname{Hess}(w^a)||\operatorname{Hess}(w^b)| \, ,
		\quad \text{for $\ell < n-2$,}
\end{equation*}
while
\begin{align}\label{rough.bis}\begin{aligned}
		\Delta|\omega^{n-2}|
		&\ge-\left(\frac{\Delta u}{u}+C(n)\delta\right)|\omega^{n-2}|
		-C(n,\theta)|\nabla|\nabla b||-C(n,\theta)\sum_{i=1}^{n-2}|\nabla|\nabla w^i|| \\
		&\quad-C(n,\theta)\sum_{a<b}|\operatorname{Hess}(w^a)||\operatorname{Hess}(w^b)|
		\, .
\end{aligned}\end{align}
We let $f_1:=\delta|\omega^\ell|+|\nabla|\nabla b||+\sum_i|\nabla|\nabla w^i||$ and $f_2:=\sum_{(a,b)\neq(n-2,n-2)}|\operatorname{Hess}(w^a)||\operatorname{Hess}(w^b)|$.
As seen in the proof of \autoref{abs.sobolev}, we can bound $|\nabla|\nabla b||\le|\operatorname{Hess}(b)|$ in $L^2(A'')$, and the same holds for $w$. Hence,
\begin{equation}
	\label{f1}
	\int_{A''}|f_1|^2\le C(n,\theta)\, .
\end{equation}
On the other hand, by Cauchy--Schwarz
\begin{equation}
	\label{hess.hess}
	\int_{A''}f_2
	\le C(n,\theta)\delta\, ,
\end{equation}
since, in each term of the sum defining $f_2$, either $w^a$ or $w^b$ is one of the $\delta$-splitting maps $v^i$.
This is enough to prove \eqref{eq:bdnegla} when $\ell < n-2$. 
\medskip

Rather, in the case $\ell = n-2$, so far we only have that $\int_{A''}(\Delta|\omega^{n-2}|)^- \le C(n,\theta)$.
From now on, we focus on the remaining case $\ell = n-2$.
Given $\tau\in(0,\mz)$, we can find a subset $E\subset A''$ such that $\haus^n(E)\le\tau$ and on $A''\setminus E$ the following properties hold:
\begin{itemize}
	\item[(i)] $|\omega^\ell|\in(\frac{1}{1+\tau},1+\tau)$ (this follows from \eqref{eq:prel});
	\item[(ii)] $||\nabla b|^2-1|+\sum_{a,b}|\ang{\nabla w^a,\nabla w^b}-\delta^{ab}|\le\tau$ (again using \eqref{eq:prel});
	\item[(iii)] $|\nabla|\nabla b||+\sum_i|\nabla|\nabla w^i||+|\nabla|\omega^\ell||\le\tau$ (using \autoref{abs.sobolev});
	\item[(iv)] $f_2\le\tau$ (using \eqref{hess.hess} and taking $\delta$ small enough);
    \item[(v)] $|\Hess b|+|\Hess w|\le C(n,\theta)$.
\end{itemize}
On $A''\setminus E$, using again \eqref{delta.uno} (but without throwing away the term $\frac{|\nabla\omega^\ell|^2-|\nabla|\omega^\ell||^2}{|\omega^\ell|}$) we obtain
\begin{align*}
	\Delta|\omega^{n-2}|&\ge\frac{|\nabla\omega^{n-2}|^2-|\nabla|\omega^{n-2}||^2+\ang{\Delta\omega^{n-2},\omega^{n-2}}}{|\omega^{n-2}|} \\
	&\ge(1-\tau)|\nabla\omega^{n-2}|^2-(1+\tau)|\nabla|\omega^{n-2}||^2-\frac{\Delta u}{u}|\omega^{n-2}|-C(n,\theta)(f_1+f_2) \\
	&\ge(1-\tau)|\nabla\omega^{n-2}|^2-\frac{\Delta u}{u}|\omega^{n-2}|-C(n,\theta)\tau\, .
\end{align*}
Moreover,
\begin{equation}
	|\nabla\omega^{n-2}|^2
	\ge |dv^1\wedge\dots\wedge dv^{n-2}\wedge\nabla du|^2-C(n,\theta)\tau\, .
\end{equation}
At each point $x\in A''\setminus E$, let $W_x:=\bigcap_{i=1}^{n-3}\operatorname{ker}(dv^i)$, which is a $3$-dimensional subspace since
the $1$-forms $dv^i$ are almost orthonormal (for $\tau\le c(n)$ small enough).\\ 
Given a local orthonormal frame $\{e_j\}$, denoting $\nabla_j=\nabla_{e_j}$, we have
\begin{align*}
	|dv^1\wedge\dots\wedge dv^{n-3}\wedge\nabla du|^2
	&=\sum_{j=1}^n|dv^1\wedge\dots\wedge dv^{n-3}\wedge\nabla_j du|^2\\
	&=\sum_{j=1}^n|dv^1\wedge\dots\wedge dv^{n-3}\wedge\operatorname{Hess}(u)[e_j,\cdot]|^2\\
	&=\sum_{j=1}^n|dv^1\wedge\dots\wedge dv^{n-3}|^2|\operatorname{Hess}(u)[e_j,\cdot]\circ\iota_{W}|^2\\
	&=|dv^1\wedge\dots\wedge dv^{n-3}|^2|\operatorname{Hess}(u)\circ\iota_W|^2\\
	&\ge(1-C(n)\tau)|\operatorname{Hess}(u)\circ\iota_W|^2\, ,
\end{align*}
where $\iota_W$ denotes the inclusion $W_x\to T_xM$ at each $x$, and where we view $\operatorname{Hess}(u)$ as an endomorphism $T_xM\to T_xM$ after the last equality.
Moreover, by definition of $E$,
\begin{equation} 
	\frac{\Delta u}{u}=\frac{n|\nabla b|^2-|\nabla w|^2}{u^2}\le \frac{2}{u^2}+C(n)\tau\, ,\quad \text{on $A''\setminus E$}\, .
\end{equation}
To sum up, we have
\begin{equation}
	\label{almost.done}
	\Delta|\omega^{n-2}|\ge(1-C(n)\tau)|\operatorname{Hess}(u)\circ\iota_W|^2-\frac{2}{u^2}-C(n,\theta)\tau
\end{equation}
on $A''\setminus E$.

\medskip

From \autoref{thm:prop_Green} (iii) we deduce that
\begin{equation} 
	\left|\operatorname{Hess}(b)-\frac{1}{b}\left(I-\frac{\nabla b}{|\nabla b|}\otimes\frac{\nabla b}{|\nabla b|}\right)\right|^2\le \tau \quad  \text{on $A''\setminus E'$} \, ,
\end{equation}
where $\haus^n(E')\le \tau$, provided $\delta \le \delta_0(\tau,n,\theta)$.
Since
\begin{equation}
	\operatorname{Hess}(u)=\frac{\nabla b\otimes\nabla b-\nabla w\otimes\nabla w+b\operatorname{Hess}(b)-v\operatorname{Hess}(v)}{u}\, ,
\end{equation}
and $\operatorname{Hess}(v)$ is small we get
\begin{equation} 
	\left|\operatorname{Hess}(u)-\frac{I-\nabla w\otimes\nabla w}{u}\right|\le C(n)\tau\, .
\end{equation}
Since $I-\nabla w\otimes\nabla w$ is close to the orthogonal projection $P$ onto $\operatorname{ker}(du)\cap W$, which has rank $2$ and satisfies $|P\circ\iota_W|^2=|P|^2=2$, this is enough to conclude that
\begin{equation}\label{eq:hessnormfrombl}
	|\operatorname{Hess}(u)\circ\iota_W|^2\ge\frac{2}{u^2}-C(n)\tau\, , 
	\quad \text{on $A''\setminus (E \cup E')$}\, .
\end{equation}
Plugging \eqref{eq:hessnormfrombl} in \eqref{almost.done}, we finally see that
\begin{equation} 
	\Delta|\omega^{n-2}|\ge -C(n)\tau|\operatorname{Hess}(u)|^2-C(n)\tau
\end{equation}
on the complement of $E\cup E'$. On the other hand, on $E\cup E'$ using \eqref{rough.bis} we have
\begin{equation} 
	\begin{split}
		\int_{A''\cap (E\cup E')}(\Delta|\omega^{n-2}|)^-
		&\le C(n,\theta)\int_{A''\cap(E\cup E')}(1+f_1+f_2)
		\\&
		\le C(n,\theta)\haus^n(E\cup E')^{1/2}+C(n,\theta)\delta\, , 
	\end{split}
\end{equation}
by \eqref{f1} and \eqref{hess.hess}. In conclusion,
\begin{equation} 
	\int_{A''}(\Delta|\omega^{n-2}|)^-\le C(n,\theta)\tau\, .
\end{equation}

\subsection{Proof of \autoref{step2}}

We argue by contradiction as in \cite{CheegerNaber15}. If the conclusion were false, we could find $\eps>0$, a sequence $\eta_j \to 0$, and complete smooth pointed manifolds $(M^n_j, g_j, p_j)$ satisfying the following properties:

\begin{itemize}
	\item[(i)] $\Ric_{g_i} \ge -\eta_i(n-2)$;
	
	\item[(ii)] $B_{100}(p_j)$ is $\delta(n)$-GH close to the ball of radius $100$ over the tip $o$ of a cone $\R^{n-3} \times C(Z_j)$ with density $\theta_j \ge \theta$;

	\item[(iii)] there exist $\eta_j$-splitting maps $v_j: B_{40}(p_j) \to \R^{n-2}$, good Green distances $b_j : B_{40}(p_j) \to (0,\infty)$ centered at $p_j$, and $q_j \in A_{9,8}(p_j)\cap \{|v_j| < 1\}$ such that $L\circ w_j : B_s(q_j) \to \R^{n-2}$ is not an $\eps$-splitting map for every lower triangular matrix $L$, where $s \ge s_{q_j}$ depends on $j$.
\end{itemize}

Let $s_j$ be the supremum of the bad radii $s$ as in (iii). In particular, there exists an lower triangular matrix with positive diagonal entries such that $L_j \circ w_j : B_{2s_j}(q_j) \to \R^{n-2}$ is an $\eps$-splitting map.
It is not hard to see that $s_j \to 0$. We then consider the (noncollapsing) sequence of pointed manifolds $(M^n, s_j^{-2}g_j, q_j)$ and maps 
\begin{equation}
	\hat w_j : = s_j^{-1}L_j \circ (w_j - w_j(q_j)) : B_{s_j^{-1}}(q_j) \to \R^{n-2} \, ,
\end{equation}
which are $\eps$-splitting in the ball $B_2(q_j)$ in the rescaled metric. In particular
\begin{equation}\label{eqz15}
	\fint_{B_2(q_j)} |\langle \nabla \hat w^a_j, \nabla \hat w^b_j\rangle - \delta^{ab}| \le \eps \, .
\end{equation}
Moreover, for every $2\le r\le cs_j^{-1}$ there is a triangular matrix $L_r$ such that $L_r \circ \hat w_j : B_r(q_j) \to \R^{n-2}$ is an $\eps$-splitting map. The first $n-3$ components of $w_j$ are harmonic; hence, by \cite[Theorem 1.32]{CheegerNaber15}, up to composing with an lower triangular matrix, they are almost splitting above the singular scale. Arguing as in \cite[Claim 2, p.\ 1119]{CheegerNaber15} we can assume that the first $n-3$ components of $L_r \circ \hat w_j$ produce an $\eps_j$-splitting map in $B_r(q_j)$ for every $2\le r \le cs_{j}^{-1}$, where $\eps_j \to 0$.

Arguing as in \cite[Claim 1, p. 1118]{CheegerNaber15}, we deduce that 
\begin{equation}
	|L_r| + |L_r^{-1}| \le  r^{ C(n)\eps}\,  ,
	\quad
	\text{for $2 \le r \le cs_j^{-1}$} \, .
\end{equation}
In particular,
\begin{equation}\label{eqz16}
	\begin{split}
		&r^2\fint_{B_r(p_j)}| \Hess (\hat w_j)|^2 \le C(n)\eps r^{2C(n)\eps} \, ,
		\\
		&\fint_{B_r(q_j)} |\nabla \hat w_j| \le (1 + C(n)\eps) r^{C(n)\eps} \, ,\quad
		\text{for every $2\le r\le cs_{j}^{-1}$}\, .
	\end{split}
\end{equation}
Recalling \eqref{lapl.u}, which gives $|\Delta u_j|\le C$ (on the domain $A_j'$, with the metric $g_j$), and the fact that $\Delta v_j=0$, we have
\begin{equation}\label{eqz17}
|\Delta\hat w_j|
=
s_j^{-1}|L_j\circ\Delta(w_j^1,\dots,w_j^{\ell+1})|
\le
s_j^{-1}|L_j||\Delta w_j|
\le 
s_j^{1-C\eps}\to 0
\, , \quad \text{on $B_{cs_{j}^{-1}}(q_j)$}\,  .
\end{equation}
Up to the extraction of a subsequence that we do not relabel, we can pass to the pointed GH-limit $(M^n, s_j^{-2}, q_j) \to (X,\dist,q)$ where $(X,\dist,\haus^n)$ is an $\RCD(0,n)$ space. Up to sub-sequence, the maps $\hat w_j$ limit to a map $\hat w: X \to \R^{n-2}$ whose first $(n-3)$ components are splitting maps. In particular, $X = \R^{n-3} \times Z$ and $\hat w^a(x,z) = x^a$ where $x\in \R^{n-3}$, $z\in Z$, $a\le n-3$.
Also the last component $\hat w^{n-2}$ is harmonic, as a consequence of \eqref{eqz17}. Arguing as in \cite[Claim 5, p.\ 1127]{CheegerNaber15}, using the sublinear growth estimates \eqref{eqz16} we deduce that 
\begin{equation}\label{eqz20}
	\hat w^{n-2}(x,z) = \xi \cdot x + w_Z^{n-2}(z) \, ,
\end{equation}
where $\xi \in \R^{n-3}$ and $w_Z^{n-2} : Z \to \R$ is harmonic. Moreover, $w_Z^{n-2} \neq 0$ as a consequence of \eqref{eqz15}.
The contradiction will follow if we can establish that
\begin{equation}\label{goal2}
	r\fint_{B_r(q_j)} |\Hess \hat w^{n-2}_j|^2 \to 0 \, , \quad \text{as $j\to \infty$}\, ,
\end{equation}
for every $2\le r \le cs_j^{-1}$. Indeed, \eqref{goal2} implies that $\omega_Z^{n-2}$ is parallel and forces a splitting. Therefore, there exists a lower triangular matrix $L$ (which is actually the identity in the first $(n-3)\times (n-3)$ block) such that $L\circ w$ is a splitting map. It is not hard to deduce that $L\circ \hat w_j$ is an $\eps_j'$-splitting map in $B_{1}(q_j)$, where $\eps_j'\to 0$. This would give the sought contradiction.

\medskip

From now on, our goal will be to prove \eqref{goal2}. This requires a study of the form
\begin{equation}
	\hat \omega := d \hat w^1 \wedge \ldots \wedge d \hat w^{n-2}\, .
\end{equation}
We begin by noticing that $\hat \omega_j = \det(L_j) \omega_j$, hence
\begin{equation}\label{eq:key}
	r^2 \fint_{B_r(q_j)} |\Delta |\hat \omega_j|| \le \eta_j \fint_{B_r(q_j)} |\hat \omega_j|
	\le \eta_j C(n)r^{C(n)\eps}\, ,
	\quad \text{for $2 \le r \le cs_{j}^{-1}$}\, .
\end{equation}
As a first step, we need to show that, for these radii $r$, it holds
\begin{equation}\label{eqz18}
	\fint_{B_r(q_j)} \left| |\hat \omega_j|^2 - \fint_{B_r(q_j)} |\hat \omega_j|^2 \right| \to 0 \, ,
	\quad \text{as $j\to \infty$}\, .
\end{equation}
This follows verbatim from the argument in \cite[Claim 3, p.\ 1121]{CheegerNaber15}. We stress that the key ingredient is \eqref{eq:key}.

The second key claim is that
\begin{equation}\label{eqz19}
	\fint_{B_r(q_j)} |\nabla \hat \omega_j|^2 \to 0
	\, ,
	\quad \text{as $j\to \infty$}\, .
\end{equation}
This claim corresponds to \cite[Claim 4, p.\ 1121]{CheegerNaber15}: it follows by integrating the Bochner inequality against a good cut-off function and then using \eqref{eqz18}. However, in our setting we have to deal with an extra term in the Bochner identity. Specifically, in the expression of $\Delta |\hat \omega_j|^2$ we have the additional term
\begin{equation}
	2\ang{d\hat w_j^1\wedge\dots\wedge d\hat w_j^{n-3}\wedge d\Delta\hat w_j^{n-2},\hat\omega_j}\, . 
\end{equation}
Note that $\Delta\hat w_j^i=0$ for $i<n-2$, since in this case $\hat w_j^i$ is just a linear combination of the harmonic maps $v_j^1,\dots,v_j^i,1$, as the transformation matrix is lower triangular.
However, on $B_r(q_j)$ we have the bound $|\nabla\hat w_j|\le C(n,\theta)r^{C(n)\eps}$, while from \eqref{lapl.u} we see that
\begin{equation} 
	|d\Delta u_j|\le C(n,\theta)(1+|\operatorname{Hess}(b_j)|+|\operatorname{Hess}(w_j)|) 
\end{equation}
with respect to the metric $g_j$. In particular,
\begin{equation}
	|\operatorname{Hess}(w_j)|\le s_j|L_j^{-1}||\operatorname{Hess}(\hat w_j)|\, . 
\end{equation}
Also, since $b_j=\sqrt{u^2_j+|v_j|^2}$, it holds $|\operatorname{Hess}(b_j)|\le C(n)|\operatorname{Hess}(w_j)|$.
Combining these bounds, we get
\begin{align*}
|d\Delta\hat w_j^{n-2}|
& = s_j^{-1}|(L_j)_{n-2,n-2}d\Delta u_j|
\\&
\le s_j^{-1}| L_j| |d \Delta u_j|
\\&
\le C s_j^{-1}|L_j|(1 + s_j |L_j^{-1}| |\operatorname{Hess}(\widehat w_j)|)
\\&
\le Cs_j^{-1-C\eps} + 
 C s_j^{-C\eps}|\operatorname{Hess}(\hat w_j)| 
\end{align*}
with respect to $g_j$, and hence
\begin{equation}\label{eqz21}
	|d\Delta \hat w_j^{n-2}|\le Cs_j^{2 - C\eps} + Cs_j^{1-C\eps}|\operatorname{Hess}(\hat w_j)| 
\end{equation}
with respect to the rescaled metric. The latter is infinitesimal in $L^2$, on any ball $B_r(q_j)$.

\medskip

The final step involves combining \eqref{eqz19} and \eqref{eqz20} to get \eqref{goal2}. This corresponds to \cite[Claim 6, p.\ 1129]{CheegerNaber15}, and its proof can be adapted to our setting with one change in the very last part of the argument. Indeed, in \cite[Claim 6, p.\ 1129]{CheegerNaber15} it is first proved that
\begin{equation}
	\fint_{B_r(q_j)} | |\nabla \hat w^{n-2}_j|^2 - 1| \to 0 \, ,
	\quad \text{for $2\le r\le cs_j^{-1}$}\, ,
\end{equation}
and then the Hessian bound \eqref{goal2} follows from the Bochner inequality, using the harmonicity of $\hat w^{n-2}_j$. In our setting, $\hat w^{n-2}_j$ is not harmonic, producing the extra term in the Bochner inequality
\begin{equation}
	\langle \nabla \hat w^{n-2}_j \, , \nabla \Delta \hat w^{n-2}_j\rangle \, .
\end{equation}
However, the latter is infinitesimal in $L^2(B_r(q_j))$ by \eqref{eqz21}.

\section{The topology of good level sets}\label{sec:topgoodlevels}

This section aims to discuss the (quantitative) topological regularity of \emph{good level sets} of almost splitting maps $u:B_1(p)\to \setR^{n-2}$, where $B_1(p)\subset X$ and $(X,\dist,\haus^n)$ is an $\RCD(-(n-1),n)$ space. Good level sets are those levels for which the map $u$ remains an almost splitting map up to transformation at all locations and sufficiently small scales on the level set. The results of Section \ref{sec:slicing} guarantee the existence of (plenty of) good level sets for suitable almost splitting maps $u:B_1(p)\to \setR^{n-2}$. 
The end goal is to prove that good level sets are locally uniformly contractible topological surfaces with empty boundary when the ambient space $(X,\dist,\haus^n)$ has empty boundary in the sense of \cite{DePhilippisGigli2}. The statement will be uniform among all $\RCD(-(n-1),n)$ spaces such that $\haus^n(B_1(p))\ge v$ for a given $v>0$. This uniformity will be key for many of the subsequent applications.

\begin{proposition}\label{prop:topsurf}
Let $(X,\dist,\haus^n)$ be an $\RCD(-(n-1),n)$ space with empty boundary such that $\haus^n(B_1(p))\ge v>0$ and let $u:B_2(p)\to \setR^{n-2}$ be an almost splitting map. Assume that $u(p)=0$, and for each $q\in B_1(p)\cap\{u=0\}$ and each $0<r<1$ there exists an lower triangular $(n-2)\times (n-2)$ matrix $T_{q,r}$ such that $T_{q,r}\circ u:B_r(q)\to \setR^{n-2}$ is a $\delta$-splitting map.
If $\delta<\delta_0(n,v)$ the following hold:
	\begin{itemize}
		\item[(i)] $\{u=0\}\cap B_1(p)$ is a topological surface;
		
		\item[(ii)] if $r\le c(n,v)$ and $q\in \{u=0\}\cap B_{1/2}(p)$, then $B_r(q)\cap \{u=0\}$ is $2$-connected in $B_{\overline C r}(q)\cap \{u=0\}$, for some $\overline C>1$. Explicitly, for $k\in\{0,1,2\}$, any continuous map $S^k \to B_r(q)\cap \{u=0\}$ can be extended to a map $\overline{D}^{k+1} \to B_{\overline C r}(q)\cap \{u=0\}$.
	\end{itemize}
\end{proposition}

\begin{remark}
In the case where $(Y,\dist_Y,\haus^2)$ is an $\RCD(-1,2)$ space (and hence an Alexandrov space with curvature $\ge -1$), $X:=\setR^{n-2}\times Y$ with the canonical product structure, and $u:X\to\setR^{n-2}$ is the projection onto the first factor, the statements above reduce to \autoref{cor:Alsurf} and \autoref{prop:Ale2loccontr}. 
\end{remark}

The proof of \autoref{prop:topsurf} will be achieved through a series of intermediate steps. We let $\eps_0=\eps_0(n)>0$ be fixed such that the metric Reifenberg theorem from \cite{CheegerColding97I} applies for $\eps<\eps_0$, i.e., $\mathcal{R}_{\eps}(X)$ is a topological $n$-manifold; see \autoref{def:effreg} for the relevant notation.

\begin{itemize}
\item[(i)] In \autoref{lemma:finiteint} we shall see that the intersection of a good level set with $X\setminus \mathcal{R}_{\eps_0}$ is locally finite. This statement will be crucial later in Section \ref{sec:topman}.

\item[(ii)]  We will prove that each good level set is a topological surface away from the (finitely many) intersection points with $X\setminus \mathcal{R}_{\eps_0}$. Then we are going to extend the conclusion to the isolated bad points, concluding the proof of \autoref{prop:topsurf} (i).

\item[(iii)] The proof of the local uniform contractibility of good level sets will hinge on the local splitting at all locations and sufficiently small scales along a good level set and on the local uniform contractibility of Alexandrov surfaces with nonnegative curvature and (uniform) Euclidean volume growth, as in \autoref{cor:2dAVR} and \autoref{prop:Ale2loccontr}.

\end{itemize}

As mentioned above, the first step will be to show that each good level set intersects $X\setminus \mathcal{R}_{\epsilon_0}$ at locally finitely many points. The precise statement follows.

\begin{lemma}\label{lemma:finiteint}
Let $(X,\dist,\haus^n)$ be an $\RCD(-(n-1),n)$ space with empty boundary such that $\haus^n(B_1(p))\ge v>0$, and let $u:B_2(p)\to \setR^{n-2}$ be an almost splitting map. Assume that for each $q\in B_1(p)\cap\{u=0\}$ and each $0<r<1$ there exists an lower triangular $(n-2)\times (n-2)$ matrix $T_{q,r}$ such that $T_{q,r}\circ u:B_r(q)\to \setR^{n-2}$ is a $\delta$-splitting map.
If $\delta<\delta_0(\eps,v,n)$ then $(\{u=0\}\cap B_1(p))\setminus \mathcal{R}_{\eps}$ is a finite set. 
\end{lemma}

The rest of this section is organized as follows. In Section \ref{subsec:prooflemmafinit} we are going to prove \autoref{lemma:finiteint}. In Section \ref{subsec:topsurfi)} we will prove \autoref{prop:topsurf} (i). In Section \ref{subsec:prooftopsurfii)} we will prove \autoref{prop:topsurf} (ii).

\subsection{Proof of \autoref{lemma:finiteint}}\label{subsec:prooflemmafinit}
For the sake of illustration, we first discuss the argument in the case $\delta=0$; namely, we recall how to prove that for any Alexandrov surface with empty boundary the effective singular set is locally finite. The statement is well known (see for instance \cite{LytchakStadler}). It is helpful to report the argument here as the proof of \autoref{lemma:finiteint} in full generality is an effective variant of it.
\medskip

For an Alexandrov surface $(Z,\dist_Z)$ with empty boundary and curvature bounded from below by $-1$, all tangent cones are unique and are metric cones over circles with diameter $\le \pi$. In particular, they are smooth (and flat) away from the tip. Assume by contradiction that there is a sequence $z_i\in Z\setminus \mathcal{R}_{\eps}(Z)$ such that $z_i\to z_{\infty}$ as $i\to \infty$, and $z_i\neq z_{\infty}$ for each $i\in\setN$. Since $z_i\in Z\setminus \mathcal{R}_{\eps}(Z)$, the classical $\eps$-regularity theorem in the context of Ricci curvature (see \cite{DePhilippisGigli2} after \cite{CheegerColding97I}) implies that
\begin{equation}\label{eq:densgap}
\theta_Z(z_i)\le 1-\eta_0
\end{equation}
for each $i\in\setN$, where
\begin{equation}
\theta_Z(z):=\lim_{r\to 0}\frac{\haus^2(B_r(z))}{\pi r^2} \, ,
\end{equation}
and $\eta_0=\eta_0(\eps)>0$. 
Let $r_i:=\dist_Z(z_i,z_{\infty})>0$ for each $i\in\setN$. We consider the sequence of pointed metric spaces $(Z_i,\dist_{Z_i},z_{\infty}):=
(Z,r_i^{-1}\dist_Z,z_{\infty})$. Up to the extraction of a subsequence that we do not relabel, $(Z_i,z_{\infty})\to (C(S^1_{\ell}),o)$ in the pointed Gromov--Hausdoff sense, where $C(S^1_{\ell})$ is the tangent cone of $(Z,\dist_Z)$ at $z_{\infty}$ and $o\in C(S^1_{\ell})$ denotes the vertex. Up to the extraction of a further subsequence, $Z_i\ni z_i\to\overline{z}\in C(S^1_{\ell})$ for some point $\overline{z}$ such that $\dist_{C(S^1_{\ell})}(\overline{z},o)=1$. By the smoothness of $C(S^1_{\ell})$ away from the vertex, 
\begin{equation}\label{eq:limdens1}
\theta_{C(S^1_{\ell})}(\overline{z})=1\, .
\end{equation}
On the other hand, by lower semicontinuity of the density with respect to pointed Gromov--Hausdorff convergence (see for instance \cite[Lemma 2.2]{DePhilippisGigli2}) and \eqref{eq:densgap}, we have
\begin{equation}
\theta_{C(S^1_{\ell})}(\overline{z})\le 1-\eta_0<1\, ,
\end{equation}
a contradiction to \eqref{eq:limdens1}.

\medskip

Next, we discuss the proof in the general case.
We assume by contradiction that there exists a sequence $z_i\in (\{u=0\}\cap B_1(p))\setminus \mathcal{R}_{\eps}$ with $z_i\to z_{\infty}$ as $i\to\infty$ and $r_i:=\dist(z_i,z_{\infty})>0$ for each $i\in\setN$. Since $z_i\in X\setminus \mathcal{R}_{\eps}(X)$, as above the classical $\eps$-regularity theorem in the context of Ricci curvature (see \cite{DePhilippisGigli2} after \cite{CheegerColding97I}) implies that
\begin{equation}\label{eq:densgap-bis}
\theta_X(z_i)\le 1-\eta_0
\end{equation}
for each $i\in\setN$, where
\begin{equation}\label{eq:densgap2}
\theta_X(z):=\lim_{r\to 0}\frac{\haus^n(B_r(z))}{\omega_nr^n}
\end{equation}
for each $z\in X$, and $\eta_0=\eta_0(\eps)>0$. We let $(X_i,\dist_i,z_{\infty}):=(X,r_i^{-1}\dist,z_{\infty})$ and $u_i:=r_i^{-1}T_{z_\infty,2r_i}\circ u:B^{X_i}_2(z_{\infty})\to \setR^{n-2}$ (where $T_{z_\infty,2r_i}$ are the $(n-2)\times (n-2)$ transformation matrices at scale $2r_i$ with center $z_\infty$, as in the assumptions). Up to the extraction of a subsequence that we do not relabel, $(X_i,\dist_i,z_{\infty})\to (C(Z),\dist_{C(Z)},o)$, as $i\to \infty$ in the pointed Gromov--Hausdorff sense for some $\RCD(n-2,n-1)$ space $(Z,\dist_Z,\haus^{n-1})$ with $\haus^{n-1}(Z)\ge v$. Moreover, $X_i\ni z_i\to \overline{z}\in C(Z)$ as $i\to\infty$ for some point $\overline{z}$ such that $\dist_{C(Z)}(\overline{z},o)=1$, and $u_i\to u_{\infty}$ uniformly along the sequence $X_i\to C(Z)$ for some limit $\delta$-splitting map $u_{\infty}:B_2(o)\to \setR^{n-2}$ such that $u(\overline{z})=u(o)=0$. 

As above, the lower semicontinuity of the density with respect to pointed Gromov--Hausdorff convergence and \eqref{eq:densgap-bis} yield
\begin{equation}\label{eq:limdensgap}
\theta_{C(Z)}(\overline{z})\le 1-\eta_0\, .
\end{equation}
We claim that \eqref{eq:limdensgap} yields to a contradiction for $\delta<\delta_0$. If not, there exist a sequence of $\RCD(n-2,n-1)$ spaces $(Z_j,\dist_j,\haus^{n-1})$ with $\haus^{n-1}(Z_j)\ge v$, a sequence of $1/j$-splitting maps $u_j:B_2(p_j)\to \setR^{n-2}$, and  points $\overline{z}_j\in C(Z_j)$ such that $u_j(\overline{z}_j)=u_j(o_j)=0$ and $\dist_{C(Z_j)}(\overline{z}_j,o_j)=1$, where  $o_j\in C(Z_j)$ is a vertex. Up to the extraction of a subsequence, $(C(Z_j),\dist_{C(Z_j)},o_j)\to (C(Z_{\infty}),\dist_{C(Z_{\infty})},o_{\infty})$ in the pointed Gromov--Hausdorff sense for some $\RCD(n-2,n-1)$ space $(Z_{\infty},\dist_{\infty},\haus^{n-1})$ such that $B_{3/2}(o_{\infty})$ is isometric to $B_{3/2}(0^{n-2},o)\subset \setR^{n-2}\times C(S^1_{\ell})$, where $o\in C(S^1_{\ell})$ denotes the vertex and $0<\ell\le 1$. Note that the absence of boundary in the limit cone follows from \cite{BrueNaberSemolabdry}.

We shall identify the two spaces $C(Z_{\infty})$ and $\setR^{n-2}\times C(S^1_{\ell})$ in the sequel with a slight abuse of notation.
Moreover, up to the extraction of a further subsequence, $C(Z_j)\ni \overline{z}_j\to \overline{z}_{\infty}=(0^{n-2},\tilde{z}_{\infty})\in C(Z_\infty)$ for some point $\tilde{z}_{\infty}\in C(S^1_{\ell})$ such that $\dist_{C(S^1_{\ell})}(\tilde{z}_{\infty},o)=1$. As already mentioned, $C(S^1_{\ell})$ is smooth away from the vertex $o$. Hence
\begin{equation}
\theta_{C(Z_{\infty})}(\overline{z}_{\infty})=1\, .
\end{equation}
By lower semicontinuity of the density, this results in a contradiction with the condition
\begin{equation}
\theta_{C(Z_i)}(\overline{z}_i)\le 1-\eta_0\, ,\quad \text{for each $i\in\setN$}\, ,
\end{equation}
as soon as $i\in \setN$ is large enough. 

\subsection{Proof of \autoref{prop:topsurf} (i)}\label{subsec:topsurfi)}

There are two main steps in the proof. In Step 1 we are going to prove that, for $\eps_0>0$ sufficiently small, for every point $q\in (\{u=0\}\cap B_1(p))\cap\mathcal{R}_{\epsilon_0}$ there is a neighbourhood $q\in U_q \subset X$ such that $U_q\cap \{u=0\}$ is homeomorphic to $\setR^2$. By \autoref{lemma:finiteint} this is sufficient to show that each good level set is a topological surface away from a discrete set of points. The aim of Step 2 will be to extend the topological regularity to these potentially bad points.
\medskip

\emph{Step 1:} We claim that, for $\eps_0>0$ sufficiently small, for every point $q$ in the set $(\{u=0\}\cap B_1(p))\cap\mathcal{R}_{\eps_0}$ there is a neighbourhood $U_q\ni q$ in $X$ such that $U_q\cap \{u=0\}$ is homeomorphic to $\setR^2$.

Indeed, if $q\in \mathcal{R}_{\epsilon_0}$ there exists $r>0$ such that $B_r(q)$ is $\eps_0 r$-close to $B_r(0)\subset \setR^n$. Then we can complete the $(n-2)$-almost splitting map $u:B_r(q)\to \setR^{n-2}$ to an almost splitting map $\overline{u}:B_r(q)\to \setR^{n}$ whose first $(n-2)$ components coincide with $u$ and with harmonic last two components. Arguing as in \cite[Section 7.5]{CheegerJiangNaber} (see also \cite[Section 3.3]{BrueNaberSemolabdry} for the $\RCD$ setting) it is possible to prove that $\overline{u}:B_{r/2}(q)\to \setR^{n}$ is a biH\"older homeomorphism and a Gromov--Hausdorff approximation with its image, which is an open set $U\subset \setR^n$ such that $B_{(1-\epsilon_0)r/2}(0)\subset U\subset B_{(1+\epsilon_0)r/2}(0)$. It follows that $\{u=0\}\cap B_{r/2}(q)$ is locally homeomorphic and Gromov--Hausdorff close to $B_{r/2}(0)\subset \setR^2$, as we claimed.

\medskip

\emph{Step 2:} We claim that also for the finitely many points $q\in (\{u=0\}\cap B_1(p))\setminus \mathcal{R}_{\eps_0}$ there exists a neighbourhood $q\in U_q\subset X$ such that $U_q\cap \{u=0\}$ is homeomorphic to $\setR^2$. Up to scaling and without loss of generality we can assume that $(\{u=0\}\cap B_1(p))\setminus \{q\}$ is a topological surface. We fix $\eta_0>0$ such that the metric Reifenberg theorem applies for $\eta_0$ below.
A slight variant of the argument employed in Step 1 proves the following claim: there exist $r = r(q)$ and $r_0=r_0(\eta)$ such that for each $0<s<r$ and for every $q'\in (B_{2s}(q)\setminus B_s(q))\cap \{u=0\}$ it holds
\begin{equation}
\dist_{\rm{GH}}(\{u=0\}\cap B_t(q'),B_t(0^2))\le \eta_0 t\, ,
\end{equation}
for each $0<t<r_0s$, provided that $\eps<\eps_0(\eta_0,n,v)$ and $\delta<\delta_0(\eta_0,n,v)$. Up to possibly choosing a smaller $r>0$, the claim can be exploited in conjunction with the metric Reifenberg theorem (alternatively, with the explicit construction of the local biH\"older homeomorphism as above) and the Gromov--Hausdorff closeness of
$\{u=0\}\cap B_s(q)$ to a cone
at all scales $s<r$ to construct a family of nested open neighbourhoods $U_i\ni q$ such that: 
\begin{itemize}
\item[(i)] $\overline{U}_i\setminus U_{i+1}\cap \{u=0\}$ is homeomorphic to a closed annulus $\overline{B}_1(0)\setminus B_{1/2}(0)\subset \setR^2$ for each $i\in\setN$;
\item[(ii)] $B_{2^{-(i+1)}r}(q)\subset U_i\subset B_{2^{-(i-1)}r}(q)$ for each $i\in \setN$
\end{itemize}
(see, e.g., the arguments from \cite[Appendix 1]{CheegerColding97I}).
The homeomorphisms with annuli in (i) can be rescaled and patched together by (ii) and (iii) to obtain a homeomorphism between $U_1\cap \{u=0\}\setminus \{q\}$ and $B_1(0)\setminus \{0\}\subset \setR^2$. This homeomorphism is easily seen to extend to a homeomorphism between $U_1\cap \{u=0\}$ and $B_1(0)\subset \setR^2$.

\begin{remark}
 In the case where the ambient space is a smooth Riemannian manifold, the proof of the topological regularity of good level sets could be greatly simplified. Indeed, arguing as in the proof of \cite[Theorem 7.10]{CheegerJiangNaber}, it is possible to check that $du$ has maximal rank at each point on each good level set. The topological regularity then follows from the implicit function theorem.

 We also note that for most purposes, it would be sufficient to assume that the level set is non-critical, as we can always neglect a set of values of $\Leb^{n-2}$-measure zero, according to Sard's theorem.
\end{remark}

\subsection{Proof of \autoref{prop:topsurf} (ii)}\label{subsec:prooftopsurfii)}

As a first step, we prove a lemma ensuring that the $\delta$-splitting map $T_{q,r}\circ u: B_r(q)\to \R^{n-2}$ can be used to build an almost isometry with the model space $\R^{n-2} \times Z^2$, which induces an almost isometry between (the ball in) the level set $\{u=0\}$ and (the ball in) $Z$.

\begin{lemma}\label{lemma:level GH}
Let $r\le c(n,v)$ and $q\in \{u=0\}\cap B_1(p)$. If $\delta \le \delta_0(\eps,n,v)$, there exists an Alexandrov surface $(Z,\dist_Z,z_0)$ with curvature $\ge 0$ such that the following holds.

\begin{itemize}
  \item[(i')] There exists a $\eps r$-GH isometry
  \begin{equation}
      \Psi : B_r(q) \to B_r(0^{n-2},z_0) \, ,
      \quad 
      (0^{n-2},z_0)\in \R^{n-2}\times Z \, ,
  \end{equation}
  whose Euclidean components are given by $T_{q,10nr}\circ u$.

   \item[(ii')] If $\eps \le \eps_0(\eps',n,v)$, then $\{u=0\}\cap B_r(q)$ is $\eps' r$-close to $B_r(z_0)$ through $\Psi$.
\end{itemize}
\end{lemma}

\begin{proof}
We suppress the dependency of constants on $v$ to ease notation.
Item (i') follows from \autoref{thm:splittings}, up to a slight dilation of $T_{q,10nr}$
to ensure that $T_{q,10nr}\circ u$ takes values in $B_r(0^{n-2},z_0)$. Since the first $(n-2)$-components of $\Psi$ are $T_{q,10nr}\circ u$, it is clear that 
\begin{equation}
    \Psi(\{u=0\}\cap B_r(q))
    \subseteq
    \{(0^{n-2},z)\mid z\in Z\}\cap B_r(0^{n-2},z_0) \, .
\end{equation}
Fix $\eps'>0$.
To prove (ii'), we need to show the opposite inclusion, i.e., that for every $(0^{n-2},z)\in B_r(0^{n-2},z_0)$ there exists $q''\in \{u=0\}\cap B_r(q)$ such that $\dist(\Psi(q''), (0^{n-2},z)) \le \eps' r$, provided $\eps \le \eps_0(\eps',n)$.
Since $Z$ is an Alexandrov surface with a lower bound on the area of each ball $B_r(z)$, there exists $s=s(\eps')\le \eps'$ and $z'\in B_{\eps' r}(z)$ such that $B_{sr}(z')$ is $\eps'sr$-GH close the Euclidean ball $B_{sr}(0^2)$.

Let $q'\in B_r(q)$ be a point $\eps r$-GH close to $(0^{n-2},z')$, i.e., $\dist(\Psi(q'),(0^{n-2},z')) \le \eps r$. 
Notice that $B_{sr}(q')$ is $C(n)\eps' sr$-GH close to the Euclidean ball $B_{sr}(0^n)$, if $\eps \le \eps_0(\eps',n)$.
Since $T_{q,10nr}\circ u$ gives the first $n-2$ components of $\Psi$, it turns out that $|T_{q,10nr}\circ u(q')|\le \eps r$. Our goal is to find $q''\in B_{\eps' r}(q')$ such that $T_{q,10nr}\circ u(q'')=0$. Once this is done, the proof of the lemma will be completed.

To this aim, we build a $\delta'$-splitting map $v: B_{sr/2}(q') \to \R^n$ whose first $n-2$ components are given by $T_{q,10nr}\circ u$ and such that $v(q')=(T_{q,10nr}\circ u(q'),0^2)$. This is possible thanks to \autoref{thm:splittings}, if $\delta' = \delta'(\eps',n)$ and $\eps\le \eps_0(\eps',n)$.

By the transformation theorem (see \cite[Section 7.5]{CheegerJiangNaber}), $v: B_{sr/2}(q') \to \R^{n}$ is biH\"older with its image and $B_{sr/4}(v(q'))\subset v(B_{sr/2}(q'))$. In particular, if $\eps \le \eps_0(\eps',n)$ we have that $0^n\in v(B_{sr/2}(q'))$ since $|v(q')|\le \eps$ and $s=s(\eps',n)>10\eps$. In particular, there exists $q''\in B_{sr/2}(q')$ such that $v(q'')=0$.
\end{proof}

\medskip

We proceed with the proof of \autoref{prop:topsurf} (ii).
Let us show the statement for $k=0$. Let $r\le c(n,v)$, $q\in \{u=0\}\cap B_{1/2}(p)$, and $x,y\in \{u=0\}\cap B_r(q)$.
We claim that we can find a finite sequence $x = x_0 ,x_1,\dots,x_5 =y$ in $\{u=0\} \cap B_{3r/2}(q)$
		such that $\dist(x_j,x_{j+1})<\frac{r}{2}$.
  
        If $\delta \le \delta_0(\eps,n)$, $B_r(q)$ is $\eps r$-GH close to $B_r(0^{n-2}, z_0)$, where $(0^{n-2},z_0)\in \R^{n-2} \times Z^2$ and $Z$ is an Alexandrov surface with nonnegative curvature. 
        Moreover, by \autoref{lemma:level GH}, $\{u=0\}\cap B_r(q)$ is $\eps r$-close to $B_r(z_0)$, and $x,y\in B_r(q)$ are $\eps r$-GH close to points $(0^{n-2}, \tilde x)$, $(0^{n-2}, \tilde y)\in \R^{n-2}\times Z$.     
        We can find a finite sequence $\tilde x = \tilde x_0, \tilde x_1, \ldots, \tilde x_5 = \tilde y$ in $Z$ such that $\dist_Z(\tilde x_j, \tilde x_{j+1})\le r/4$. We then define $x = x_0 ,x_1,\dots,x_5 =y$ in $\{u=0\}\cap B_{3r/2}(q)$ so that $x_i$ is $\eps r$-GH close to $\tilde x_i$. It is immediate to see that $\dist(x_j,x_{j+1})\le r/2$ if $\eps < \frac18$.

		Iterating, replacing $q$ with each $x_j$, we can find a chain of $5^2+1$ points of pairwise distances at most $s/4$, and so on.
		Repeating this procedure infinitely many times, we obtain a map $\gamma$ defined on the set
		$$[0,1]\cap\bigcup_{m\in\N}5^{-m}\Z$$
		(specifically, $\gamma(j/5)=x_j$, while on multiples of $5^{-2}$ we use the finer chain, and so on). Since $\gamma$ is H\"older continuous, it extends to a map defined on $[0,1]$. Also, the latter takes values in $\{u=0\} \cap B_{4r}(q)$.

    \medskip

		Let us now turn to $k=1$. Given a loop in $B_r(q)\cap \{u=0\}$, which we view as a function $\alpha$ defined on the boundary $\de[0,1]^2$ of the square, we claim that we can find $\ell\ge2$ and a function
		\begin{equation}
        h:\left\{\frac{j}{\ell}\mid j=0,\dots,\ell\right\}^2\cup\de[0,1]^2\to B_{Cr}(q)\cap \{u=0\}
        \end{equation}
		which agrees with $\alpha$ on $\de[0,1]^2$, such that the image of each square $[\frac{j}{\ell},\frac{j+1}{\ell}]\times[\frac{j'}{\ell},\frac{j'+1}{\ell}]$
		(intersected with the domain of $h$) has diameter less than $\frac{r}{10}$. 
       Arguing as above, we know that 
      $B_r(q)$ is $\eps r$-GH close to $B_r(0^{n-2}, z_0)$, where $(0^{n-2},z_0)\in \R^{n-2} \times Z^2$ and $Z$ is an Alexandrov surface with nonnegative curvature. Moreover, by \autoref{lemma:level GH} again, $\{u=0\}\cap B_r(q)$ is $\eps r$-close to $B_r(z_0)$, and $\alpha$ is uniformly $10\eps r$-GH close to a loop $\beta$ in $Z$, which is contained in $B_{2r}(z_0)$. The loop $\beta$ can be build by dividing $\de[0,1]^2$ in $4m$ small pieces, with $m$ so large that their image (through $\alpha$) has small diameter, selecting a point $\beta(t)\in Z$ close to $\alpha(t)$ for each endpoint $t$, and completing $\beta$ by making it piecewise geodesic.
  
     Since Alexandrov surfaces with nonnegative curvature are locally uniformly contractible (see \autoref{prop:Ale2loccontr}), we can fill $\beta$ with a continuous map $\overline\beta:[0,1]^2\to Z$ supported in $B_{Cr}(z_0)$ for some $C>1$. We now take $\ell$ large, such that
     $\overline\beta([\frac{j}{\ell},\frac{j+1}{\ell}]\times[\frac{j'}{\ell},\frac{j'+1}{\ell}])$ has small diameter for each $0\le j,j'<\ell$.
     We can now define $h$ at each point $(x,y)\in\{\frac{j}{\ell}\mid j=1,\dots,\ell-1\}^2$ by taking a point in $\{u=0\}$ close to $(0^{n-2},\overline\beta(x,y))$, completing the proof of the claim.
		
     Now, by applying the case $k=0$, we can extend $h$ to a function $\overline h$ defined on the full 1-skeleton 
     \begin{equation}
     \bigcup_{0\le j,j'<\ell}\de\left(\left[\frac{j}{\ell},\frac{j+1}{\ell}\right]\times\left[\frac{j'}{\ell},\frac{j'+1}{\ell}\right]\right)\, .
     \end{equation}
		The loop $\alpha$ is thus homotopic to the concatenation of $\ell^2$ loops based at $\alpha(0,0)$, each of which has the form
		$$ \eta_{jj'}*\alpha_{jj'}*i(\eta_{jj'}) $$
		(where $*$ denotes concatenation and $i(\cdot)$ indicates a path traveled backwards),
		where $\eta_{jj'}$ joins $\alpha(0,0)$ to $\alpha(\frac{j}{\ell},\frac{j'}{\ell})$, while $\eta_{jj'}$ is a loop based at $\alpha(\frac{j}{\ell},\frac{j'}{\ell})$
		with diameter (of the image) less than $\frac{r}{2}$.
        We can now iterate and apply exactly the same argument used in \cite{WangRCD}.

     \medskip
  
		Finally, let us deal with the case $k=2$. The set $\{u=0\} \cap B_{r}(q)$ is a proper subset of $\{u=0\}$. Indeed, by \autoref{lemma:level GH}, again $\{u=0\} \cap B_{2r}(q)$ is $\eps r$-GH close to $B_{2r}(z_0)$, where $z_0\in Z$. In particular $\{u=0\}\cap (B_{2r}(q) \setminus B_{r}(q)) \neq \emptyset$.
        Applying the case $k=0$ to join $q$ with another point $q'$ in the latter set,
        we see that $\{u=0\} \cap B_{r}(q)$ is not compact.
  
		Hence, by \autoref{prop:topsurf} (i), $\{u=0\} \cap B_{r}(q)$ is an open surface, implying that its second homology group vanishes. A local version of Hurewicz theorem (see
		\cite[Theorem 0.8.3]{DavermanVenema})
		now implies that the inclusion $\{u=0\}\cap B_{r}(q)\hookrightarrow \{u=0\} \cap B_{Cr}(q)$ induces the trivial map on the second homotopy groups.


\section{Proof of \autoref{thm:noRP2}}

The goal of this section is to prove \autoref{thm:noRP2}, which we restate below for the ease of readability.

\begin{theorem}\label{thm:noRP^2main}
Let $(M_i^n,g_i,p_i)\to (X^n,\dist,\haus^n,p)$ be a noncollapsed Ricci limit space. Assume that $X^n=\setR^{n-3}\times C(Z^2)$ is an $(n-3)$-symmetric cone. Then $(Z^2,\dist_Z)$ is homeomorphic to the $2$-sphere $S^2$.  
\end{theorem}

The strategy for proving \autoref{thm:noRP^2main} will be to approximate the $2$-dimensional cross-section $Z^2$ with level sets of maps $(v_i,u_i):B_2(p_i)\to \setR^{n-2}$, where $v_i:B_2(p_i)\to\setR^{n-3}$ are harmonic almost splitting maps (which approximate the splitting functions from the $\setR^{n-3}$ factor in the limit), $u_i:= \sqrt{b_{p_i}^2-|v_i|^2}$, and $b_{p_i}:B_2(p_i)\to\setR$ are Green-type distances as in Section \ref{sec:slicing}. The results from Section \ref{sec:slicing} show that for a generic good level set of these maps, the ambient Riemannian manifold $M_i$ almost splits a factor $\setR^{n-2}$ in the direction perpendicular to the level set, for all points in the level set at all scales below an effectively quantified one. 
This geometric almost rigidity can be used via \autoref{prop:topsurf} (ii) to show that generic level sets of the map above have uniformly controlled, and hence stable, topology. This will show that for $i$ large enough the generic good level set of the above map is homeomorphic to the limit cross-section, and hence to either $S^2$ or $\mathbb{RP}^2$. However, generic good level sets bound their respective sub-level sets, and hence they have even Euler characteristic. This will be enough to show that $Z^2$ is homeomorphic to $S^2$.

\medskip

The broad outline of the argument is similar to the one devised for the proof of the analogous result in the case of noncollapsing sequences with $|\Ric_i|\to 0$: see \cite[Theorem 5.12]{CheegerNaber15}. However, in \cite{CheegerNaber15} the homeomorphism between the level sets and the limit cross-section follows from smooth elliptic estimates and $C^{1,\alpha}$ convergence relying on the previously proved non-existence of codimension $2$ singularities (see \cite[Theorem 5.2]{CheegerNaber15}). This strategy is unfeasible in the present setting, so that establishing the sought homeomorphism requires some new insights. 

\begin{proof}[Proof of \autoref{thm:noRP^2main}]
Since the limit space is a metric cone, we can assume without loss of generality that  $\Ric_i\ge -1/i$.
If $(M_i^n,g_i,p_i)\to \setR^{n-3}\times C(Z^2)$ with $\Ric_i\ge -1/i$, then there exists a sequence $\eps_i \to 0$ such that for every $i\in\N$ big enough, there are good Green distances $b_{p_i}: B_{100}(p_i) \to \R$ (see \autoref{thm:prop_Green}), and $\eps_i$-splitting maps $v_i: B_{100}(p_i) \to \R^{n-3}$ normalized so that $v_i(p_i)=0$ (see \autoref{thm:splittings}).
We set 
\begin{equation}
u_i: A_i\to \setR\, ,\quad u_i:= \sqrt{b_{p_i}^2 - |v_i|^2}\, , 
\end{equation}
where the annular domain $A_i\subset M_i$ was introduced in \autoref{trans}. To ease notation, we write $w_i:=(v_i,u_i)$.

\medskip

Fix $\eps>0$. If $i\in \N$ is big enough so that $\eps_i \le \eps_i(\eps,n)$, we can find a good level set corresponding to the value $(x_i,y_i)\in B_1(0^{n-3})\times (8,9)$ for $w_i = (v_i,u_i)$, as in \autoref{trans}. This means that for every $q\in \Sigma_i := \{(v_i, u_i)=(x_i,y_i)\}$ and $r\le r_0 = r_0(\eps,n)$ there exists a transformation matrix $L_{q,r}$ such that $L_{q,r}\circ w_i : B_r(q) \to \R^{n-2}$ is a $\eps$-splitting map, and moreover $\Sigma_i\neq\emptyset$.

As a consequence of \autoref{prop:topsurf}, $\Sigma_i$ is a uniformly locally contractible two-dimensional manifold, uniformly also in $i$.
Moreover, by Sard's theorem we can assume without loss of generality that $(x_i,y_i)$ is a regular value for $(v_i,u_i)$, hence $\{u_i \le y_i\} \cap \{v_i = x_i\}$ is a compact $3$-manifold bounded by $\Sigma_i$.

\medskip

To conclude the proof, it is enough to show that (up to extracting a subsequence) there exist $x_\infty\in \R^{n-3}$, $y_\infty\in [8,9]$ such that
\begin{equation}\label{eqz10}
   \Sigma_i \to \Sigma_\infty 
   := 
   \{(x_\infty, z)\in \R^{n-3} \times C(Z)\, : \, \dist_Z(z,0)=y_\infty\} \, ,
   \quad \text{in the GH-topology}\, ,
\end{equation}
where $\Sigma_i$ and $\Sigma_\infty$ are endowed with the restriction of the respective ambient distances.

Indeed, if \eqref{eqz10} holds, then from the uniform local contractibility of the sequence $\Sigma_i$, we deduce that $\Sigma_i$ is eventually homotopically equivalent (and hence homeomorphic since they are surfaces) to $\Sigma_\infty \cong Z$, by \cite{Petersen90}. 
Since $Z$ is an Alexandrov surface with curvature $\ge 1$ by \cite{Ketterer15} and \autoref{thm:RCD2Alex} and it has empty boundary by \cite[Theorem 6.1]{CheegerColding97I}, it is either homeomorphic to $S^2$ or $\mathbb{RP}^2$, by \autoref{cor:2dK>0}. However, only the first possibility can occur since $\Sigma_i$ is the boundary of the three-manifold $\{u_i\le y_i\}\cap \{v_i = x_i\}$ and hence it must have even Euler characteristic.

\medskip

We now pass to the proof of \eqref{eqz10}. Up to extracting a subsequence, we assume that $x_i \to x_\infty \in \bar B_1(0^{n-3})$, $y_i \to y_\infty \in [8,9]$.

Notice that $w_i : A_i\subset M_i \to \R^{n-2}$ converges uniformly to the mapping
\begin{equation}
 w_\infty : \R^{n-3} \times C(Z) \to \R^{n-2}\, ,
 \quad
 (x,r,z) \mapsto (x, r)
\end{equation}
(up to a possibly different identification of the limit space with $\R^{n-3} \times C(Z)$),
whose level set $\{w_\infty = (x_\infty, y_\infty)\}$ coincides with $\Sigma_\infty$.

In particular, denoting by $\Psi_i : B_{100}(p_i) \to B_{100}(0^{n-2},0) \subset \R^{n-2}\times C(Z)$ an $\eps_i$-GH isometry, it is clear that $\Psi_i(\Sigma_i) \subset B_{\eps'}( \Sigma_\infty)$, provided $i\ge i(\eps')$. To conclude the proof of \eqref{eqz10}, we need to show the converse inclusion which easily follows from the following.

\medskip

\emph{Claim:}
Fix $q_\infty \in \Sigma_\infty$. Then there exists $q_i'\in \Sigma_i$ such that $q_i'\to q_\infty$ as $i\to\infty$.

\medskip

We notice that \autoref{trans} applies to $w_\infty: \R^{n-3} \times C(Z) \to \R^{n-2}$; moreover, because of the symmetries, $(x_\infty, y_\infty)$ is automatically a good value, i.e., the corresponding level set is good. In particular, for every $r\le 2r_0(\eps',n)$ there exists a transformation such that $L_{r,q_\infty} \circ w_\infty : B_r(q_\infty) \to \R^{n-2}$ is an $\eps'$-splitting map. If $i \ge i(\eps')$, it is easy to deduce that there exists $q_i\in M_i$ such that the map 
\begin{equation}
    L_{r_0(\eps',n),q_\infty} \circ w_i : B_{r_0(\eps',n)}(q_i) \to \R^{n-2}
    \quad \text{is an $\eps''$-splitting map}
\end{equation}
(with $\eps''\to0$ as $\eps'\to0$)
and $|w_i(q_i) - w_\infty(q_\infty)|\to0$. In particular, $|w_i(q_i) - (x_i,y_i)|\to0$.

We are in the same situation as in \autoref{lemma:level GH} (ii): arguing in the same way we can find $q_i'\in B_{10\eps''}(q_i)$ such that $w_i(q_i)=(x_i,y_i)$, thus completing the proof of the claim.
\end{proof}

\section{Manifold recognition: setup and preliminaries}\label{sec:genmanprel}

This is the first of three sections whose aim is to prove \autoref{thm:RCDtopma}, which we restate below for the reader's ease.

\begin{theorem}\label{thm:RCDtopmain}
Let $(X,\dist,\haus^3)$ be an $\RCD(-2,3)$ space. Then $(X,\dist)$ is a topological $3$-manifold without boundary if and only if all the cross-sections of all tangent cones of $X$ at any point are homeomorphic to $S^2$.
\end{theorem}

\begin{remark}
The assumptions of \autoref{thm:RCDtopmain} are equivalent to the requirement that all tangent cones of $(X,\dist)$ are homeomorphic to $\setR^3$, since the cross-sections are homeomorphic to either $S^2$ or $\mathbb{RP}^2$. Moreover, it is sufficient to assume that one tangent cone at each point is homeomorphic to $\setR^3$, as this implies that each tangent cone is homeomorphic to $\setR^3$.
Indeed, by \autoref{lemma:crossconn}, the set of cross-sections at a given point is connected;
by \autoref{prop:Ale2loccontr} and \cite{Petersen90}, it follows that they are homeomorphic.
\end{remark}

The proof of the implication from topological regularity to the homeomorphism of tangent cones with $\setR^3$ will be completed in this section: see \autoref{cor:impl1} in particular.
In the proof of the converse implication, from the structure of tangent cones to the manifold regularity, there are two main steps. 
The goal of Section \ref{sec:genman2} will be to prove that any $(X,\dist,\haus^3)$ as in the statement of \autoref{thm:RCDtopmain} is a generalized $3$-manifold with empty boundary, according to \autoref{def:genmanifold}.

\begin{proposition}\label{prop:genmanrec}
Let $(X,\dist,\haus^3)$ be an $\RCD(-2,3)$ space. Assume that all tangent cones at every point in $X$ have cross-section homeomorphic to $S^2$. Then $X$ is a generalized $3$-manifold (without boundary). 
\end{proposition}

In Section \ref{sec:topman} we are going to upgrade the conclusion of \autoref{prop:genmanrec} from generalized manifold to topological manifold and hence to complete the proof of \autoref{thm:RCDtopmain}. This second step will rely on Perelman's resolution of the Poincar\'e conjecture and the recognition theory for $3$-manifolds among generalized $3$-manifolds: see \cite{Thickstuna,Thickstunb,DavermanRepovs} and the discussion in Section \ref{subsec:recogn}. 

\medskip

The key to proving \autoref{thm:RCDtopmain} will be to further refine our understanding of the topology of Green-balls and Green-spheres at sufficiently conical scales in an $\RCD(-2,3)$ space $(X,\dist,\haus^3)$. Below we collect into \autoref{main.green.sphere} the relevant properties about Green-balls and Green-spheres that were obtained in the previous sections, specialized to the case of $3$-dimensional $\RCD$ spaces with empty boundary. This will provide us with the terminology for discussing the proof strategy of \autoref{prop:genmanrec}.

\medskip

Recall that an $\RCD(-(n-1),n)$ space $(X,\dist,\haus^n)$ is said to have empty boundary provided it has no tangent cone isometric to the half-space $\setR^n_{+}$ \cite{DePhilippisGigli2}. Moreover, by \cite[Theorem 1.6]{BrueNaberSemolabdry} noncollapsed (pointed) GH-limits of sequences of $\RCD(-(n-1),n)$ spaces with empty boundary have empty boundary.
In the setting of $\RCD(-2,3)$ spaces with empty boundary, we can fix $\eps_0=\eps_0(v)>0$,
where $v>0$ is the volume noncollapsing parameter such that
\begin{equation}
    \haus^3(B_1(p)) \ge v \, , \quad \text{for every $p\in X$}\, ,
\end{equation}
in such a way that $X\setminus \mathcal{S}_{\eps_0}^1$ is an open set where the metric Reifenberg theorem applies; in particular, $X\setminus \mathcal{S}_{\eps_0}^1$ is locally homeomorphic to $\R^3$: see \cite{KapovitchMondino} after \cite{CheegerColding97I}.

\begin{proposition}
\label{main.green.sphere}
		
        For every $\eta>0$, if $\delta \le \delta_0(\eta,v)$ the following statement holds.
        Let $(X,\dist,\haus^3)$ be an $\RCD(-\delta,3)$ space with empty boundary such that $\haus^3(B_1(p))\ge v$ for every $p\in X$.
        Assume that
		\begin{equation}
		\dist_{\rm{GH}}\left(B_{20}(p),B_{20}(o)\right)\le \delta  \, ,
        \quad o \in C(Z)\, , \quad p\in X \, ,
		\end{equation}
		where $(Z,\dist_Z)$ is a $2$-dimensional Alexandrov surface with curvature $\ge 1$. Then there exists a good Green distance $b_p: B_2(p)\to \R$ and
		a Borel set of radii $r\in(1/4,1)$ of measure at least $(1-\eta)\frac34$ such that
		the topological boundary $\mathbb{S}_r(p)$ of the sub-level set $\mathbb{B}_r(p):=\{b_p<r\}$ satisfies the following properties:
		\begin{itemize}
			\item[(i)] $\mathbb{S}_r(p) = \{b_p=r\}$;
			\item[(ii)] up to rescaling by $r$, $\mathbb{S}_r(p)$ is $\eta$-GH-close to $(Z,\tilde{\dist}_Z)$, where $\tilde{\dist}_Z:=2\sin(\dist_Z/2)$;
			\item[(iii)] $\mathbb{S}_r(p)\cap \mathcal{S}_{\eps_0}^1$ is a finite set;
			\item[(iv)] $\mathbb{S}_r(p)$ is a closed topological surface with empty boundary;
			\item[(v)] there exist $\rho_0=\rho_0(v)\ge 0$ and $\overline C = \overline C(v)$ such that, whenever $q\in\mathbb{S}_r(p)$ and $0< s\le\rho_0 r$, $B_s(q)\cap \mathbb{S}_r(p)$ is 2-connected in $B_{\overline C s}(q)\cap\mathbb{S}_r(p)$.
		Explicitly, for $k\in\{0,1,2\}$, any continuous map $S^{k}\to B_s(q)\cap\mathbb{S}_r(p)$ can be extended to a continuous map
		$\overline D^{k+1}\to B_{\overline C s}(q)\cap\mathbb{S}_r(p)$;
		\item[(vi)] $\mathbb{S}_r(p)$ is homeomorphic to $Z$. Hence they are both either homeomorphic to $S^2$ or to $\mathbb{RP}^2$.
		\end{itemize}
\end{proposition}

	\begin{remark}
		In \autoref{main.green.sphere} (ii), $\mathbb{S}_r(p)$ is endowed with the restriction of the ambient distance from $X$. In this regard, we note that $\tilde{\dist}_Z$ is the restriction of the cone distance on $C(Z)$ to the sphere of radius $1$ centered at the tip.
	\end{remark}

\begin{remark}
Any Alexandrov surface $(Z,\dist_Z)$ appearing in \autoref{main.green.sphere} has empty boundary by \cite[Theorem 1.6]{BrueNaberSemolabdry}.  
\end{remark}

\begin{proof}[Proof of \autoref{main.green.sphere}]

The first part of the statement follows from the results of Section \ref{sec:slicing}.

\medskip

We begin by proving (i). We need to show that $\overline{\{b_p < r\}} = \{b_p\le r\}$. The inclusion $\overline{\{b_p < r\}} \subseteq \{b_p\le r\}$ follows from the continuity of the Green distance. To show the converse, we pick $q\in \mathbb{S}_r(p)$ and $s\le c(\eps,n, v)r$ sufficiently small so that \autoref{trans n=3} applies. Then we can rescale $b_p$ to obtain an $\eps$-splitting map $\hat b: B_s(q) \to \R$. It is immediate to check that the image $\hat b(B_s(q))$ is $\frac{s}{10}$-dense in $(\hat b(q)-s , \hat b(q) + s)$, provided $\eps$ is small enough. Hence there exists $q'\in B_s(q)$ such that $\hat b(q') < \hat b(q)$, and therefore $b(q') < r$.

\medskip

Conclusion (ii) follows by observing that $\mathbb S_r(p)$ and the boundary of the ball $B_r(p)$ are $\eta$-close in the Hausdorff topology, and that the latter is $\eta$-GH close to the corresponding sphere in $C(Z)$.

To show the first observation we use the inequality
\begin{equation}\label{eqz13}
    |b_p - \dist_p|\le \eta\, ,
    \quad
    \text{if $\delta \le \delta_0(\eta,v)$}\, ,
\end{equation}
(see \autoref{thm:prop_Green} for its proof).
Given $q\in X$ such that $\dist(p,q) = r \in (1/4,1)$, it is enough to show the existence of $q''\in \mathbb S_r(p)\cap B_{10\eta}(q)$. Since $B_{20}(p)$ is GH-close to a cone,
we can find $q'\in B_{3\eta}(q)$ such that $\dist(p,q')\ge r + 2\eta$, and hence
 $r+\eta \le b_p(q') \le r + 3\eta$ as a consequence of \eqref{eqz13}. We then connect $q'$ to $p$ by a minimizing geodesic $\gamma$, and we find an intermediate time such that $b_p(\gamma(t)) = r$. We set $q''=\gamma(t)$ and we rely once more on \eqref{eqz13}.
Similarly, any $q\in\mathbb{S}_r(p)$ is $\eta$-close to $\de B_r(p)$.
\medskip

Conclusions (iii), (iv), (v) follow from \autoref{lemma:finiteint}, \autoref{prop:topsurf} (i), and \autoref{prop:topsurf} (ii), respectively.

\medskip

The homeomorphism of good level sets with the cross-section $Z$ follows from \cite{Petersen90} by (ii), (iv), and (v) arguing as in the proof of \autoref{thm:noRP^2main}. 
\end{proof}

\begin{definition}
Fix $0 <\delta_0, \eta_0 \le 1/10$ and $v>0$.
Let $(X,\dist,\haus^3)$ be an $\RCD(-2,3)$ space with empty boundary and such that $\haus^3(B_1(p))\ge v>0$ for any $p\in X$. We let $\mathcal{G}_p'$ be the set of radii $r\in(0, \delta_0^2)$ such that the following holds:
\begin{itemize}
    \item[(i)] $B_{s}(p)$ is $\delta_0$-conical, for all $s\in[\delta_0r,r/\delta_0]$;

    \item[(ii)] there exists a good Green distance $b:B_{3r}(p)\to (0,\infty)$, possibly depending on $r$, such that $\{b=r\} = \mathbb{S}_r(p)$ satisfies the conclusions in \autoref{main.green.sphere} with $\eta = \eta_0$.
\end{itemize}
\end{definition}

Building upon \autoref{trans n=3}, together with \autoref{lemma:manyconsc} and \autoref{lemma:infveryc}, it is immediate to prove the following.

\begin{lemma}\label{cor:densegoodr}
    If $\eta_0 \le 1/10$, $\delta_0 \le \delta_0(v)$,
    there exists a measurable set $\mathcal{G}_p \subset \mathcal{G}_p'$ and $C = C(\eta_0,\delta_0)>1$ such that the following holds:
    \begin{itemize}
        \item[(i)] $(r,Cr)\cap \mathcal{G}_p \neq \emptyset$ for every $r\in (0,\delta_0^2/C)$;

        \item[(ii)] $\mathcal{L}^1(\mathcal{G}_p)>0$ and every $r\in \mathcal{G}_p$ has density one, i.e.,
        \begin{equation}
             \lim_{s\to 0} \frac{\mathcal{L}^1(\mathcal{G}_p\cap (r-s,r + s))}{2s} = 1 \, , \quad \text{for every $r\in \mathcal{G}_p$}\, ;
        \end{equation}

        \item[(iii)] $r=0$ is a point of density one, i.e.,
        \begin{equation}
            \lim_{s\to 0} \frac{\mathcal{L}^1(\mathcal{G}_p\cap (0,s))}{s} = 1 \, .
        \end{equation}
    \end{itemize}
\end{lemma}

We will refer to Green-spheres $\mathbb{S}_r(p)$ and Green-balls $\mathbb{B}_r(p)$ with $r\in \mathcal{G}_p$ as \emph{good} Green-spheres and Green-balls respectively.

	\begin{remark}
		Even if $\delta_0(v), \eta_0(v)>0$ should be thought of as fixed, we will later require extra smallness assumptions on $\delta_0$ and $\eta_0$ (depending only on the noncollapsing constant $v$),
		to guarantee further properties of Green-type balls for good radii $r\in\mathcal{G}_p$. For the sake of clarity, we will sometimes write $\mathcal{G}_p(\delta_0,\eta_0)$ rather than just $\mathcal{G}_p$, to emphasize its dependence on a specific choice of $\delta_0$ and $\eta_0$.
	\end{remark}

The key geometric insight for the proof of \autoref{prop:genmanrec} is that good Green-balls such that the boundary Green-sphere is homeomorphic to $S^2$ are contractible, and moreover there are plenty of them under the assumption that all tangent cones have cross-section homeomorphic to $S^2$. Further, for sufficiently small radii, the punctured Green-ball deformation retracts onto the Green-sphere. The combination of these two properties will allow us to prove local contractibility and that the local relative homology of $X$ is the same as for $\setR^3$. 

Note that for those good Green-balls that are contained in the manifold part of $X$ the first conclusion above would follow quite easily as soon as we prove that they are simply connected manifolds with boundary, the boundary being the Green-sphere: see \autoref{contractible} for the details. This will be an important step in the proof of \autoref{prop:genmanrec} and requires some new ideas already: see \autoref{rm:Alexander} for a discussion about the main challenge in proving this claim. However, establishing the conclusion in general for good Green-balls (without knowing a priori that they are contained in the manifold part) will be considerably more delicate. The argument in this case will occupy most of Section \ref{sec:genman2}.
\medskip

In the present section we are going to prove several conditional statements that will be instrumental later when fully establishing the above claims. 
\begin{itemize}
\item In Section \ref{subsec:goodretra} we exploit the uniform local contractibility of good Green-spheres to prove the existence of a retraction from a definite size neighbourhood of the Green-ball onto the Green-ball. 
\item In Section \ref{subsec:locsc} we prove that good Green-balls whose boundary Green-spheres are homeomorphic to $S^2$ are simply connected. 
\item In Section \ref{subsec:topGgm} we show that the good Green-balls contained in the manifold part are contractible manifolds with boundary homeomorphic to $S^2$; analogously, good Green-balls contained in the generalized manifold part of an $\RCD(-2,3)$ space $(X,\dist)$ are contractible generalized manifolds with boundary homeomorphic to $S^2$. 
\item In Section \ref{subsec:partcontr} we will prove that at a sufficiently conical scale a ball can be deformation retracted onto a small neighbourhood of the quantitative singular set $\mathcal{S}^1_{\eps_0}$. 
\end{itemize}
All of these results will be key for the proof of \autoref{prop:genmanrec} which is deferred to Section \ref{sec:genman2}. Moreover, all the statements for Green-balls contained in the manifold part will become true unconditionally once we have completed the proof of \autoref{thm:RCDtopmain}. These properties will be fundamental also later in Section \ref{sec:unifcontr} when we will prove uniform local contractibility for $\RCD(-2,3)$ manifolds.

\medskip

Below, we will often drop the term \emph{good} in order not to burden the presentation. It is understood that all the Green-spheres and Green-balls considered will be good unless otherwise stated.

\subsection{Good retractions onto Green-spheres}\label{subsec:goodretra}

Our first aim is to exploit the local uniform contractibility of the good Green-spheres to construct a well-behaved retraction from a uniform neighbourhood of a good Green-ball onto the good Green-ball. The idea is classical: see \cite[Propostion 7.1]{BassoMartiWenger} for a recent application in the metric setting.

\begin{lemma}\label{retraction}
Let $(X,\dist,\haus^3)$ be an $\RCD(-2,3)$ space with empty boundary such that $\haus^3(B_1(p))\ge v>0$ for all $p\in X$.
If $\delta_0 \le \delta_0(v)$ and $\eta_0\le \eta_0(v)$ there exists $\tau=\tau(v)>0$ such that if $r\in\mathcal{G}_p$ then $\overline{\mathbb{B}}_{(1+\tau)r}(p)$ retracts onto $\overline{\mathbb{B}}_s(p)$ for any $s\in\mathcal{G}_p\cap[r-\tau,r]$ (thus for $s=r$ and certain radii $s<r$ arbitrarily close to $r$), with a retraction map $\rho:\overline{\mathbb{B}}_{(1+\tau)r}(p)\to \overline{\mathbb{B}}_s(p)$ such that 
\begin{equation}\label{eq:contrretra}
\dist(\rho(x),y)\le C(v)\dist(x,y)\, ,
\end{equation}
for all $x\in \overline{\mathbb{B}}_{(1+\tau)r}(p)$ and $y\in \overline{\mathbb{B}}_s(p)$.
\end{lemma}

\begin{proof}
The proof will be just sketched since it is based on a simple modification of the arguments in \cite[Section 3]{Petersen90}. 

 \medskip
	
		Fix $\tau>0$ small. We consider a $(2\tau r)$-neighbourhood $U$ of $\overline{\mathbb{B}}_r(p)$
		and let $V:=U\setminus \overline{\mathbb{B}}_r(p)$. 
		Consider a countable collection of open sets $(V_a)_{a\in A}$ covering $V$, with $V_a\subseteq V$, $\operatorname{diam}(V_a)<\tau r$
		and such that $\operatorname{diam}(V_a)<\dist(V_a,\overline{\mathbb{B}}_r(p))$ (in other words, the sets $V_a$ are smaller and smaller as we approach $\mathbb{S}_r(p)$).
		By \autoref{lemma:dimcov}, up to pass to a refinement, we can assume that no point of $V$ belongs to $V_a$ for more than four different indices $a$. For each $a\in A$, we select $p_a\in V_a$.

        \medskip
		
		Let $K$ be the nerve of the open cover $(V_a)_{a\in A}$, which we view as an abstract simplicial complex of dimension at most 3.
		We identify $A$ with the 0-skeleton of $K$.
		As in \cite{Petersen90},
		using a partition of unity we obtain a continuous map $\varphi:V\to K$.
		On the other hand, we can construct a map $f:K\to\mathbb{S}_r(p)$ inductively on the skeleta: for each vertex $a\in A$,
		we choose a point $q_a\in \mathbb{S}_r(p)$ minimizing distance from $p_a$, so that $\dist(p_a,q_a)<2\tau r$ (as $B_{2r}(p)$ is almost conical), and we let $f(a):=q_a$, thus defining $f$ on the 0-skeleton.
		
		If $a,a'\in A$ are adjacent in $K$ (i.e., $V_a\cap V_{a'}\neq\emptyset$)
		then $\dist(f(a),f(a'))<6\tau r$. In this case \autoref{main.green.sphere} (v) guarantees that we can find a small path joining them in $\mathbb{S}_r(p)$, provided $6\tau<\rho_0(v)$.
		Hence, on the edge connecting $a$ and $a'$ we can define $f$ to be this path, completing the extension of $f$ to the 1-skeleton of $K$.
		Finally, for each 2-simplex $\Delta$ in $K$, the function $f$ is already defined on $\de\Delta\cong S^1$ (where it equals a small loop),
		and \autoref{main.green.sphere} (v) again gives an extension to $\Delta$, completing the extension to the 2-skeleton of $K$; similarly, we arrive at a map defined on the 3-skeleton, which is all of $K$.

        \medskip

		The composition $f\circ\varphi:V\to\mathbb{S}_r(p)$ is a continuous map. Also, as $b_p(x)\to r$ we have $\dist(x,f\circ\varphi(x))\to0$:
		indeed, $\dist(x,p_a)$ and $\dist(p_a,q_a)$ (for $a\in A$ such that $x\in V_a$) become smaller and smaller; calling $\Delta$ a simplex in $K$ containing $\varphi(x)$ in its interior, we have $x\in V_a$ for each vertex $a$ of $\Delta$, which clearly gives the claim as $f(\Delta)$ has vanishing diameter. We can then extend $f\circ\varphi$
		to a retraction $U\to\overline{\mathbb{B}}_r(p)$, by taking the identity map on $\overline{\mathbb{B}}_r(p)$.

        \medskip
  
		Finally, if $\eta_0\le \eta_0(v,\tau)$ and $\delta_0 \le \delta_0(v,\tau,\eta_0)$ in the definition of $\mathcal{G}_p$, we can assume that $\overline{\mathbb{B}}_{(1+\tau)r}(p)$ is included in $U$, i.e., the $(2\tau r)$-neighbourhood of $\overline{\mathbb{B}}_r(p)$.
        Moreover, we can take any $s\in\mathcal{G}_p\cap(0,r)$ close enough to $r$,
		replace $\overline{\mathbb{B}}_r(p)$ with $\overline{\mathbb{B}}_s(p)$ and repeat the argument above. The estimate \eqref{eq:contrretra} is obtained arguing as in the proof of \cite[Proposition 7.1]{BassoMartiWenger}.
	\end{proof}
	
     In the sequel, we will apply the following consequence a few times, for constants $\Lambda=\Lambda(v)$ depending only on the noncollapsing constant $v$.

	\begin{corollary}\label{lambda}
		For every $\Lambda>1$, if $\eta_0 \le \eta_0(v)$ and $\delta_0\le \delta_0(\Lambda,v)$ in the definition of $\mathcal{G}_p$, whenever $r\in\mathcal{G}_p\cap(0,\delta_0^2/\Lambda)$ there exists a retraction
		$\overline{\mathbb{B}}_{\Lambda r}(p)\to\overline{\mathbb{B}}_s(p)$, for suitable radii $s<r$ arbitrarily close to $r$, as well as $s=r$.
		In particular, if $\mathbb{B}_r(p)$ is $k$-connected in $B_{\Lambda r}(p)$, then $\mathbb{B}_r(p)$ is $k$-connected.
	\end{corollary}
	
	\begin{proof}
        Let $\delta_0 = \delta_0(v)$, $\eta_0=\eta_0(v)$, and $\tau = \tau(v)$, be as in \autoref{retraction}.
		Let $R:=\sqrt{1+\tau}$ and $J\ge 1$ such that $R^J\in[\Lambda,2\Lambda)$. If $\delta = \delta(\eta_0,\delta_0,\Lambda)$ and 
		$B_{s}(p)$ is $\delta$-conical for all $s\in(\delta r,r/\delta)$ then
		all intervals $(R^jr,R^{j+1}r)$ intersect $\mathcal{G}_p=\mathcal{G}_p(\delta_0,\eta_0)$, for $j=0,\dots,J-1$. 
		Letting $r_j\in(R^jr,R^{j+1}r)\cap\mathcal{G}_p$,
		since $r_{j+1}\le(1+\tau)r_j$ we obtain from \autoref{retraction} a retraction
		\begin{equation} 
		 \overline{\mathbb{B}}_{r_{j+1}}(p)\to\overline{\mathbb{B}}_{r_j}(p)
		 \end{equation}
		for each $j$, together with another one $\overline{\mathbb{B}}_{r}(p)\to\overline{\mathbb{B}}_s(p)$ (for $s<r$ arbitrarily close to $r$ or $s=r$).
		The composition of these $J+1$ maps is the desired retraction. Thus, it suffices to replace $\delta_0$ with $\delta=\delta(\eta_0,\delta_0,\Lambda)$ in the definition of $\mathcal{G}_p$
		(i.e., to replace $\mathcal{G}_p(\delta_0,\eta_0)$ with $\mathcal{G}_p(\delta,\eta_0)$).
  
		\medskip
		
		The second part of the statement easily follows from the first one, noting that by the very definition of $\mathcal{G}_p$ the good open Green-ball $\mathbb{B}_r$ admits an exhaustion into closed good Green balls $\overline{\mathbb{B}}_{r_i}(p)$ with $r_i\uparrow r$ as $i\to\infty$.
	\end{proof}

\subsection{Local simple-connectedness}\label{subsec:locsc}

The goal of this section is to prove that those Green-balls whose Green-spheres are homeomorphic to $S^2$ are simply connected: see \autoref{simply.conn.clean} for the precise statement. This statement will be a key ingredient later to prove that $\RCD(-2,3)$ spaces $(X,\dist,\haus^3)$ such that no cross-section of a tangent cone is homeomorphic to $\mathbb{RP}^2$ are locally uniformly simply connected. 

We address the reader to \cite{MondinoWei,WangRCD} for the relevant background about universal covers of $\RCD(K,N)$ spaces.

In order to illustrate the idea of the proof by avoiding some of the technicalities, we start by proving a \emph{global} version of \autoref{simply.conn.clean}.

\begin{lemma}\label{lemma:avrS2}
Let $(Y,\dist,\haus^3)$ be an $\RCD(0,3)$ space with Euclidean volume growth and such that the cross-section of each blow-down is homeomorphic to $S^2$. Then $Y$ is simply connected. 
\end{lemma}

\begin{proof}
We argue by contradiction. Assume that $\pi_1(Y)$ is not trivial and fix a base-point $p\in Y$. Let $(\tilde{Y},\dist_{\tilde{Y}},\haus^3)$ be the universal cover of $(Y,\dist_Y,\haus^3)$ and let $\pi:\tilde{Y}\to Y$ be the covering map.  By the Euclidean volume growth assumption, $\pi_1(Y)$ is finite: see \cite[Theorem 1.6]{MondinoWei} generalizing the previous \cite{Li86,Andersonshort} and note that the revised fundamental group coincides with the usual fundamental group by the semi-local simple-connectedness of $\RCD(K,N)$ spaces proved in \cite{WangRCD}. Let $N>1$ denote the cardinality of $\pi_1(Y)$. 

Let $G:Y\setminus\{p\}\to (0,\infty)$ be the unique Green function of the Laplacian with pole at $p$ and decaying to $0$ at infinity. Since the cross-section of each blow-down of $(Y,\dist_Y,\haus^3,p)$ is homeomorphic to $S^2$, by \autoref{main.green.sphere} there exists a sequence $r_i\downarrow 0$ such that the Green-spheres $\{G=r_i\}\subset Y$ are good in the usual sense and homeomorphic to $S^2$.

Let $\tilde{p}_1\in \tilde{Y}$ be such that $\pi(\tilde{p}_1)=p$ and note that the orbit of $\tilde{p}_1$ under the action of $\pi_1(Y)$ by deck transformations on $\tilde{Y}$ is finite with $N$ elements $\{\tilde{p}_1,\dots,\tilde{p}_N\}$. Let $\tilde{G}:=G\circ\pi:\tilde{Y}\setminus\{\tilde{p}_1,\dots,\tilde{p}_N\}\to (0,\infty)$. By a minor variation of the proof of \autoref{main.green.sphere} we can assume that the level sets $\{\tilde{G}=r_i\}$ are good in the usual sense also for $\tilde{G}$.
Note that $\tilde G$ looks like $\dist_{\tilde p_1}$ at large scales.

In particular, these level sets are either homeomorphic to $S^2$ or to $\mathbb{RP}^2$, depending on the topology of the cross-section of the blow-downs of $\tilde{Y}$. We claim that both cases lead to a contradiction. 

Indeed, for each such good level $r_i$, it is elementary to check that $\pi$ restricts to a covering map $\pi_{r_i}:\{\tilde{G}=r_i\}\to \{G=r_i\}$ with covering degree equal to $N$. However, if $N>1$ there is no such covering map 
when $\{\tilde{G}=r_i\}$ is connected.
\end{proof}

\begin{remark}
In combination with \autoref{thm:noRP2}, applied to blow-downs, \autoref{lemma:avrS2} shows that complete $3$-manifolds with $\Ric\ge 0$ and Euclidean volume growth are simply connected, a statement originally proved with a completely different argument in \cite{Zhu93}.
\end{remark}

\begin{remark}
In dimensions $\ge 4$ there exist complete manifolds with $\Ric\ge 0$ and Euclidean volume growth that are not simply connected. The examples can be also taken to be Ricci flat: see for instance \cite{Andersonsurvey}.
\end{remark}

 Below we state and prove an effective version of \autoref{lemma:avrS2}.

 	\begin{proposition}\label{simply.conn.clean}
		If $\eta_0 \le \eta_0(v)$ and $\delta_0 \le \delta_0(v)$, for any $r\in \mathcal{G}_p$
		such that the Green-sphere $\mathbb{S}_r(p)$ is homeomorphic to $S^2$, the Green ball $\mathbb{B}_r(p)$ and its closure $\overline{\mathbb{B}}_r(p)$ are simply connected.
	\end{proposition}

The main step of the proof of \autoref{simply.conn.clean} is to show that (scale invariantly) sufficiently small loops in a good Green-ball with Green-sphere homeomorphic to $S^2$ are contractible inside the Green-ball, as follows.

\begin{lemma}\label{lemma:Gballsc}
Let $v>0$ be fixed. There exists $\eps_0=\eps_0(v)$ such that if $\eta_0 \le \eta_0(v)$ and $\delta_0 \le \delta_0(v)$ the following statement holds. For any $\RCD(-2,3)$ space with empty boundary $(X,\dist,\haus^3)$ and for any $p\in X$ such that $\haus^3(B_1(p))\ge v$, for all $r\in \mathcal{G}_p$ such that the Green-sphere $\mathbb{S}_r(p)$ is homeomorphic to $S^2$ the inclusion
\begin{equation}
B_{\eps_0 r}(p)\hookrightarrow \mathbb{B}_r(p)
\end{equation}
induces the trivial map $\pi_1(B_{\eps_0 r}(p),p)\to\pi_1(\mathbb{B}_{r}(p),p)$.
\end{lemma}

\begin{proof}
The argument will be an effective version of the one employed for proving \autoref{lemma:avrS2}. 

We argue by contradiction and assume that no such $\eps_0>0$ exists. Then there exists a sequence of $\RCD(-2,3)$ spaces $(X_i,\dist_i,\haus^3,p_i)$ such that $\haus^3(B_1(p_i))\ge v>0$, $r_i\in \mathcal{G}_{p_i}$ is such that $\mathbb{S}_{r_i}(p_i)$ is homeomorphic to $S^2$ and there exist loops $\gamma_i\subset B_{r_i/i}(p_i)$ that are not contractible in $\mathbb{B}_{r_i}(p_i)$.

Up to scaling the distances $\dist_i$, we can assume that $r_i=1$ for all $i\in\setN$, $i\ge 1$. We consider good local Green functions $G_{i}:B_{2}(p_i)\to (0,\infty)$ and the associated Green-type distances $b_{i}$ inducing good Green-balls and spheres as in \autoref{main.green.sphere} for all $i\in\setN$.
While in principle we are changing the Green distance, since good Green-spheres $\mathbb{S}_r(p_i)$ are homeomorphic to the cross-section of the cone
approximating $B_{2r}(p_i)$ we still have $\mathbb{S}_r(p_i)\cong S^2$.

Then we argue as in the proof of \cite[Lemma 3.3]{PanWei} and consider the universal coverings $(U_i,\tilde{p}_i)$ of $(\mathbb{B}_{1}(p_i),p_i)$ endowed with the covering groups $H_i:=\pi_1(\mathbb{B}_{1}(p_i),p_i)$ and projection maps $\pi_i:U_i\to \mathbb{B}_{1}(p_i)$. Let $\Gamma_i<H_i$ be the subgroup generated by the loop $\gamma_i$. A minor variant of the argument in \cite{Andersonshort} shows that $\Gamma_i$ is a finite group with cardinality bounded by $N=N(v)$ independently of $i\in\setN$. Since $\gamma_i\subset B_{1/i}(p_i)$ and it has order $\le N(v)$, we obtain
\begin{equation}\label{eq:orbitto0}
\mathrm{diam} (\Gamma_i\cdot \tilde{p}_i)\to 0\, ,\quad \text{as $i\to\infty$}\, .
\end{equation}
We lift the Green-type distances $b_{i}$ to the universal covers $U_i$ by setting $\tilde{b}_i:=b_i\circ\pi_i:U_i\to [0,\infty)$. By \eqref{eq:orbitto0}, for $i$ sufficiently large the functions $\tilde{b}_i$ have properties analogous to those of the Green-type distances $b_i$, namely they satisfy the conclusions of \autoref{main.green.sphere}.
Hence we can find good radii $s_i\in \mathcal{G}_{p_i}$ such that $\mathbb{S}_{s_i}(p_i)$ is homeomorphic to $S^2$ and $\{\tilde{b}_i=s_i\}$ is either homeomorphic to $S^2$ or to $\mathbb{RP}^2$. Arguing as in the last part of the proof of \autoref{lemma:avrS2}, if $\Gamma_i\neq \{e\}$ we reach a contradiction in both cases. 
\end{proof}

	\begin{proof}[Proof of \autoref{simply.conn.clean}]
		If $\eta_0\le \eta_0(v)$ and $\delta_0\le \delta_0(v)$ are sufficiently small, then $\overline{\mathbb{B}}_r(p)\subseteq B_{2r}(p)$ and, setting $\Lambda:=4\eps_0^{-1}$, $B_{2r}(p)$ is 1-connected in $B_{\Lambda r}(p)$ by \autoref{lemma:Gballsc}.
		Thus, $\overline{\mathbb{B}}_r(p)$ is 1-connected in $B_{\Lambda r}(p)$. The conclusion follows from \autoref{lambda}, provided $\delta_0\le \delta_0(\Lambda,v)$ is small enough.
	\end{proof}

\subsection{Topology of Green-balls in the (generalized) manifold part}\label{subsec:topGgm}

	Our goal in this section is to study the topology of those good Green-balls that do not intersect the non-manifold part of $X$. It will turn out that they are contractible $3$-manifolds with boundary homeomorphic to $S^2$. This conclusion will be achieved in two steps: first, we are going to prove that they are $3$-manifolds with boundary homeomorphic to $S^2$. The contractibility then follows by a classical argument taking into account the simple-connectedness obtained in \autoref{simply.conn.clean}: see \autoref{contractible} for the details. 
Therefore, the main challenge in this section is to prove that good Green-balls are manifolds with boundary. 

\begin{remark}\label{rm:Alexander}
It is well known that for an open set in a topological $3$-manifold $M^3$ with compact closure and topological boundary homeomorphic to $S^2$ it is not necessarily true that the closure is homeomorphic to a $3$-manifold with boundary $S^2$. A classical example is the so-called Alexander horned sphere: see \cite{Cannon} for a discussion very much in the spirit of some of the techniques of the present paper. 
\end{remark}

In the case where the ambient space is a smooth Riemannian manifold, an argument similar to the conclusion of the proof of \cite[Theorem 7.10]{CheegerJiangNaber} shows that any Green-ball of good radius is a smooth submanifold bounded by the respective Green-sphere. In the general case, we will exploit \autoref{retraction}, which in turn is a consequence of the local splitting at all sufficiently small scales at points on the Green-sphere, to prove that Bing's \emph{tameness} criterion from \cite{Bing 1ULC} is satisfied.

\begin{definition}
Let $(X,\dist,\haus^3)$ be an $\RCD(-2,3)$ space. We let $\mathcal{R}_{\rm{top}}(X)=\mathcal{R}_{\rm{top}}\subseteq X$ be the open set of those points $x\in X$ such that there exists a neighbourhood $U_x\ni x$ homeomorphic to $\setR^3$. We will refer to $\mathcal{R}_{\rm{top}}(X)$ as the \emph{manifold set} of $X$.
We shall denote $\mathcal{S}_{\rm{top}}(X)=\mathcal{S}_{\rm{top}}:=X\setminus\mathcal{R}_{\rm{top}}$ the topologically singular set. 
\end{definition}

\begin{remark}
As we already remarked, the non-manifold set is included in the effective $1$-dimensional singular set for $\RCD(-2,3)$ spaces $(X,\dist,\haus^3)$ with empty boundary, namely it holds $\mathcal{S}_{\rm{top}}\subseteq \mathcal{S}^1_{\eps_0}$.
\end{remark}

	\begin{proposition}\label{prop:GballinR}
		Let $(X,\dist,\haus^3)$ be an $\RCD(-2,3)$ space. 
  Assume that $\eta_0\le \eta_0(v)$ and $\delta_0 \le \delta_0(v)$.
  If a Green-type ball $\overline{\mathbb{B}}_r(p)$ with radius $r\in\mathcal{G}_p$ is included in the manifold part $\mathcal{R}_{\rm{top}}(X)\subseteq X$, then it is a $3$-manifold with boundary and its boundary $\mathbb{S}_r(p)$ is homeomorphic to $S^2$.
	\end{proposition}
	
	\begin{proof}
		We claim that $\overline{\mathbb{B}}_r(p)$ is a manifold with boundary $\mathbb{S}_r(p)$. If this is the case, then it follows that $\mathbb{S}_r(p)$ is homeomorphic to $S^2$. Indeed, by \autoref{main.green.sphere} (vi), $\mathbb{S}_r(p)$ is either homeomorphic to $S^2$ or to $\mathbb{RP}^2$. Moreover, it bounds a compact $3$-manifold and hence it has even Euler characteristic. Hence $\mathbb{S}_r(p)$ is homeomorphic to $S^2$.

  \medskip
		
		The remaining part of the proof is dedicated to establishing the claim.
		Since $\mathbb{S}_r(p)$ is a closed surface, it suffices to show that the topological embedding $\mathbb{S}_r(p)\hookrightarrow\mathcal{R}_{\rm{top}}$ is locally tame
		(see \cite[p.\ 294]{Bing 1ULC} for the definition).
		
		By virtue of \cite[Theorem 7]{Bing 1ULC} and the compactness of a neighbourhood of $\mathbb{S}_r(p)$, it is enough to show the following:
		for any $q\in\mathbb{S}_r(p)$ and any open neighbourhood $U$ of $q$, there exists a smaller open neighbourhood $V\subseteq U$ of $q$
		such that any continuous map $S^1\to V\setminus\mathbb{S}_r(p)$ extends to a continuous map $\overline D^2\to U\setminus\mathbb{S}_r(p)$.
		
		Given $U\ni q$, let $V\subseteq U$ a small contractible neighbourhood of $q$, which exists as $q\in\mathcal{R}_{\rm{top}}$.
		Given a loop $\gamma:S^1\to V\setminus\mathbb{S}_r(p)$, recalling that $\mathbb{S}_r(p)=\{b_p=r\}$
		we see by continuity that either $\gamma( S^1)\subset\{b_p<r\}$ or $\gamma( S^1)\subset\{b_p>r\}$.
		Assume that we are in the first case. We select $s<r$ such that \autoref{retraction} applies and $\gamma(S^1)\subset\{b_p\le s\} = \overline{\mathbb{B}}_s(p)$.
		Since $V$ is contractible, $\gamma$ extends to a map $\Gamma:\overline D^2\to V$. Calling $\rho_s$ the retraction to $\overline{\mathbb{B}}_s(p)$ given by \autoref{retraction},
		whose domain includes $V$ (if $V$ is taken small enough), the required map is the composition $\rho_s\circ\Gamma$.
		Indeed, by \eqref{eq:contrretra}, $\sup_x\dist(x,\rho_s(x))\to0$ as both $b_p(x)\to r$ and $s\to r$.
		Taking $V$ small enough, this guarantees that $\rho_s(V)\subseteq U$. Thus, $\rho_s\circ\Gamma$ takes values in $U\setminus\mathbb{S}_r(p)$, as desired.
		
		The case where $\gamma( S^1)\subset\{b_p>r\}$ is similar, and follows from an analogous version of \autoref{retraction}, where we push a neighbourhood of $\mathbb{S}_r(p)$
		to a superlevel set $\{b_p\ge s\}$ with $s>r$. Given \cite[Theorem 7]{Bing 1ULC} this completes the proof that $\bar{\mathbb{B}}_r(p)$ is a manifold with boundary $\mathbb{S}_r(p)$ and hence of the proposition.

	\end{proof}

   \begin{remark}
       The tameness criterion in \cite{Bing 1ULC}, used in the proof of \autoref{prop:GballinR}, is based in turn on a previous result of Bing \cite{Bingtame}. The latter offers a different condition to check that a surface is tame, and it is based on the so-called approximation property from both sides. It seems possible that one could use the criterion from \cite{Bingtame} directly in our work. However, we choose to employ \cite{Bing 1ULC} for the sake of simplicity.
   \end{remark}

	Thanks to \autoref{prop:GballinR} we can already prove one of the two implications in \autoref{thm:RCDtopmain}. 
	
	\begin{corollary}\label{cor:impl1}
	Let $(X,\dist,\haus^3)$ be an $\RCD(-2,3)$ space. If $X$ is homeomorphic to a $3$-manifold without boundary, then all the cross-sections of all tangent cones of $X$ are homeomorphic to $S^2$.
	\end{corollary}
	
	\begin{proof}
	By \cite[Theorem 1.4]{BrueNaberSemolabdry}, if $(X,\dist)$ has nonempty boundary $\partial X$ then there is $p\in \partial X$ such that a neighbourhood of $p$ is homeomorphic to a neighbourhood of $0\in \setR^3_+$. This contradiction shows that $\partial X=\emptyset$ and \autoref{main.green.sphere} applies. For any $p\in X$ by \autoref{main.green.sphere} (vi) and \autoref{cor:densegoodr} we can find $r\in \mathcal{G}_p$ sufficiently small such that $\mathbb{S}_r(p)$ is homeomorphic to the cross-section of all tangent cones of $X$ at $p$. Since $X$ is a manifold by assumption, \autoref{prop:GballinR} implies that $\mathbb{S}_r(p)$ is homeomorphic to $S^2$. Hence all cross-sections of all the tangent cones of $X$ are homeomorphic to $S^2$. 
	\end{proof}

It is well known that a compact simply connected $3$-manifold with nonempty connected boundary homeomorphic to $S^2$ is contractible. We recall the proof here.

	\begin{proposition}\label{contractible}
		If $r\in\mathcal{G}_p$ and $\overline{\mathbb{B}}_r(p)\subseteq\mathcal{R}_{\rm{top}}$ then $\overline{\mathbb{B}}_r(p)$ is a contractible $3$-manifold with boundary, provided $\eta_0 \le \eta_0(v)$, $\delta_0\le \delta_0(v)$.
	\end{proposition}
	
	\begin{proof}
		By \autoref{prop:GballinR}, $M:=\overline{\mathbb{B}}_r(p)$ is a manifold with boundary homeomorphic to $S^2$.
		In particular, we have a long exact sequence in cohomology
		\begin{equation}
		\dots\to H^0(M)\to H^0(\de M)\to H^1(M,\de M)\to H^1(M)\to H^1(\de M)\to\dots \, .
		\end{equation}
		Since the first map $H^0(M)\to H^0(\de M)$ is surjective (as both $M$ and $\de M$ have one connected component), the next map is trivial; hence, we obtain
		the exact sequence
		\begin{equation}
		0\to H^1(M,\de M)\to H^1(M)\, .
		\end{equation}
		Thus, $H^1(M,\de M)$ is isomorphic to a subgroup of $H^1(M)$. Also, $M$ is simply connected by \autoref{simply.conn.clean}.
		In particular, we have $H_1(M)=0$, and the universal coefficient theorem gives
		$H^1(M)=0$ as well, and hence $H^1(M,\de M)=0$. Finally, by Lefschetz duality
		\begin{equation}
		H_2(M)\cong H^1(M,\de M)=0\, ,
		\end{equation}
		while $H_3(M)=0$ as $M$ is connected with nonempty boundary. Thus, $M$ has the weak homotopy type of a point. Since a compact manifold is always homotopy equivalent to
		a CW complex, Whitehead's theorem applies to $M$, and hence $M$ is contractible.
	\end{proof}

	\begin{remark}\label{poinc.ball}
	We note that by the resolution of the Poincar\'e conjecture any Green-ball as in the statement of \autoref{contractible} is homeomorphic to $\overline{D}^3$. This stronger statement will be important for some of the subsequent arguments. However, it plays no role in the proof of \autoref{prop:genmanrec}. 
	\end{remark}

	Later on, to complete the proof of \autoref{prop:genmanrec}, it will be important to know that the contractibility of good Green-balls holds under a weaker assumption (with respect to inclusion in the manifold part).

\begin{definition}\label{def:genmanset+}
    Given an $\RCD(-2,3)$ space $(X,\dist,\haus^3)$ we shall denote by $\mathcal{R}_{\rm{gm^+}}(X)=\mathcal{R}_{\rm{gm^+}}\subseteq X$ the open subset of $X$ consisting of those points $x\in X$ such that there exists a neighbourhood $U_x\ni x$ with the following two properties:
	\begin{itemize}
	\item[(i)] any good Green-ball $\mathbb{B}_r(y)\subset U_x$ is contractible;
	\item[(ii)] $H_{*}(X,X\setminus\{y\};\mathbb{Z})=H_{*}(\setR^3,\setR^3\setminus\{0\};\mathbb{Z})$ for all $y\in U_x$.
	\end{itemize}
		\end{definition}
		
	\begin{remark}\label{rm:gm+gm}
	It is clear from the definition that $\mathcal{R}_{\rm{gm^+}}$ is a generalized $3$-manifold without boundary. Later in Section \ref{sec:genman2} we will prove \autoref{prop:genmanrec} by showing that $\mathcal{R}_{\rm{gm^+}}(X)=X$.
	\end{remark}	
	
	Our next claim is that the local conditions in \autoref{def:genmanset+} yield global implications on the topology of all good Green-balls contained in $\mathcal{R}_{\rm{gm^+}}$.
	
	\begin{proposition}\label{prop:Gballscongm}
	Let $(X,\dist,\haus^3)$ be an $\RCD(-2,3)$ space.
 Assume that $\eta_0\le \eta_0(v)$ and $\delta_0 \le \delta_0(v)$.
 Let $p\in X$ and $r\in\mathcal{G}_p$ be such that $\overline{\mathbb{B}}_r(p)\subset \mathcal{R}_{\rm{gm^+}}$. Then $\overline{\mathbb{B}}_r(p)$ is a contractible generalized $3$-manifold with boundary $\mathbb{S}_r(p)$ homeomorphic to $S^2$. 
	\end{proposition}
	
	\begin{proof}
	We note that $\mathcal{R}_{\rm{gm^+}}\cap\partial X=\emptyset$. Indeed, if this is not the case, then by \cite[Theorem 1.4]{BrueNaberSemolabdry} there is a regular boundary point $x\in \mathcal{R}_{\rm{gm^+}}$ with a neighbourhood homeomorphic to $\mathbb{R}^3_+$. This is clearly in contradiction with the condition on the relative homology in \autoref{def:genmanset+} (ii).

	By \cite[Lemma 3.1]{DavermanThickstun} any good Green ball $\overline{\mathbb{B}}_r(p)\subset \mathcal{R}_{\rm{gm^+}}$ is a generalized $3$-manifold with boundary, the boundary being the Green-sphere $\mathbb{S}_r(p)$.

     \medskip
	
\emph{Claim 1:}
There exists $r'<r$ such that $\overline{\mathbb{B}}_r(p)\setminus \mathbb{B}_s(p)$ is homotopically equivalent to $\mathbb{S}_r(p)$ for all $r'<s<r$ such that $s\in \mathcal{G}_p$.

By \autoref{def:genmanset+} (i) and compactness, there exists $\rho>0$, independent of $s$, such that each Green ball of radius $r<\rho$ centered in $\overline{\mathbb{B}}_r(p)$ is contractible. In particular, employing the retraction built in \autoref{retraction} we can find $r''<r$ such that for any $r''<s<r$, $s\in \mathcal{G}_p$ the metric spaces $\overline{\mathbb{B}}_r(p)\setminus \mathbb{B}_s(p)$ (endowed with the restriction of the ambient distance $\dist$) are uniformly locally contractible, with contractibility radius independent of $s$.

As $s\to r$, $\overline{\mathbb{B}}_r(p)\setminus \mathbb{B}_s(p)$ Hausdorff converges to $\mathbb{S}_r(p)$. Hence by \cite{Petersen90} there exists $r''<r'<r$ such that the claimed homotopy equivalence holds for every $r'<s<s$ such that $s\in\mathcal{G}_p$.	

    \medskip

\emph{Claim 2:} $\mathbb{S}_r(p)$ is homeomorphic to $S^2$.
In order to prove this, we notice that the usual argument to prove that the Euler characteristic of the boundary of a closed manifold is even applies in the present setting, thanks to Claim 1.

More precisely, we consider the doubling $\widehat{\mathbb B}_r(p)$ obtained by gluing two copies of $\mathbb B_r(p)$ along the boundary $\mathbb S_r(p)$. To compute the Euler characteristic of $\mathbb S_r(p)$ we apply Mayer--Vietoris with two open sets intersecting in a neighbourhood of $\mathbb S_r(p)$ homoeomorphic to the doubling along $\mathbb S_r(p)$ of $\overline{\mathbb{B}}_r(p)\setminus \mathbb{B}_s(p)$, for some $s<r$ as in the proof of Claim 1.

Claim 1 ensures that this neighbourhood is homotopically equivalent to $\mathbb S_r(p)$. From this point on, the computation of the Euler characteristic of $\mathbb S_r(p)$ goes exactly as in the case of closed manifolds with boundary.

Hence, the Euler characteristic of $\mathbb{S}_r(p)$ is even and $\mathbb{S}_r(p)$ must be homeomorphic to $S^2$, and thus $\overline{B}_r(p)$ is simply connected by \autoref{simply.conn.clean}. Moreover, we claim that the Poincar\'e--Lefschetz duality holds for $\overline{B}_r(p)$, hence the argument in the first part of the proof of \autoref{contractible} can be repeated verbatim to obtain that $\overline{\mathbb{B}}_r(p)$ has the weak homotopy type of a point.  Note that $H_3(M)\cong H^0(M,\de M)=0$, again by Poincar\'e--Lefschetz.

 In order to prove the claim we note that Poincar\'e duality holds for generalized manifolds, even in the non-compact case \cite{Begleduality,Borelduality}. Moreover, as noted above there is a family of neighbourhoods of $\mathbb S_r(p)$ whose intersection coincides with $\mathbb S_r(p)$ and such that they are all homotopy equivalent to $\mathbb S_r(p)$. These two tools can be combined to show that Poincar\'e--Lefschetz duality holds for $\overline{B}_r(p)$ by following the classical proof in the case of manifolds with boundary: see for instance \cite[Theorem 3.43]{Hatcherbook}.
 
 Since $\overline{\mathbb{B}}_r(p)$ is a generalized manifold, it is an absolute neighbourhood retract in particular. Hence Whitehead theorem applies and from the fact that $\overline{\mathbb{B}}_r(p)$ has the weak homotopy type of a point we conclude that it is contractible.
\end{proof}

\subsection{Partial contractibility at conical scales}\label{subsec:partcontr}
		The proof of the local contractibility part of \autoref{prop:genmanrec} later in Section \ref{sec:genman2} will be based on an iterative argument. The main idea will be to break each $k$-cycle for $k\le 3$ into several $k$-cycles with ``smaller'' supports whose sum is homologous to the original cycle. The technical statements in this section will be instrumental in this decomposition procedure.
	\medskip
	
	The effect of \autoref{lemma:packedsingular} below is to show that, at each conical scale, the complement of the quantitative Reifenberg set is packed into a small tubular neighbourhood of a union of a uniformly bounded number of segments coming out of the base point. Moreover, \autoref{lemma:pushin} will tell us that, roughly speaking, all the topology at this conical scale, if there is any, lives in this tubular neighbourhood. The key insight is that, on an $\RCD(0,3)$ cone without boundary, the non-Reifenberg region is the cone over the non-Reifenberg region of the cross-section, and this is well controlled by \autoref{prop:sing2dsmall}.
\medskip

Recall that we fixed $\eps_0>0$ such that the metric Reifenberg theorem applies.

	\begin{lemma}\label{lemma:packedsingular}
		Let $v>0$, $\tilde C>1$ be fixed. For every $\overline \gamma >0$, if $0<\delta<\delta_0(v,\overline \gamma, \tilde C)$ the following statement holds. If $(X,\dist,\haus^3)$ is an $\RCD(-\delta,3)$ space without boundary and the ball $B_{2}(p)\subset X$ is $\delta$-conical then there exists a collection of at most $C(v)$ geodesic segments
		$\ell_j$ joining $p$ to $p_j\in\de B_1(p)$ and $\gamma\in (\overline \gamma, C(v,\tilde C)\overline{\gamma})$ such that
		\begin{equation}\label{eq:inclusionsing}
		(B_1(p)\setminus B_{5\gamma}(p)) \cap\mathcal{S}^1_{\eps_0}
        \subseteq (B_1(p)\setminus B_{5\gamma}(p))\setminus\mathcal{R}_{\eps_0,\gamma}
        \subseteq 
        \bigcup_j B_{(2\tilde{C})^{-1}\gamma}(\ell_j) \, .
		\end{equation}
		Moreover, the $\gamma$-tubular neighbourhoods of $\ell_j$ are disjoint, i.e.,
		\begin{equation}\label{eq:disjoint}
		[B_{\gamma }(\ell_j)\cap \left(B_1(p)\setminus B_{5 \gamma}(p)\right)] \cap [B_{\gamma }(\ell_{j'})\cap \left(B_1(p)\setminus B_{5 \gamma }(p)\right)]
        =\emptyset
		\end{equation}
		for each $j\neq j'$.
	\end{lemma}

\begin{proof}
The statement will be proved in two steps: in Step 1 we prove that it holds (actually in a stronger form) for cones. In Step 2 we argue that it holds in general via a contradiction argument based on Step 1.

 \medskip
	
\emph{Step 1:}
We claim that the statement holds on cones. Indeed, for an $\RCD(0,3)$ cone $C(Z)$ the effective Reifenberg set is the cone over the effective Reifenberg set of the cross-section, away from the vertex. Moreover, up to dimensional constants that we neglect to unburden the notation and the presentation, it holds
\begin{equation}
\mathcal{R}_{\eps_0,r}(C(Z))\setminus B_r(o)
=
C(\mathcal{R}_{\eps_0,r}(Z))\setminus B_r(o)\, ,
\end{equation}
for any $0<r<r(v)$.

Note that for an $\RCD(0,3)$ cone $C(Z)$ without boundary, $(Z,\dist_Z)$ must be an Alexandrov surface with empty boundary and curvature bounded below by $1$, by \cite{Ketterer15} and \cite{LytchakStadler}.

Hence, we can cover $Z\setminus\mathcal{R}_{\eps_0,\gamma}(Z)$ with at most $N(\eps_0)$ balls of radius $\gamma$, by \autoref{prop:sing2dsmall}. To conclude, we need to show that for every $\overline \gamma \le \overline\gamma(v)$ there exists $\gamma\in (\overline \gamma, C(v,\tilde C)\overline \gamma)$ such that $Z\setminus\mathcal{R}_{\eps_0,\gamma}(Z)$ can be covered with balls of radius $\tilde C^{-1}\gamma$ and their enlargement of radius $\gamma$ is a disjoint family.

The latter statement follows from an elementary packing argument: given any $\gamma \le \gamma(v)$, we can find a Vitali cover of $Z\setminus\mathcal{R}_{\eps_0,\gamma}(Z)$ with balls of radius $\tilde C^{-1}\gamma$ that are disjoint when we decrease the radii to $(5\tilde C)^{-1}\gamma$. If the family of balls with the same centers and with radius $\gamma$ is not disjoint, there exist two centers $x,y$ such that $2(5\tilde C)^{-1}\gamma \le \dist_Z(x,y) \le 2\gamma$. So the set of radii $\gamma$ such that the covering does not satisfy the desired properties is included in the finite union of intervals
\begin{equation}
    \bigcup_{x,y} \left[\frac{1}{2}\dist_Z(x,y), \frac{5 \tilde C}{2} \dist_Z(x,y)\right] \, .
\end{equation}
The claim easily follows.

\medskip

\emph{Step 2:} The statement follows from Step 1 by a contradiction argument based on the volume convergence theorem, volume monotonicity and volume almost-rigidity for $\RCD(K,n)$ spaces $(X,\dist,\haus^n)$: see \cite{DePhilippisGigli2} after \cite{Coldingvol,CheegerColding97I}. 
\end{proof}

\begin{lemma}\label{lemma:pushin}
Under the same assumptions and with the same notation of \autoref{lemma:packedsingular}, $\mathbb{B}_1(p)\subseteq B_{1+\gamma}(p)$ and the identity $\mathbb{B}_1(p)\to \mathbb{B}_1(p)$ is homotopic to a map $\mathbb{B}_1(p)\to\bigcup_j B_{\tilde{C}^{-1}\gamma}(\ell_j)\cup B_{10\gamma}(p)$, with a homotopy among maps with values in $\mathbb{B}_{1+5\gamma}(p)$.
\end{lemma}

\begin{proof}
The statement will follow from the metric Reifenberg theorem  \cite[Appendix 1]{CheegerColding97I}. 
\medskip

For the sake of the illustration, we first discuss the argument in the case where $B_1(p)\setminus \mathcal{R}_{\eps_0,\gamma}\subset B_{\gamma}(p)$, or equivalently, $B_2(p)$ is $\delta$-GH close to a cone $C(Z)$ with uniformly Reifenberg cross-section $\mathcal{R}_{\eps_0,\gamma}(Z)=Z$.

Under this simplifying assumption, the metric Reifenberg theorem \cite[Theorem A.1.1]{CheegerColding97I} and \cite[Theorem 4.6]{KapovitchPer} combine to show that there is a topological embedding
\begin{equation}
i:Z\times (8\gamma,1+2\gamma)\to B_{1+3\gamma}(p)\setminus B_{7\gamma}(p) 
\end{equation}
which is surjective onto 
\begin{equation}
B_{1+\gamma}(p)\setminus B_{9\gamma}(p) 
\end{equation}
and a $\gamma$-GH approximation. Hence $\mathbb{B}_1(p)$ deformation retracts onto $B_{10\gamma}(p)$ with a deformation with values into $B_{1+3\gamma}(p)\subset \mathbb{B}_2(p)$ and the statement clearly follows.

\medskip

In the general case, when $Z$ is not uniformly Reifenberg, an analogous strategy applies. By \cite[Theorem A.1.1]{CheegerColding97I} and \cite[Theorem 4.6]{KapovitchPer} again, there is a topological embedding
\begin{equation}\label{eq:embeddingprod}
i:\left(Z\setminus \bigcup_jB_{\tilde{C}^{-1}\gamma/2}(x_j)\right)
\times (8\gamma, 1+2\gamma)\to B_{1+3\gamma}(p)\setminus B_{7\gamma}(p) \, ,
\end{equation}
which is surjective onto 
\begin{equation}\label{eq:almostsurj}
B_{1+\gamma}(p)\setminus \left(B_{9\gamma}(p) \cup \bigcup_j B_{\tilde{C}^{-1}\gamma}(\ell_j) \right)
\end{equation}
and a $\gamma$-GH approximation. 
In \eqref{eq:embeddingprod} the points $x_j\in Z$ are chosen so that 
\begin{equation}
Z\setminus \bigcup_jB_{\tilde{C}^{-1}\gamma/2}(x_j)\subset \mathcal{R}_{\eps_0,\gamma}(Z) \, ,
\end{equation}
and the geodesic segments $\ell_j$ from the application of \autoref{lemma:packedsingular} are $\delta$-GH close to the segments joining the vertex $o$ to the points $(1,x_j)\in C(Z)$ for each $j$. From the product structure at the left-hand side in \eqref{eq:embeddingprod} and the surjectivity \eqref{eq:almostsurj}, we obtain as before that $\mathbb{B}_1(p)$ deformation retracts onto
\begin{equation*}
B_{10 \gamma}(p)\cup\bigcup_j B_{\tilde{C}^{-1}\gamma}(\ell_j) \, .\qedhere
\end{equation*}
\end{proof}

\section{Generalized manifold recognition: proof of \autoref{prop:genmanrec}}\label{sec:genman2}

	The goal of this section is to complete the proof of \autoref{prop:genmanrec}. We are going to argue that $\mathcal{R}_{\rm{gm}^+}(X)=X$; see \autoref{def:genmanset+} for the relevant notation. As noted in \autoref{rm:gm+gm}, this will be enough to show that $X$ is a generalized $3$-manifold with empty boundary. 
\medskip	
	
Let us fix $\eta_0, \delta_0>0$ sufficiently small so that all the results of Section \ref{sec:genmanprel} apply. In particular, the good Green-balls obtained by \autoref{main.green.sphere}:
\begin{itemize}
\item[(i)] admit well-behaved retractions from a neighbourhood, by \autoref{retraction};
\item[(ii)] are simply connected if the Green-sphere is homeomorphic to $S^2$, by \autoref{simply.conn.clean};
\item[(iii)] are contractible if included in $\mathcal{R}_{\rm{gm}^+}$, by \autoref{prop:Gballscongm};
\item[(iv)] deformation retract onto a scale-invariantly small tubular neighbourhood of their effective singular set, by \autoref{lemma:packedsingular} and \autoref{lemma:pushin}.
\end{itemize}
We argue by contradiction and assume that $\mathcal{S}_{\rm{gm}^+}:=X\setminus\mathcal{R}_{\rm{gm}^+}\neq\emptyset$.	

 \medskip
	
For $r>0$, we let
\begin{align}\label{delta.0}
&\mathcal{S}_{\rm{ gm}^+}^r:=\{p\in\mathcal{S}_{\rm{gm}^+}\,:\,B_s(p)\text{ is $\delta_0$-conical for all $s\in(0,r)$}\}\, .
\end{align}
It is elementary to check that $\mathcal{S}_{\rm{ gm}^+}^r$ is closed in $X$ for each $r>0$, by compactness of the class of $\RCD(1,2)$ spaces $(X,\dist,\haus^2)$ with volume uniformly bounded away from $0$. Moreover, by \autoref{lemma:infveryc}, 
\begin{equation}
\mathcal{S}_{\rm{gm}^+}=\bigcup_{r\in(0,1)}\mathcal{S}_{\rm{gm}^+}^r\, . 
\end{equation}
By Baire's category theorem, we can find $r\in(0,1)$ such that $\mathcal{S}_{\rm{gm}^+}^r$ has nonempty interior in $\mathcal{S}_{\rm{gm}^+}$.
Up to rescaling, this means that we can find $p_0\in \mathcal{S}_{\rm{gm}^+}$ such that
\begin{align}
\label{baire}
\mathcal{S}_{\rm{gm}^+}\cap B_{10}(p_0)\subseteq\mathcal{S}_{\rm{gm}^+}^{10}\, .
\end{align}
In other words, for all $p\in \mathcal{S}_{\rm{gm}^+}\cap B_{10}(p_0)$ and all $r\in(0,10)$, the ball $B_r(p)$ is $\delta_0$-conical.

\medskip

In the rest of this section, we are going to argue that any good Green-ball contained in $B_{10}(p_0)$ is contractible and that at each point $p\in B_{10}(p_0)$ the local relative homology coincides with the one of $(\setR^3,\setR^3\setminus\{0\})$. This results in a contradiction since we assumed that $\mathcal{S}_{\rm{gm}^+}^{10}\cap B_{10}(p_0)\neq \emptyset$ (in particular, this set contains $p_0$).
	
	There will be three main steps for the proof, corresponding to different subsections.
In Section \ref{subsec:GS^2} we will exploit the uniform conicality at singular points \eqref{baire} in combination with the assumption that the cross-sections of all tangent cones are homeomorphic to $S^2$ and \autoref{prop:Gballscongm} to prove the following.

	\begin{proposition}\label{mfd.with.bdry}
		For any $p\in B_{5}(p_0)$ and any good radius $r\in\mathcal{G}_p\cap(0,1)$, the Green-sphere $\mathbb{S}_r(p)=\de\mathbb{B}_r(p)$
		is homeomorphic to $S^2$.
	\end{proposition}

In Section \ref{subsec:loccontr} we will prove that Green-balls contained in $B_{1}(p_0)$ are contractible.

	\begin{proposition}\label{prop:3ctoc}
		For any $p\in B_1(p_0)$ and $r\in(0,c(\nu))\cap\mathcal{G}_p$ the closed Green-ball $\overline{\mathbb{B}}_r(p)$ is contractible.
	\end{proposition}

The contractibility will be achieved by first arguing that the following holds.

	\begin{proposition}
		For any $p\in B_1(p_0)$ and $r\in(0,c(\nu))\cap\mathcal{G}_p$
		the Green-ball $\mathbb{B}_r(p)$ and its closure are $3$-connected.
	\end{proposition}

We recall that $3$-connectedness amounts to say that the $k$-th homotopy group is trivial for each $k\le 3$. The statement will be proved inductively on $k$. Note that the base step, corresponding to simple-connectedness, follows from \autoref{mfd.with.bdry} and \autoref{simply.conn.clean}.
In Section \ref{subsec:locrelhom} we will finish by computing the local relative homology.

\begin{proposition}\label{prop:locrelhom}
	For every $p\in B_1(p_0)$ it holds 
	\begin{equation}
	H_{*}(X,X\setminus\{p\};\Z)=H_{*}(\setR^3,\setR^3\setminus\{0\};\Z)\, .
	\end{equation} 
	\end{proposition}

The key insight will be that all sufficiently small punctured Green-balls deformation retract onto the respective Green-spheres, as we will see in \autoref{punct}.
\medskip

By \autoref{prop:3ctoc} and \autoref{prop:locrelhom}, $B_1(p_0)\subset \mathcal{R}_{\rm{gm}^+}$. This results in a contradiction with the assumption that $p_0\in\mathcal{S}_{\rm{gm}^+}$ and hence the proof of \autoref{prop:genmanrec} will be completed at the end of this section.

\subsection{Green-spheres are homeomorphic to $S^2$: proof of \autoref{mfd.with.bdry}}\label{subsec:GS^2}	

The goal of this subsection is to prove \autoref{mfd.with.bdry}. The fact that Green-balls not intersecting $\mathcal{S}_{\rm{gm}^+}$ have a boundary Green-sphere homeomorphic to $S^2$ follows already from \autoref{prop:Gballscongm}. The next step is to prove that the statement holds for Green-balls centered on $\mathcal{S}_{\rm{gm}^+}$.

\begin{lemma}\label{lemma:GspheresS}
For all $p\in\mathcal{S}_{\rm{gm}^+}\cap B_{10}(p_0)$ and all good radii $r\in\mathcal{G}_p$, the Green-sphere $\mathbb{S}_r(p)$ is homeomorphic to $S^2$.
\end{lemma}

\begin{proof}	
The assumption that all the cross-sections of all tangent cones of $X$ are homeomorphic to $S^2$ yields the statement for Green-spheres with sufficiently small radii by \autoref{main.green.sphere} (vi). We shall use the assumption of uniform conicality at all scales for points $p\in \mathcal{S}_{\rm{gm}^+}$ to propagate the conclusion up to scale $1$.

\medskip

By \eqref{baire}, for each $r\in(0,1)$, there exists an Alexandrov surface $(Z_r,\dist_r)$ with curvature $\ge 1$ and $\haus^2(Z_r)\ge C(v)$ such that $\dist_{\rm{GH}}(B_{2r}(p),B_{2r}(o))\le2\delta_0 r$, where $B_{2r}(o)\subset C(Z_r)$. Note that $Z_r$ has empty boundary by \cite{BrueNaberSemolabdry}, since $(X,\dist)$ has empty boundary by assumption. 
		
		A standard compactness and contradiction argument shows that, by taking $\delta_0$ small, we can guarantee that any two such Alexandrov spaces $(Z_s,\dist_s)$ and $(Z_{s'},\dist_{s'})$ with $s<s'<2s$ are arbitrarily close to each other, and in particular homotopy equivalent by \cite[Theorem A]{Petersen90}. Hence the Alexandrov surfaces $Z_r$ are either all homotopy equivalent to $S^2$ or all homotopy equivalent to $\mathbb{RP}^2$. The assumption that the cross-sections of all tangent cones at $p$ are homeomorphic to $S^2$ implies that the first possibility occurs. By \autoref{main.green.sphere} (vi), the same is true for all the good Green-spheres $\mathbb{S}_r(p)$ for $r\in \mathcal{G}_p$.
	\end{proof}

We are now ready to complete the proof of \autoref{mfd.with.bdry}.
	
	\begin{proof}[Proof of \autoref{mfd.with.bdry}]
		If $p\in\mathcal{S}_{\rm{gm}^+}$ the statement corresponds to \autoref{lemma:GspheresS}. If $\overline{\mathbb{B}}_r(p)\subseteq\mathcal{R}_{\rm{gm}^+}$, the statement is part of \autoref{prop:Gballscongm}. Hence we can assume that
		$\overline{\mathbb{B}}_r(p)\cap\mathcal{S}_{\rm{gm}^+}\neq\emptyset$. 
		
		Let $p'\in B_{10}(p_0)\cap \overline{\mathbb{B}}_r(p)\cap\mathcal{S}_{\rm{gm}^+}$
	and notice that $\mathbb{B}_r(p)\subseteq B_{2r}(p)\subseteq{B}_{4r}(p')$. We will distinguish two cases, depending on the scale-invariant distance $\dist(p,p')/r$. In the case when it is small, the statement will follow from \autoref{lemma:GspheresS}. If the scale-invariant distance between the centers is large, this will force an additional almost splitting by cone-splitting, thus ruling out the case of cross-sections homeomorphic to $\mathbb{RP}^2$.

     \medskip
 
By \eqref{baire}, $B_{4r}(p')$ is $(4\delta_0 r)$-close in the Gromov--Hausdorff sense to $B_{4r}(o)\subset C(Z')$, where $(Z',\dist_{Z'})$ is an Alexandrov sphere with curvature $\ge 1$. By definition of $\mathcal{G}_p$, $B_{2r}(p)$ is also $(2\delta_0 r)$-close in the Gromov--Hausdorff sense to $B_{2r}(o)\subset C(Z)$, where $(Z,\dist_Z)$ is an Alexandrov surface with curvature bounded below by $1$.

If $\dist(p,p')\le c(v)r$ is small enough, it follows from \cite[Theorem A]{Petersen90} that $Z$ and $Z'$ are homotopy equivalent (provided $\delta_0$ is suitably small).
Otherwise $\dist(p,p')\ge c(v)r$ and, assuming again without loss of generality that $\delta_0$ is small enough, then $B_{2r}(p)$ almost splits a factor $\R$ by almost cone-splitting, i.e., it is close to the ball of radius $2r$ inside a cone of the form $\R\times Y$, where $Y$ is conical with tip $y_0$, with $p$ and $p'$ corresponding to $q=(0,y_0)$ and $q'=(d(p,p'),y_0)$: see \cite{CheegerNaber13}, where the argument originates from and the subsequent \cite[Theorem 1.17]{AntonelliBrueSemola} for the $\RCD$ setting. The cross-section of any such cone is a spherical suspension and an Alexandrov surface with curvature bounded below by $1$ with empty boundary, by the stability of $\RCD(K,n)$ spaces $(X,\dist,\haus^n)$ with empty boundary under noncollapsed GH-convergence \cite{BrueNaberSemolabdry}. Hence it is homeomorphic to $S^2$. 

By \autoref{main.green.sphere} (vi), also $\mathbb{S}_r(p)$ is homeomorphic to $S^2$.
\end{proof}

\subsection{Local uniform contractibility}\label{subsec:loccontr}

This section aims to prove \autoref{prop:3ctoc} below, namely the contractibility of closed Green-balls $\bar{\mathbb{B}}_r(p)\subset B_1(p_0)$. The key step will be to prove that they are $3$-connected, i.e., they have trivial $\pi_k$ for each $k\le 3$.

	\begin{proposition}\label{prop:3conn}
		If $\eta_0 \le \eta_0(v)$, $\delta_0 \le \delta_0(v)$ in the definition of $\mathcal{G}_p$, for any $p\in B_1(p_0)$ and $r\in(0,c(v))\cap\mathcal{G}_p$
		the Green-ball $\mathbb{B}_r(p)$ and its closure are 3-connected.
	\end{proposition}

	Borrowing the terminology from \cite{Hubook}, it follows from \autoref{prop:3conn} and the first part of \autoref{cor:densegoodr} that $B_1(p_0)$ is uniformly locally $3$-connected. Since $B_1(p_0)$ has covering dimension $3$ by \autoref{lemma:dimcov}, the general theory from \cite[Chapters III--V]{Hubook} then implies that $B_1(p_0)$ is locally contractible. We are going to prove a slightly stronger statement in our setting, namely that all the Green-type balls with good radii are contractible. We start by establishing this claim, thus completing the proof of \autoref{prop:3ctoc} modulo the proof of \autoref{prop:3conn} that will be discussed below.
	
	\begin{proposition}
		For $p\in B_1(p_0)$ and $r\in(0,c(v))\cap\mathcal{G}_p$, the closed Green-type ball $\overline{\mathbb{B}}_r(p)$ is contractible.
	\end{proposition}
	
	\begin{proof}
		The argument is similar to \autoref{retraction} and hence it will only be sketched; see also \cite[Chapter V, Theorem 7.1]{Hubook} for an analogous argument.

		\medskip
		
		Letting $Y:=\overline{\mathbb{B}}_r(p)$, we need to find a map $F:[0,1]\times Y\to Y$
		such that $F(0,x)=x$ and $F(1,x)=p$ for any $x\in Y$. As in the proof of \autoref{retraction}, we cover $(0,1)\times Y$ with open sets $U_a$ whose diameter shrinks to zero as they approach $\{0,1\}\times Y$.
		Since $Y$ has covering dimension $3$ by \autoref{lemma:dimcov}, $[0,1]\times Y$ has covering dimension 4. Hence we can assume that at most five $U_a$'s have nonempty intersection.
		Thus, the associated nerve $K$ has dimension 4 as well. Let $\varphi:(0,1)\times Y\to K$ (as in the proof of \autoref{retraction})
		and let us build $f:K\to Y$.
		
		We can define $f(a)$ to be $p$ if $U_a\subseteq(\mz,1)\times Y$
		and to be a chosen point $p_a\in U_a$ otherwise, and extend $f$ inductively on the skeleta of $K$, taking care that $f(\Delta)=\{p\}$
		for each simplex such that all its vertices $a$ correspond to a set $U_a\subseteq(\mz,1)\times Y$. The extension can be done since $Y$ is 3-connected by \autoref{prop:3conn},
		guaranteeing that a continuous map $\de\Delta\to Y$ extends to a map $\Delta\to Y$ for any simplex $\Delta$ of dimension $1\le k\le 4$
		(as $\de\Delta\cong S^{k-1}$).
		By \autoref{prop:3conn} and \autoref{cor:densegoodr}, we can guarantee that $f(\Delta)$ has diameter comparable with that of $f(\de\Delta)$ for each simplex, by taking an extension $\Delta\to \bar{\mathbb{B}}_s(p')$ with values in a smaller ball $\bar{\mathbb{B}}_s(p')$ and composing with the retraction onto $\bar{\mathbb{B}}_r(p)$.
		
		It is easy to see that $f\circ\varphi$ extends to a map $F:[0,1]\times Y\to Y$ as above.
	\end{proof}

	To prove \autoref{prop:3conn}, roughly speaking, we argue that the homotopy groups $\pi_k$ are trivial for each $k\le 3$ by finite induction on $k$. The base step corresponds to simple-connectedness and it follows already from \autoref{simply.conn.clean} and \autoref{mfd.with.bdry}. We give a precise statement below.

		\begin{lemma}\label{simply.conn}
		For every $p\in B_1(p_0)$ and for every $r\in \mathcal{G}_p\cap(0,1)$ the Green-ball $\mathbb{B}_r(p)$ and its closure $\overline{\mathbb{B}}_r(p)$ are simply-connected.
		Moreover, there exists a constant $C(v)>1$ such that,
		for any $p\in B_1(p_0)$ and any $r\in(0,1/C)$,
		the inclusion $B_r(p)\hookrightarrow B_{Cr}(p)$ induces the trivial map $\pi_1(B_r(p),p)\to\pi_1(B_{Cr}(p),p)$.
	\end{lemma}
	
	\begin{proof}
		We already argued that the first part of the statement follows from \autoref{simply.conn.clean} and \autoref{mfd.with.bdry}.
		The second conclusion follows from the first one by taking into account \autoref{cor:densegoodr}.
	\end{proof}

The proof of the $k$-connectedness of Green-balls for $k\ge 2$ requires a different argument. We discuss the idea for $k=2$; the case $k=3$ will be completely analogous.
 
Note that by Hurewicz it is sufficient to argue that the second homology $H_2$ of a Green-ball $\mathbb{B}_r(p)$ is trivial. This statement will be proved with an iterative construction. We outline the key steps.

\emph{Step 1:} By \autoref{lemma:pushin} we can push any $2$-cycle $[\sigma]$ supported in $\mathbb{B}_r(p)$ into a homologous $2$-cycle $[\sigma']$ supported in a scale-invariantly small tubular neighbourhood of the effective singular set (of the form discussed in \autoref{lemma:packedsingular}).  

\emph{Step 2:} We carefully cover the support of $[\sigma']$ with finitely many scale-invariantly small (and simply connected, by \autoref{simply.conn}) Green-balls. By applying the Mayer--Vietoris sequence to this covering, we can break $[\sigma']$ into a homologous $\sum[\sigma_j]$, where each $[\sigma_j]$ is supported in one of the small Green-balls. 

\emph{Step 3:} For those Green-balls in the covering whose center is at a distance less than the radius from $\mathcal{S}_{\rm{gm}^+}$, this decomposition step can be iterated (without deterioration of the parameters) by the uniformity assumption \eqref{baire}. 

\emph{Step 4:} For those Green-balls that are contained in $\mathcal{R}_{\rm{gm}^+}$, \autoref{prop:Gballscongm} applies and hence the corresponding $2$-cycles in the sum are trivial in homology.

\medskip

	The proof will be completed with the help of the following useful and general criterion to obtain $k$-connectedness. Its proof is inspired by the argument in \cite{WangRCD}.

	\begin{lemma}\label{cinese.plus}
		Let $(X,\dist)$ be a complete metric space.
		Given $k\ge 2$, let $\mathcal{U}$ be a family of bounded $(k-1)$-connected open sets such that, if $U\in\mathcal{U}$ and $\sigma$ is a $k$-cycle in $U$, there exist finitely many $\{U_j\}_{j\in J} \subset \mathcal{U}$ and $k$-cycles $\sigma_j$ in $U_j$ such that
         \begin{itemize}
             \item[(i)] $U_j\cap U \neq \emptyset$ for every $j\in J$;

             \item[(ii)] $\diam(U_j) < \frac{1}{8} \diam(U)$;

             \item[(iii)] $\sum_j[\sigma_j]=[\sigma]$ in $H_k(B_{\diam(U)/4}(U))$.
         \end{itemize}
        Moreover, we assume that the $\frac{1}{4}\diam(U)$-neighbourhood
        of $U$ is included in a $(k-1)$-connected open set $U'$, itself included in the $\frac{1}{2}\diam(U)$-neighbourhood of $U$.
		Then each $U\in\mathcal{U}$ is $k$-connected in a $\diam(U)$-neighbourhood of $U$.
	\end{lemma}

	\begin{proof}

      We identify $S^k$ with the boundary of the unit cube $Q=Q^{k+1}$, quotiented by the union $Z$ of all sides but the top one.
		Thus, given $U\in\mathcal{U}$ and fixing a basepoint $p\in U$, we identify a map $\gamma: S^k\to U$ based at $p$ with a map $\de Q\to U$ with constant value $p$ on $Z$.

		We can associate a homology class $[\sigma]\in H_k(U)$ to any element $[\gamma]\in\pi_k(U,p)$ in a canonical way by taking the image of the fundamental class $[\de Q]$ through $\gamma:\de Q\to U$.
		By assumption, we can find $(k-1)$-connected open sets $U_1,\dots,U_J$ of diameter less than $\frac{1}{8}\operatorname{diam}(U)$ and cycles $\sigma_j$ in $U_j$
		such that
		\begin{equation}
		\sum_{j=1}^J[\sigma_j]=[\sigma]\, , \quad \text{in $H_k(U)$}\, .
		\end{equation}
		Since the open sets $U_j$ are $(k-1)$-connected, we can select basepoints $p_j\in U_j$ and apply Hurewicz's theorem (in $U_j$)
		in order to realize each $[\sigma_j]$ as the image of a map $\gamma_j:\de Q\to U_j$ based at $p_j$.
		Further, for each $j=1,\dots, J$ we can select a path $\eta_j$ joining $p$ to $p_j$ in $U'$. Then we can define new maps $\eta_j\cdot\gamma_j:\de Q\to U'$ based at $p$
		by rescaling the domain of $\gamma_j$ to be $\de Q'$, with $Q':=[\frac13,\frac23]^k\times[\frac23,1]$, extending to $Q\setminus \operatorname{int}(Q')$ by copying the path $\eta_j$ along rays from the point $(\mz,\dots,\mz,1)$ on the top side (i.e., along their intersection with $Q\setminus \operatorname{int}(Q')$), and restricting to $\de Q$. Each of these maps is based at $p$ and their concatenation
		yields $\sum_{j=1}^J[\sigma_j]=[\sigma]$ under the natural homomorphism $\pi_k(U',p)\to H_k(U')$.
		By Hurewicz (applied in $U'$), the latter is an isomorphism. Hence, the concatenation of the maps $\eta_j\cdot\gamma_j$ is homotopic to 
		$\gamma$ in $(U',p)$. 
		
		Now we select smaller cubes $Q_j$ in $Q$, sitting next to each other along the $x_1$-direction. For instance, we can let
		\begin{equation}
		Q_j:=\left[\frac{j}{J+2},\frac{j+1}{J+2}\right]\times\left[\frac{1}{J+2},\frac{2}{J+2}\right]^{k} \, ,\quad \text{for $j=1,\dots, J$}\, .
		\end{equation}
		We define a map $\Gamma$ on $\de Q\cup\bigcup_j\de Q_j$
		by copying $\gamma$ on $\de Q$ and $\eta_j\cdot\gamma_j$ on $\de Q_j$; this map is continuous since the value is $p$ on each intersection $\de Q_j\cap \de Q_{j'}$.
		The homotopy obtained before allows to extend $\Gamma$ to a continuous map defined on $Q\setminus\bigcup_j\operatorname{int}(Q_j)$, with values in $U'$. We can further extend it to each $Q_j\setminus\operatorname{int}(Q_j')$, as before, so that the image of each $\de Q_j'$ is included in $U_j$.
		\medskip
		
		The previous procedure can be iterated with each of the maps $\gamma_j$ in $U_j$, and so on.
        After $\ell$ steps, we then obtain a map defined on a set of the form $Q\setminus\bigcup_{i\in I_\ell}\operatorname{int}(Q_{\ell,i}')$,
        taking values in an $\eta_\ell$-neighbourhood of $U$, with
        $$\eta_\ell:=\frac{1}{2}\operatorname{diam}(U)\sum_{\ell'=0}^{\ell-1}\frac{1}{8^{\ell'}}.$$
        Moreover, the cubes $Q_{\ell,i}'\subset Q$ have disjoint interiors and diameter at most $(k+1)3^{-\ell}$, and the image of each $\de Q_{\ell,i}$
        has diameter at most $8^{-\ell}\operatorname{diam}(U)$.
        After repeating the procedure infinitely many times, we obtain a uniformly continuous map defined on a dense subset of $Q$,
		which then extends to a map $Q\to X$ by completeness.
        This gives a homotopy between $\gamma$ and the constant map in a
        $\operatorname{diam}(U)$-neighbourhood of $U$.
	\end{proof}

In the next lemma, we fix parameters $\delta_0'=\delta_0'(v)$ and $\eta_0'=\eta_0(v)$ such that \autoref{simply.conn} and \autoref{cor:densegoodr} apply.

		\begin{lemma}\label{lemma:2connected}
        If $\delta_0 \le \delta_0(\eta_0',\delta_0',v)$ there exists $C=C(\eta_0',\delta_0',v)>1$ such that the following holds.		
		For any $p\in B_1(p_0)$ and $r\in \mathcal{G}_p\cap (0,C^{-1})$ the closed Green-ball $\overline{\mathbb{B}}_r(p)$ is $2$-connected. Moreover, for every $p\in B_1(p_0)$ and $r\in(0, C^{-1})$, $B_r(p)$ is $2$-connected in $B_{Cr}(p)$, i.e., the inclusion $B_r(p)\hookrightarrow B_{Cr}(p)$ induces the trivial map on the first and second homotopy groups.
	\end{lemma}

       \begin{proof}
		If $\overline{\mathbb{B}}_r(p)\subset \mathcal{R}_{\rm{gm}^+}$, then $\mathbb{B}_r(p)$ is contractible by \autoref{contractible}.
		Otherwise, take $p'\in\overline{\mathbb{B}}_r(p)\cap\mathcal{S}_{\rm{gm}^+}$. Then it holds
		\begin{equation}
		\mathbb{B}_r(p)\subseteq B_{2r}(p)\subseteq B_{4r}(p') \, ,
		\end{equation}
		and there is $s\in(4r,4C(v)r)$ such that $\mathbb{B}_s(p')$ is simply connected by \autoref{simply.conn}. Thus, it suffices to show that $\mathbb{B}_s(p')$ is 2-connected in a larger ball $B_{10s}(p')$. Then the first conclusion of \autoref{lemma:2connected} follows from the second one by applying \autoref{lambda}.

To prove the desired statement, we apply \autoref{cinese.plus} with $k=2$ and $\mathcal{U}$ the family of good simply connected Green-balls $\mathbb{B}_{r}(p')$ with $p'\in \mathcal{S}_{\rm{gm}^+}\cap B_2(p_0)$ and
$$r\in \mathcal{G}_{p'}(\eta_0',\delta_0')\cap \left(0,\frac{2-\dist(p_0,p')}{10}\right),$$
assuming $\delta_0 \le \delta_0(\delta_0',\eta_0',v)$.

\medskip

Let $\mathbb{B}_{s}(p')\in \mathcal{U}$.
Note that $p'\in\mathcal{S}_{\rm{gm}^+}\cap B_2(p_0)\subseteq\mathcal{S}_{\rm{gm}^+}^{10}$ by \eqref{baire}. We apply \autoref{lemma:packedsingular} and \autoref{lemma:pushin} in the ball $B_{2s}(p')$.
Following the notation of \autoref{lemma:packedsingular},  we fix $\overline \gamma=\overline \gamma(v)$ and $\tilde C := \overline \gamma^{-1}$. If $\delta_0 \le \delta_0(v,\overline \gamma)$ there exists $\gamma\in (\overline \gamma^3, C(v)\tilde C^{-1} \overline \gamma^3) \subset (\overline \gamma^3, \overline \gamma)$
and a set of at most $C(v)$ geodesic rays $\ell_j$ connecting $p'$ to $\partial B_{2s}(p')$ such that $B_{2s}(p')$ deformation retracts onto $B_{10 \gamma s}(p') \cup \bigcup_j B_{\gamma \overline \gamma s}(\ell_j)$. Moreover, $\{B_{\gamma s}(\ell_j)\}$ is a disjoint family in $B_{2s}(p')\setminus B_{10\gamma s}(p')$.
In particular, given a 2-cycle $\sigma$ in $\mathbb{B}_s(p')$, we can replace it with a homologous cycle supported in $B_{10 \gamma s}(p') \cup \bigcup_j B_{\gamma \overline \gamma s}(\ell_j)$.

Let $M:=\lceil\frac{1}{2 \gamma}\rceil+10$
and replace each $\ell_j$ with its intersection with $B_{(1+10\gamma)s}(p')$. It holds
\begin{equation}\label{eq:biginclusion}
\bigcup_j B_{\gamma \overline \gamma s}(\ell_j)\subset B_{10\gamma s}(p')\cup\bigcup_j\bigcup_{m=1}^{M} B_{\gamma s}(\ell_j(2m\gamma s))\cup\bigcup_j\bigcup_{m=1}^{M} B_{2\gamma\overline \gamma s}(\ell_j(2m\gamma s-\gamma s))\, .
\end{equation}
Since each $\ell_j$ is a length-minimizing geodesic, taking into account also \eqref{eq:disjoint} and neglecting $B_{10\gamma s}(p')$, we infer that the only intersecting pairs of balls at the right-hand side in \eqref{eq:biginclusion} are $B_{\gamma s}(\ell_j(2m\gamma s))$ with $B_{2\gamma \overline \gamma s}(\ell_j(2m\gamma s\pm\gamma s))$.
For each of the smaller balls $B_{2\gamma \overline \gamma s}(\ell_j(2m\gamma s-\gamma s))$, we enlarge it to a simply connected Green-ball with parameters $\eta_0', \delta_0'$ of radius at most $C(\eta_0',\delta_0',v)\bar\gamma\gamma s$, by \autoref{simply.conn} and \autoref{cor:densegoodr} again. 

We will denote by $\tilde B_{2\gamma \overline \gamma s}(\ell_j(2m\gamma s-\gamma s))$ the enlarged ball. We can assume that $\overline \gamma \le \overline \gamma(\eta_0', \delta_0')$ so that $10C(\eta_0',\delta_0',v)\gamma \overline \gamma <\gamma$. In particular, the enlarged Green balls are disjoint.

Consider the sets
\begin{equation}
\begin{split}
\mathcal{P}_0
& :=B_{11\gamma s}(p') \cup \tilde B_{2\gamma \overline \gamma s}(\ell_j(11\gamma s)) 
\\
\mathcal{P}_j
& :=\bigcup_{m = 6}^{M} B_{\gamma s}(\ell_j(2m\gamma s))\cup \bigcup_{ m = 6 }^{M} \tilde B_{2\gamma \overline \gamma s}(\ell_j(2m\gamma s-\gamma s))\, . 
\end{split}
\end{equation}
It is elementary to check that the support of $\sigma$ is contained in $\cup_j\mathcal{P}_j$. Moreover, $\mathcal{P}_j\cap \mathcal{P}_{j'}=\emptyset$, for each $1\le j<j'$, and $\mathcal{P}_0\cap\mathcal{P}_j=\tilde B_{2\gamma \overline \gamma s}(\ell_j(11\gamma s))$ is simply connected, for all $j\ge1$. Thus, a straightforward application of Mayer--Vietoris shows that $\sigma$ is homologous to a sum $\sum_{j\ge0}\sigma_j$, where each 2-cycle $\sigma_j$ is supported in $\mathcal{P}_j$.

\medskip

For each $j\ge0$, we can consider the open sets
\begin{equation}
\tilde B_{2\gamma \overline \gamma s}(\ell_j(2m\gamma s-\gamma s))\cup B_{\gamma s}(\ell_j(2m\gamma s))\cup \tilde B_{2\gamma \overline \gamma s}(\ell_j(2m\gamma s+\gamma s))
\end{equation}
for $m\ge 6$, as well as $B_{11\gamma s}(p')\cup\bigcup_{j>0}\tilde B_{2\gamma\overline \gamma s}(\ell_j(11\gamma s))$ for $m=5$.
Again, any two of them, corresponding to two indices $m<m'$, are either disjoint (if $m+1<m'$) or intersect on the simply connected open set $\tilde B_{2\gamma\overline \gamma s}(\ell_j(2m\gamma s+\gamma s))$ (if $m+1=m'$).
Thus, up to boundaries, we can further split each $\sigma_j$ into a sum of 2-cycles, each supported in one of these sets.

If $\overline \gamma \le \overline \gamma(v,\eta_0',\delta_0')$, each of these sets is included in a simply connected Green-type ball of radius $\le C(v)\gamma s$.
If the closure of the latter is included in $\mathcal{R}_{\rm{gm}^+}$, then the corresponding cycle is homologically trivial by \autoref{contractible}.
Otherwise, we can replace it with a larger simply connected Green-type ball $\mathbb{B}\in\mathcal{U}$, centered at a point $p''\in \mathcal{S}_{\rm{gm}^+}$.
The conclusion follows from \autoref{cinese.plus}.
\end{proof}
	
The very same argument as in the proof of \autoref{lemma:2connected}, shifting up by one all the dimensions and relying on the $2$-connectedness of Green-balls rather than on their simple-connectedness in the application of \autoref{cinese.plus}, yields the following.

\begin{lemma}\label{lemma:3con}
For any $p\in B_1(p_0)$ and any $r\in(0,C^{-1})$, $B_r(p)$ is $3$-connected in $B_{Cr}(p)$, for some $C=C(v)$. Moreover, if $r\in\mathcal{G}_p\cap(0,C^{-1})$ then $\mathbb{B}_r(p)$ and its closure are $3$-connected, up to decreasing $\delta_0=\delta_0(v)$ and $\eta_0=\eta_0(v)$.
\end{lemma}

	\subsection{Local relative homology: proof of \autoref{prop:locrelhom}}\label{subsec:locrelhom}

	The goal of this subsection is to prove \autoref{prop:locrelhom}. Namely we will show that the local relative homology at each point $p\in B_1(p_0)$ coincides with the one of $(\setR^3,\setR^3\setminus\{0\})$.

	The combination of \autoref{prop:3ctoc} and \autoref{prop:locrelhom} leads to a contradiction to the assumption $p_0\in\mathcal{S}_{\rm{gm}^+}$ and hence completes the proof of \autoref{prop:genmanrec}.
	\medskip
	
	The main new tool that we need for the proof of \autoref{prop:locrelhom} is the following.
	
	\begin{lemma}\label{punct}
	For any $p\in B_1(p_0)$ and for any $r\in\mathcal{G}_p$ small enough (depending on $p$),
		the punctured Green-ball $\overline{\mathbb{B}}_r(p)\setminus\{p\}$ deformation retracts onto its boundary $\mathbb{S}_r(p)$.
	\end{lemma}
	
	\begin{proof}
		Given two radii $r>r'$, both in $\mathcal{G}_p$, if $r/r'$ is sufficiently close to $1$ we can build
		a deformation retraction
		\begin{equation}\label{eq:retrcrown}
		\rho:\overline{\mathbb{B}}_r(p)\setminus\mathbb{B}_{r'}(p)\to\mathbb{S}_r(p)\, .
		\end{equation}
		Indeed, the existence of a retraction $\rho$ follows from an obvious variant of \autoref{retraction}. Proceeding as in the proof of \autoref{prop:3ctoc}, we can find
		a homotopy between $\rho$ and the identity on $U:=\overline{\mathbb{B}}_r(p)\setminus\mathbb{B}_{r'}(p)$ (equal to the identity on $\mathbb{S}_r(p)$ at all times), using the local uniform contractibility that was obtained in \autoref{prop:3ctoc}. Finally,
		by composing the homotopy with a retraction onto $U$, we obtain a homotopy among maps $U\to U$. Thus $\rho$ is a deformation retraction.
		
		We then argue as in the proof of \autoref{lambda}.
		We consider a sequence of radii $r_0=r>r_1>\dots>r_j\to0$ with $r_j\in\mathcal{G}_p$ and $r_j/r_{j+1}$ sufficiently close to $1$ for each $j\in\mathbb{N}$. The existence of a sequence with these properties for $r\in \mathcal{G}_p$ sufficiently small follows from \autoref{cor:densegoodr}.
		These yield deformation retractions
		\begin{equation}
		\overline{\mathbb{B}}_r(p)\setminus\mathbb{B}_{r_{j+1}}(p)\to
       \overline{\mathbb{B}}_r(p)\setminus\mathbb{B}_{r_{j}}(p)\, ,
		\end{equation}
		for each $j\in\setN$.
		The infinite composition
		\begin{equation}
		\rho_0\circ\rho_1\circ\cdots\circ\rho_j\circ\cdots 
		\end{equation}
		is well-defined and gives the desired map.
	\end{proof}

	\begin{proof}[Proof of \autoref{prop:locrelhom}]
		By excision, for any $p\in X$ and for each $k\in\setN$ we have 
		\begin{equation}
		H_k(X,X\setminus\{p\})=H_k(\overline{\mathbb B}_r(p), \overline{\mathbb B}_r(p)\setminus\{p\})
		\end{equation}
		for a small Green-type ball $\mathbb{B}_r(p)$ such that \autoref{punct} applies. Since $\overline{\mathbb{B}}_r(p)$ is contractible by \autoref{prop:3ctoc}, $\overline{\mathbb{B}}_r(p)\setminus \{p\}$ deformation retracts onto $\mathbb{S}_r(p)$ by \autoref{punct}, and $\mathbb{S}_r(p)$ is homeomorphic to $S^2$ by \autoref{mfd.with.bdry}, by the long exact sequence of pairs it holds
		\begin{equation}
		 H_k(\overline{\mathbb B}_r(p),\overline{\mathbb B}_r(p)\setminus\{p\})\cong\tilde H_{k-1}(\overline{\mathbb B}_r(p)\setminus\{p\})\cong\tilde H_{k-1}(S^2)\, ,
		\end{equation}
		for each $k\in\setN$, where $\tilde H$ denotes the reduced homology. The latter homology group is isomorphic to $\Z$ for $k=3$ and to $0$ otherwise. Hence it is isomorphic to $H_k(\setR^3,\setR^3\setminus\{0\})$.
	\end{proof}

\section{Manifold recognition: completion of the proof of \autoref{thm:RCDtopma}}\label{sec:topman}

The goal of this section is to upgrade the conclusion of \autoref{prop:genmanrec} from generalized $3$-manifold to topological $3$-manifold and hence to complete the proof of \autoref{thm:RCDtopma}. 

\medskip

The strategy will be to further exploit the regularity and richness of good Green-balls and spheres to see that the resolution \autoref{thm:recog2Thick} and the recognition \autoref{thm:recog2DR} can be applied under the assumptions of \autoref{thm:RCDtopma}.

The proof has two main steps, which we outline under the additional assumption that $(X,\dist)$ is compact. The general case will be discussed below as it requires an additional argument.

In the first step we will show that any compact $\RCD(-2,3)$ space $(X,\dist,\haus^3)$ as in the assumptions of \autoref{thm:RCDtopma} admits a resolution in the sense of \autoref{def:resolvable}. The existence of the resolution follows from \autoref{thm:recog2Thick} via the following.

\begin{proposition}\label{prop:singGDPO}
Let $(X,\dist,\haus^3)$ be an $\RCD(-2,3)$ space such that all tangent cones have cross-section homeomorphic to $S^2$. Then the non-manifold set $\mathcal{S}_{\rm{top}}\subset X$ has general-position dimension one. 
\end{proposition}

The general-position dimension one property was introduced in \autoref{def:GPDO}, borrowing the terminology from \cite{Thickstunb}. 
The statement of \autoref{prop:singGDPO}  is completely nontrivial because after removing the non-manifold set (which has codimension $2$ a priori) $X$ might have a very complicated topology. The proof will hinge on \autoref{main.green.sphere} (vi) as we will prove that any continuous map $f:\overline{D}\to X$ can be slightly perturbed so as to have image contained in a finite union of Green-spheres in a neighbourhood of $\mathcal{S}_{\rm{top}}$.

\medskip

In the second step we shall see that the assumptions of the recognition \autoref{thm:recog2DR} are met by proving the following.

\begin{proposition}\label{prop:critDR}
Let $(X,\dist,\haus^3)$ be an $\RCD(-2,3)$ space such that all tangent cones have cross-section homeomorphic to $S^2$.
For every $p\in X$ there exist arbitrarily small neighbourhoods $U\ni p$ and homeomorphisms $f:S^2\to f(S^2)\subset U\setminus \{p\}$ such that:
\begin{itemize}
\item[(i)] $U\setminus \{p\}$ is simply connected;
\item[(ii)] $f:S^2\to U$ is homotopically trivial; 
\item[(iii)] $f:S^2\to U\setminus\{p\}$ is not homotopically trivial;
\item[(iv)] $f(S^2)$ is $1$-LCC (see \autoref{def:kcoco} for the relevant terminology). 
\end{itemize}
\end{proposition}

We shall see that a good choice for the maps $f$ as in the statement is given by the parametrizations of a family of good Green-spheres around each point.

\medskip

Once \autoref{prop:singGDPO} and \autoref{prop:critDR} are proven, the proof of \autoref{thm:RCDtopma} can be easily completed.
Indeed by \autoref{prop:genmanrec} $(X,\dist)$ is a generalized $3$-manifold. By \autoref{prop:singGDPO} and the resolution \autoref{thm:recog2Thick}, $(X,\dist)$ is resolvable. Note that this is the step of the argument where the compactness of $(X,\dist)$ is seemingly needed for the application of \autoref{thm:recog2Thick}.
By \autoref{prop:critDR} and the recognition \autoref{thm:recog2DR}, $(X,\dist)$ is a $3$-manifold.

\medskip

In the case where $(X,\dist)$ is not necessarily compact, we need a slight extension of the strategy outlined above. Roughly speaking, we are going to show that the very same statements hold for the spaces obtained by collapsing the Green-sphere of a good Green-ball to a point, for each point and each sufficiently small good radius, even though these spaces do not have any synthetic lower Ricci curvature bound.

\medskip

Given $p\in X$ and $r\in \mathcal{G}_p$ sufficiently small we shall consider the metric space $(\widehat{\mathbb{B}}_r(p),\widehat{\dist})$, where
\begin{equation}
\widehat{\mathbb{B}}_r(p):=\overline{\mathbb{B}}_r(p)/_{\sim}\, ,
\end{equation}
with $x \sim y$ if and only if $x=y\in \mathbb{B}_r(p)$ or $x,y\in \mathbb{S}_r(p)$, and $\widehat{\dist}$ is the usual quotient distance:
\begin{equation}
\widehat{\dist}([x],[y]):=\inf\sum_{j=1}^k\dist(x_j,y_j)\, ,
\end{equation}
where the infimum runs among all $k$-tuples of pairs $(x_j,y_j)$ such that $x_1\in [x]$, $y_k\in [y]$ and $y_j\sim x_{j+1}$ for each $j$. 
It is elementary to verify that $\widehat{\dist}$ is a distance giving a length metric space and inducing the quotient topology on $\widehat{\mathbb{B}}_r(p)$. 

We are going to argue that $\widehat{\mathbb{B}}_r(p)$ is a topological $3$-manifold for each $p\in X$ and every sufficiently small $r\in\mathcal{G}_p$. This statement will be enough to show that $X$ is a topological $3$-manifold. Indeed, the projection to the quotient $\pi: \overline{\mathbb{B}}_r(p)\to \widehat{\mathbb{B}}_r(p)$ is easily seen to be a homeomorphism with the image on $\mathbb{B}_r(p)$.

\medskip

With this aim, the main steps are given by the following three propositions, which are slight variants of \autoref{prop:genmanrec}, \autoref{prop:singGDPO} and \autoref{prop:critDR} respectively. 

\begin{proposition}\label{prop:Bhatgm}
Let $(X,\dist,\haus^3)$ be an $\RCD(-2,3)$ space such that all the cross-sections of all tangent cones are homeomorphic to $S^2$. For any $p\in X$ and for any $r\in \mathcal{G}_p$ sufficiently small, $(\widehat{\mathbb{B}}_r(p),\widehat{\dist})$ is a closed generalized $3$-manifold with empty boundary. 
\end{proposition}

\begin{proposition}\label{prop:gpdoBhat}
Under the same assumptions as in \autoref{prop:Bhatgm}, the non-manifold set $\mathcal{S}_{\rm{top}}(\widehat{\mathbb{B}}_r(p))\subset \widehat{\mathbb{B}}_r(p)$ has general-position dimension one.
\end{proposition}

\begin{remark}
In the proof of \autoref{prop:gpdoBhat} we will show that $\pi(\mathcal{S}_{\eps_0}^1\cup\mathbb{S}_r(p))\subset \widehat{\mathbb{B}}_r(p)$ has general-position dimension one. The conclusion about $\mathcal{S}_{\rm{top}}$ follows immediately by the inclusion
\begin{equation}
\mathcal{S}_{\rm{top}}(\widehat{\mathbb{B}}_r(p))\subset \pi(\mathcal{S}_{\eps_0}^1\cup\mathbb{S}_r(p))\, .
\end{equation}
\end{remark}

\begin{proposition}\label{prop:DavRepBhat}
Under the same assumptions as in \autoref{prop:Bhatgm} above, for every $q\in  \widehat{\mathbb{B}}_r(p)$ there exist arbitrarily small neighbourhoods $U\ni q$ and homeomorphisms $f:S^2\to f(S^2)\subset U\setminus \{q\}$ such that:
\begin{itemize}
\item[(i)] $U\setminus \{q\}$ is simply connected;
\item[(ii)] $f:S^2\to U$ is homotopically trivial; 
\item[(iii)] $f:S^2\to U\setminus\{q\}$ is not homotopically trivial;
\item[(iv)] $f(S^2)$ is $1$-coconnected (see \autoref{def:kcoco} for the relevant terminology). 
\end{itemize}
\end{proposition}

The rest of this section is dedicated to the proofs of \autoref{prop:Bhatgm}, \autoref{prop:gpdoBhat} and \autoref{prop:DavRepBhat}. Given these statements, the proof of \autoref{thm:RCDtopma} will be easily completed also in the non-compact case.
Indeed, by \autoref{prop:Bhatgm}, \autoref{prop:gpdoBhat} and \autoref{thm:recog2Thick}, $\widehat{\mathbb{B}}_r(p)$ is a resolvable generalized $3$-manifold. By \autoref{prop:DavRepBhat} and \autoref{thm:recog2DR}, it is then a $3$-manifold.

We note that \autoref{prop:DavRepBhat} is stronger than \autoref{prop:critDR}. Moreover, it will be clear that the argument proving \autoref{prop:gpdoBhat} proves also \autoref{prop:singGDPO}.
\medskip

\begin{proof}[Proof of \autoref{prop:Bhatgm}]
Given \autoref{prop:genmanrec}, it is sufficient to note that the following hold:
\begin{itemize}
\item[(i)] for each $r'<r$ such that $r'\in \mathcal{G}_p$ and $r'/r$ is sufficiently close to $1$ there exists a deformation retraction
\begin{equation}
		\rho:\overline{\mathbb{B}}_r(p)\setminus\mathbb{B}_{r'}(p)\to\mathbb{S}_r(p)\, ;
\end{equation}
\item[(ii)] for each $r'<r$ such that $r'\in \mathcal{G}_p$ and $r'/r$ is sufficiently close to $1$ we can construct deformation retractions
		\begin{equation}
		{\mathbb{B}}_r(p)\setminus{\mathbb{B}}_{r'}(p)\to\mathbb{S}_{r'}(p)\, .
		\end{equation}
\end{itemize}
Property (i) was obtained during the proof of \autoref{punct}: see in particular \eqref{eq:retrcrown}. 

\medskip

By (i), $\pi(\overline{\mathbb{B}}_r(p)\setminus\mathbb{B}_{r'}(p))\subset \widehat{\mathbb{B}}_r(p)$ is contractible for each $r'<r$ such that $r'\in \mathcal{G}_p$ and $r'/r$ is sufficiently close to $1$. Hence, if we set $\{\widehat{x}\}:=\pi(\mathbb{S}_r(p))$, for each neighbourhood $\widehat{U}\subset \widehat{\mathbb{B}}_r(p)$ of $\widehat{x}$ there exists a contractible open set $V\subset U$ with $\widehat{x}\in V$. In particular, $\widehat{\mathbb{B}}_r(p)$ is locally contractible.

\medskip

Analogously, by (ii) we can argue as in the proof of \autoref{prop:locrelhom} and verify that 
\begin{equation}
H_{*}(\widehat{\mathbb{B}}_r(p),\widehat{\mathbb{B}}_r(p)\setminus \{\widehat{x}\})=H_*(\setR^3,\setR^3\setminus\{0\})\, .
\end{equation}
The combination of these two properties shows that $\widehat{\mathbb{B}}_r(p)$ is a generalized $3$-manifold without boundary
(since $\pi$ is Lipschitz, $\widehat{\mathbb{B}}_r(p)$ has Hausdorff dimension 3, and hence covering dimension 3).
\end{proof}

\begin{remark}\label{rm:goodshat}
The sets $\pi(\overline{\mathbb{B}}_r(p)\setminus\mathbb{B}_{r'}(p))\subset \widehat{\mathbb{B}}_r(p)$ and $\pi(\mathbb{S}_{r'}(p))\subset \widehat{\mathbb{B}}_r(p)$ for $r'\in\overline{\mathcal{G}}_p$ sufficiently close to $r$ and $r'<r$ should be thought of as a replacement of good Green-balls and Green-spheres, respectively, centered at the point $\widehat{x}\in \widehat{\mathbb{B}}_r(p)$. Below they will be exploited in the same way without further notice, with a slight abuse of terminology.
\end{remark}

For the proof of \autoref{prop:gpdoBhat} it will be helpful to know that the non-manifold set of $\mathcal{S}_{\rm{top}}(\widehat{\mathbb{B}}_r(p))$ does not disconnect $\widehat{\mathbb{B}}_r(p)$. This is a consequence of the well-known statement that the codimension $2$ singular set of any $\RCD(K,n)$ space $(X,\dist,\haus^n)$ does not disconnect it: see \cite[Theorem 3.8]{CheegerColding00II} or \cite[Appendix A]{KapovitchMondino}. However, we present a short proof tailored for the present setting for the sake of completeness; recall that $\mathcal{S}_{\eps_0}^1$ is the complement of $\mathcal{R}_{\eps_0}$.

	\begin{lemma}\label{not.disconnect}
		Under the same assumptions of \autoref{prop:Bhatgm}, let $U\subseteq \widehat{\mathbb{B}}_r(p)$ be a connected open set. 
		Then no closed set $K\subseteq\pi(\mathcal{S}_{\eps_0}^1\cup\mathbb{S}_r(p))$ disconnects $U$. In particular, also $\mathcal{S}_{\rm{top}}(\widehat{\mathbb{B}}_r(p))$ does not disconnect $U$.
	\end{lemma}
	
	\begin{proof}
		Given two distinct points $q,q'$ in $U\setminus K$, let $U_q$ denote the (open) connected component of $U\setminus K$ and define $U_{q'}$ similarly. In this proof, we replace $\widehat{\dist}$ with a metric on $U$ (inducing the same topology, still denoted by $\widehat{\dist}$) such that closed balls are compact and $(U,\widehat{\dist})$ is still a length metric space.
		
		Let $\delta:=\widehat{\dist}(q',U_q)$. If $\delta=0$ then $U_q$ intersects arbitrarily small balls $B_s(q')$, and thus $U_q=U_{q'}$ (as $B_s(q')\subseteq U_{q'}$ for $s$ small enough), as desired. Assume by contradiction that $\delta>0$. By compactness of $\overline B_{2\delta}(q')$ we can find $\overline{q}\in\de U_q$ such that $\widehat{\dist}(\overline{q},q')=\delta$.
		Let us now take a small good radius $s\in\mathcal{G}_q\cap(0,\delta)$ such that $q,q'\notin\overline{\mathbb{B}}_s(\overline{q})$, and $\mathbb{S}_s(\overline{q})\not\ni\hat x$ is homeomorphic to $S^2$. Since $b_{\overline{q}}(U_q)$ is connected
		and its closure contains $0$ and $b_{\overline{q}}(q)>s$, we have $s\in b_{\overline{q}}(U_q)$. Thus, $\mathbb{S}_s(\overline{q})\cap U_q\neq \emptyset$.
		Moreover, $\Sigma:=\mathbb{S}_s(\overline{q})\cap(\hat{\mathbb{B}}_r(p)\setminus K)$ is homeomorphic to a sphere with finitely many points removed. Hence it is connected, giving $\Sigma\subseteq U_q$. On the other hand, $\mathbb{S}_s(\overline{q})$ (and hence $\Sigma$) contains points where $\widehat{\dist}(\cdot,q')<\delta$. For instance,
		this holds at the intersection between $\mathbb{S}_s(\overline{q})$ and a short geodesic between $\overline{q}$ and $q'$. This property contradicts the fact that $\Sigma\subseteq U_q$, where it holds $\hat{\dist}(\cdot,q')\ge\delta$. 
	\end{proof}

	\begin{proof}[Proof of \autoref{prop:gpdoBhat}]
		In order to shorten the notation let $D:=D^2$, $Y:=\widehat{\mathbb{B}}_r(p)$ and $\mathcal{S}_{\rm{top}}:=\mathcal{S}_{\rm{top}}(Y)$. We need to prove that any continuous map $f:\overline D\to Y$ can be approximated with continuous maps $g:\overline D\to Y$ (i.e., with $\sup_{x\in\overline D}\widehat{\dist}(f(x),g(x))$ arbitrarily small)
		such that $g(\overline D)\cap\mathcal{S}_{\rm{top}}$ is finite. The key idea will be to perturb the original map in such a way that the image of the perturbation is contained in a finite union of good Green-spheres in a neighbourhood of the singular set. This will be sufficient to conclude by \autoref{main.green.sphere} (iii).
	\medskip
		
		There are three main steps. We let $U\subset \widehat{\mathbb{B}}_r(p)$ be any open and connected set such that $f(\overline{D})\subseteq U$. In Step 1 we reduce to the case where $f(\de D)\cap\mathcal{S}_{\rm{top}}=\emptyset$. In Step 2 we prove that for any such $f$ there exists $g:\overline{D}\to U$ with $g=f$ on $\partial D$ and $g(\overline D)\cap\mathcal{S}_{\rm{top}}$ finite. It is important to note that the perturbation obtained in this step is localized to $U$ in the target. In Step 3 we complete the proof by applying a ``scale-invariant'' version of Step 2 to the restriction of $f$ to all $2$-simplices in a sufficiently fine triangulation of $D^2$.
	\medskip
	
	\emph{Step 1:}	
		We claim that for any continuous function $f:\overline{D}\to U$ as above there is a continuous map $f':\overline{D}\to U$ arbitrarily close to $f$ such that $f'(\de D)\cap \mathcal{S}_{\rm{top}}=\emptyset$.

		Indeed, we can find a loop in $U\setminus\mathcal{S}_{\rm{top}}$
		arbitrarily close to $f|_{\de D}$ subdividing $\de D$ into many small arcs, mapping each endpoint $x$ to a point in $U\setminus\mathcal{S}_{\rm{top}}$ close to $f(x)$,
		and connecting the images of two consecutive endpoints with a small arc in $U\setminus\mathcal{S}_{\rm{top}}$. The existence of this arc follows from \autoref{not.disconnect}.
		Since $U$ is locally contractible by \autoref{prop:Bhatgm}, the new loop is homotopic to $f|_{\de D}$ (provided that they are sufficiently close). We can extend $f$ to a map defined on a slightly larger disk $\overline D'$,
		given by this homotopy on $\overline D'\setminus D$. The claim follows by rescaling $\overline D'$ back to $\overline D$.

		\medskip

		\emph{Step 2:}
		We claim that for every continuous $f:\overline{D}\to U$ such that $f(\partial D)\cap \mathcal{S}_{\rm{top}}=\emptyset$ there exists a continuous map $g:\overline{D}\to U$ such that $f|_{\de D}=g|_{\de D}$ and $g(\overline{D})\cap \mathcal{S}_{\rm{top}}$ is finite.

		Indeed, we can first cover $f(\overline D)\cap\mathcal{S}_{\rm{top}}$ with finitely many good Green balls $\overline{\mathbb{B}}_i \subset U$
		disjoint from $f(\de D)$. We shall denote by $p_i$ and $r_i$ the centers and radii of the good Green balls $\mathbb{B}_i$ respectively. Arguing as in the previous step and exploiting \autoref{punct} to check that sufficiently small punctured Green-balls are simply connected, we claim that we can replace $f$ with a map $\tilde{f}$ with the same boundary values and the additional property that $p_i\notin\tilde{f}(\overline{D})$ for all $i$. 

        Indeed, we can select arbitrarily small, disjoint Green-balls $\mathbb{B}_i'$ centered at $p_i\in \mathcal{S}_{\rm{top}}$.
        For $i=1$, we can find smooth loops $\gamma_{ij}\subset f^{-1}(\mathbb{B}_i')$
        enclosing disjoint regions $D_{ij}\subset D$, each homeomorphic to $\bar D$, such that $f^{-1}(p_i)\subset\bigcup_j\operatorname{int}(D_{ij})$. In particular, $f|_{\gamma_{ij}}$ takes values in $\bar{\mathbb{B}}_i'\setminus\{p_i\}$, which is simply connected, and hence we can replace $f|_{D_{ij}}$ with a map with values in the same punctured ball, showing the claim. We repeat this for all $i$.
        The new map $\tilde f$ takes values in $U$, but could be far from $f$ since we could have $f(D_{ij})\not\subseteq \mathbb{B}_i'$.
        
        Up to slightly perturbing the radii of the Green-balls $\mathbb{B}_i$ we can assume that $p_i\notin \bigcup_j \mathbb{S}_j$ for each $i$.

		By \autoref{punct} there exists a retraction $\rho_i:\overline{\mathbb{B}}_i\setminus\{p_i\}\to\mathbb{S}_i$ for each $i$, which obviously extends to a retraction $\rho_i: Y\setminus\{p_i\} \to Y\setminus \mathbb{B}_i$. To establish the claim it suffices to compose $\tilde f$ with all these retractions. Indeed, after the first composition $\tilde f_1 := \rho_1 \circ \tilde f: \overline{D} \to X\setminus \mathbb{B}_1$, we still have $p_i \notin \tilde f_1(\overline{D})$ for $i\ge 2$, since $p_i\notin \bigcup_j \mathbb{S}_j$ for each $i$. So we can iterate and compose with all the other retractions obtaining a map $g: \overline{D}\to X$. By construction, $g$ takes values in $[g(\overline D)\setminus\bigcup_i\overline{\mathbb{B}}_i]\cup\bigcup_i\mathbb{S}_i$. However, $\mathbb{S}_i\cap \mathcal{S}_{\rm{top}}$ is a finite set for each $i$ by \autoref{main.green.sphere} (iii).

\medskip

\emph{Step 3:}
		We claim that the map $g$ as above can be chosen to be arbitrarily close to $f$.
		There is no harm in replacing $\overline D^2$ with the closed square $[0,1]^2$. For every $\eps>0$ we can divide $[0,1]^2$ into $\ell^2$ squares, with $\ell>2$ large enough so that 
		\begin{equation}
		f\left(\left[\frac{j}{\ell},\frac{j+1}{\ell}\right]\times\left[\frac{j'}{\ell},\frac{j'+1}{\ell}\right]\right)\subset B_{\eps}(p_{jj'})
		\end{equation}
		for some $p_{jj'}\in X$, for all $j,j'=0,\dots,\ell-1$.
		By repeating the argument in Step 1, we can assume that $f$ maps the 1-skeleton of the subdivision to a subset of $Y\setminus\mathcal{S}_{\rm{top}}$ (as we did at the beginning of the proof, up to replacing $f$ with a close map).
  
	By applying the argument in Step 2 to each restriction 
	\begin{equation}
	f_{jj'}:=f\Big|_{\left[\frac{j}{\ell},\frac{j+1}{\ell}\right]\times\left[\frac{j'}{\ell},\frac{j'+1}{\ell}\right]}:\left[\frac{j}{\ell},\frac{j+1}{\ell}\right]\times\left[\frac{j'}{\ell},\frac{j'+1}{\ell}\right]\to B_{\epsilon}(p_{j,j'})\, ,
	\end{equation}
		we can find continuous maps $g_{jj'}$ that coincide with the original maps $f_{jj'}$ on the boundaries of their domains such that $\widehat{\dist}(f_{jj'},g_{jj'})\le2\eps$ and the image of $g_{jj'}$ intersects $\mathcal{S}_{\rm{top}}$ in finitely many points, for any $j,j'$. The maps $g_{jj'}$ replacing each restriction
		glue together to a map $g$ defined on $[0,1]^2$, and clearly $\widehat{\dist}(g,f)\le 2\eps$.
	\end{proof}

\begin{proof}[Proof of \autoref{prop:DavRepBhat}]
We claim that suitable small Green-type balls and Green-type spheres around any point satisfy all the conditions in the statement. We will discuss the case of points $q\in Y:=\widehat{\mathbb{B}}_r(p)$ with $q\neq \widehat{x}$, where we denote $\widehat{x}=\pi(\mathbb{S}_r(p))$, as above. The remaining case can be dealt with completely analogous considerations: see \autoref{rm:goodshat}.
 
	\medskip
	
		Let $q\in Y$ be fixed and let us consider two sufficiently small radii $s,t\in\mathcal{G}_q$ with $s<t$. Let $U:=\mathbb{B}_t(q)$ and consider $\Sigma:=\mathbb{S}_s(q)$. By \autoref{prop:Gballscongm}, $\Sigma$ is homeomorphic to $S^2$ and we let $f:S^2\to \Sigma$ be any homeomorphism.
		By a variant of \autoref{punct}, $U\setminus \{q\}$ is simply connected, as it deformation retracts onto $\mathbb{S}_s(q)$. 	
		Moreover, $f$ is homotopically trivial in $U$, since $\overline{\mathbb{B}}_s(q)$ is contractible by \autoref{prop:Gballscongm},
		but not in $U\setminus\{q\}$, and in fact not even in $\overline U\setminus\{q\}$.
		Indeed, the proof of \autoref{punct} shows that the inclusion $\overline{\mathbb{B}}_s(p)\setminus\{q\}\hookrightarrow\overline{\mathbb{B}}_t(q)\setminus\{q\}$
		induces an isomorphism on the homology groups and that $[\Sigma]\in H_2(\overline{\mathbb{B}}_s(q)\setminus\{q\})$ is not trivial,
		implying that it is not trivial in $H_2(\overline{\mathbb{B}}_t(q)\setminus\{q\})$ as well.

\medskip

The discussion above proves that conditions (i), (ii) and (iii) in the statement are met with these choices. We claim that condition (iv) is also met.

Indeed, for any $t\in \mathcal{G}_q$ we can find $a,b\in\mathcal{G}_q$ arbitrarily close to $t$ with $a<t<b$. With similar arguments as in the proof of \autoref{punct} above, we can construct deformation retractions
		\begin{equation}
		{\mathbb{B}}_t(q)\setminus\overline{\mathbb{B}}_a(q)\to\mathbb{S}_{a'}(q)\, ,\quad{\mathbb{B}}_b(q)\setminus\overline{\mathbb{B}}_t(q)\to\mathbb{S}_{b'}(q)\, ,
		\end{equation}
		for intermediate levels $a'\in(a,t)\cap\mathcal{G}_q$ and $b'\in(t,b)\cap\mathcal{G}_q$.
		Hence, the annular regions ${\mathbb{B}}_t(q)\setminus\overline{\mathbb{B}}_a(q)$ and ${\mathbb{B}}_b(q)\setminus\overline{\mathbb{B}}_t(q)$ are simply connected. It follows that $\Sigma$ is $1$-LCC, as we claimed.

	\end{proof}

\section{Proof of the topological results for Ricci limits}
\label{sec:Proofmaintheorems}

The goal of this section is to conclude the proof of the topological results for Ricci limit spaces,  \autoref{cor:4d}, \autoref{thm:omeo n-4 sym}, and \autoref{thm:top n-3}, as well as to prove the uniform local contractibility of noncollapsed $\RCD(-2,3)$ topological manifolds, as stated in \autoref{thm:unifcontrtopRCD3}. Additionally, we will demonstrate \autoref{thm:volumegrowthR^3} along the way.

\medskip

The key tools will be the rigidity of $(n-3)$-symmetric cones that are noncollapsed Ricci limits, as in \autoref{thm:noRP2}, and the manifold recognition \autoref{thm:RCDtopma}.

\subsection{Manifold structure of cross-sections and sections}

We begin by proving the that cross-sections of $(n-4)$-symmetric cones are topological manifolds.

\begin{theorem}\label{thm:topreg}
Let $(M_i^n, g_i, p_i)\to (X^n,\dist, p)$ be a noncollapsed Ricci limit space with $n\ge 4$. Assume that $X^n = \setR^{n-4}\times C(Z^3)$ is an $(n-4)$-symmetric cone. Then $(Z^3,\dist_Z)$ is homeomorphic to a topological $3$-manifold whose universal cover is homeomorphic to $S^3$. 
\end{theorem}

\begin{proof}

In view of \autoref{thm:RCDtopma} and the solution to the Poincar\'e conjecture, it is enough to check that all the tangent cones of $Z$ have the cross-section homeomorphic to $S^2$. Indeed, if this is the case, then $Z$ is a $3$-manifold by \autoref{thm:RCDtopma}. Moreover, the universal cover $(\tilde Z, \dist_{\tilde Z})$ of $Z$ is a simply connected, noncollapsed $\RCD(2,3)$ space (see \cite{MondinoWei}) which is a topological $3$-manifold. Since $\RCD(2,3)$ spaces are compact, the solution of the Poincar\'e conjecture implies that $\tilde Z$ is homeomorphic to $S^3$.

\medskip

We now verify that, for every $z \in Z$, all the tangent cones at $z$ have a cross-section homeomorphic to $S^2$. Let $C(\Sigma)$ be a tangent cone at $z \in Z$. It is easy to check that there exist scaling factors $r_i \to 0$ and a sequence of points $x_i \in M_i^n$ such that
\begin{equation}
(M_i^n, r_i^{-1} g_i, x_i) \to \R^{n-3} \times C(\Sigma) \, 
\end{equation}
as $i\to\infty$. The conclusion then follows from \autoref{thm:noRP2}.
\end{proof}

With an analogous argument we can prove \autoref{thm:top n-3}, whose statement is repeated below for the ease of readability.

\begin{theorem}\label{thm:top n-3 main}
Let $(M_i^n, g_i, p_i)\to (X^n,\dist, p)$ be a noncollapsed Ricci limit space with $n\ge 3$. 
Assume that $X^n =\setR^{n-3}\times Z^3$ as metric measure spaces.
Then $Z^3$ is homeomorphic to a topological $3$-manifold. 

\end{theorem}

\begin{proof}
If $\Ric_i\ge -(n-1)$ for each $i\in\setN$, then $(X,\dist,\haus^n)$ is an $\RCD(-(n-1),n)$ space. By \cite[Proposition 2.15]{DePhilippisGigli2}, $(Z,\dist_Z,\haus^3)$ is an $\RCD(-(n-1),3)$ space. With the very same argument as in the proof of \autoref{thm:topreg} it is possible to verify that the cross-section of every tangent cone at every point $z\in Z^3$ is homeomorphic to $S^2$. The conclusion follows from \autoref{thm:RCDtopma}.    
\end{proof}

\begin{remark}
 We note that the conclusion that $Z^3$ is homeomorphic to a $3$-manifold in \autoref{thm:top n-3 main} above is strictly stronger than the conclusion that $X^n$ is homeomorphic to an $n$-manifold. For instance, Bing constructed an example of a topological space $B$ such that $B\times \setR$ is homeomorphic to $\setR^4$ even though $B$ is not a topological $3$-manifold \cite{Bingproduct}.
\end{remark}

\subsection{$\RCD(0,3)$ manifolds with Euclidean volume growth}

In this section, we prove \autoref{thm:volumegrowthR^3}, which we restate below for clarity.

\begin{theorem}\label{thm:RCD03AVR}
Let $(X^3,\dist,\haus^3)$ be an $\RCD(0,3)$ manifold with Euclidean volume growth, i.e., there exist $p\in X$ and $v>0$ such that $\haus^3(B_r(p))\ge vr^3$ for any $r>0$. Then $X^3$ is homeomorphic to $\setR^3$. 
\end{theorem}

Then, by leveraging on cone rigidity (see \autoref{thm:noRP2}), we have the following corollary.

\begin{corollary}\label{cor:secR3AVR}
Let $(M_i^n, g_i ,p_i)\to (X^n,\dist,\haus^n,p)$ be a noncollapsed Ricci limit space with $\Ric_i\ge -1/i$ for any $i\in\setN$. Assume that $X$ has Euclidean volume growth and is $(n-3)$-symmetric, i.e., $X=\setR^{n-3}\times Z^3$.
Then $Z^3$ is homeomorphic to $\setR^3$. 
\end{corollary}

\begin{proof}
By \autoref{thm:top n-3 main}, $Z^3$ is a topological $3$-manifold. Hence, taking into account also \cite[Proposition 2.15]{DePhilippisGigli2}, $(Z^3,\dist_Z,\haus^3)$ is an $\RCD(0,3)$ manifold with Euclidean volume growth. The conclusion follows from \autoref{thm:RCD03AVR}.
\end{proof}

\begin{remark}\label{rm:cfZhu}
We mention that \autoref{thm:unifcontrtopRCD3main} below and the conclusion that $X^3$ is contractible under the assumptions of \autoref{thm:RCD03AVR} should follow also from the arguments in \cite{Zhu93} where the analogous statements are proved for smooth Riemannian manifolds. The arguments we present below are more in the spirit of the other parts of the paper and lead to stronger topological control.  
\end{remark}

\begin{proof}[Proof of \autoref{thm:RCD03AVR}]
We claim that any $\RCD(0,3)$ manifold $(X^3,\dist,\haus^3)$ with Euclidean volume growth admits an exhaustion $X=\bigcup_i U_i$ where each $U_i$ is homeomorphic to the $3$-ball $\overline{D}^3$. 
Then the main theorem of \cite{Brown} applies and proves that $X$ is homeomorphic to $\setR^3$. 

\medskip

The remaining part of the proof is dedicated to the verification of the claim above. We fix a point $p\in X$, let $G_p:X\setminus\{p\}\to (0,\infty)$ be the Green function of the Laplacian with pole at $p$, and let $b_p$ be the induced Green-type distance (see Section \ref{sec:Green}). A family of domains $U_i$ with the properties claimed above is given by sub-level sets $\{b_p\le r_i\}$ for a suitably chosen sequence $r_i\to\infty$. Indeed, by \autoref{cor:densegoodr} applied to rescalings of $X$, there exists a sequence of good radii $r_i\to \infty$. By \autoref{contractible}, all the closed Green-balls $\overline{\mathbb{B}}_{r_i}(p)$ are $3$-manifolds with boundary, which by \autoref{poinc.ball} are homeomorphic to $\overline{D}^3$, as we claimed.
\end{proof}

\subsection{Uniform local contractibility}\label{sec:unifcontr}

The goal of this section is to establish the uniform local contractibility result stated in \autoref{thm:unifcontrtopRCD3}, which is restated below for clarity. The proof will rely on \autoref{contractible} and \autoref{cor:densegoodr}.

\begin{theorem}\label{thm:unifcontrtopRCD3main}
Let $v>0$ be fixed. There exist constants $C=C(v)>0$ and $\rho=\rho(v)>0$ such that if $(X,\dist,\haus^3)$ is an $\RCD(-2,3)$ manifold with $\haus^3(B_1(p))\ge v$ for any $p\in X$, then the ball $B_r(p)$ is contractible inside $B_{Cr}(p)$ for every $r<\rho$ and for every $p\in X$.
\end{theorem}

\begin{remark}[Comparison with \autoref{thm:RCD03AVR}]
    The uniform local contractibility of noncollapsed $\RCD(-2,3)$ spaces should be interpreted as a localized version of the assertion that $\mathrm{RCD}(0,3)$ spaces with Euclidean volume growth are contractible. In this context, \autoref{thm:RCD03AVR} yields a stronger conclusion by providing a homeomorphism with $\R^3$. As mentioned before (see \autoref{rm:cfZhu}), it should be possible to obtain an alternative proof of \autoref{thm:unifcontrtopRCD3main} by following the arguments in \cite{Zhu93}.
\end{remark}

\begin{remark}
It is possible to achieve the same conclusion of \autoref{thm:unifcontrtopRCD3main} independently of the results of Section \ref{sec:topman}, still under the assumption that all tangent cones of $(X,\dist)$ have cross-section homeomorphic to $S^2$. 
Indeed, using \autoref{prop:Gballscongm} in place of \autoref{contractible}, we obtain that the space $X$ admits an exhaustion into open sets with contractible closure. This guarantees that $X$ is itself contractible by Whitehead's theorem, since all the homotopy groups of $X$ are trivial and $X$ is an absolute neighbourhood retract; alternatively, we can apply the main result of \cite{AncelEdwards}.
\end{remark}

\begin{proof}[Proof of \autoref{thm:unifcontrtopRCD3main}]
It is sufficient to show that there exist constants $C=C(v)>0$ and $\rho=\rho(v)>0$ such that if $(X,\dist,\haus^3)$ is an $\RCD(-2,3)$ manifold with $\haus^3(B_1(p))\ge v$ for every $p\in X$, then for every $0<r<\rho$ and for every $p\in X$ there exists $U\subset X$ such that $U$ is contractible and $B_r(p)\subset U\subset B_{Cr}(p)$. If this is the case, then the inclusion of $B_r(p)$ into $B_{Cr}(p)$ induces the trivial map in homotopy. 

\medskip

As in the proof of \autoref{thm:RCD03AVR}, the sought domain is going to be a suitably chosen good Green-ball. Indeed, we can fix $\eta_0 = \eta_0(v)$ and $\delta_0 = \delta_0(v)$ such that any Green-ball $\mathbb{B}_r(p)$ with $r\in \mathcal{G}_p(\eta_0,\delta_0)$ in an $\RCD(-2,3)$ space $(X,\dist,\haus^3)$ with $\haus^3(B_1(p))\ge v$ for all $p\in X$ verifies the conclusion of \autoref{contractible}. Then we can choose $C=C(v)=C(\eta_0(v), \delta_0(v), v)$ as given by \autoref{cor:densegoodr}. For each $0 < r < \rho_0:=\delta_0^2/C$ there exists $r'\in (r,Cr) \cap \mathcal{G}_p$, and $r'$ can be chosen so that $B_r(p)\subset \overline{\mathbb{B}}_r(p)\subset B_{Cr}(p)$. 
\end{proof}

\subsection{Topological stability: proof of \autoref{thm:topstabintro}}

The goal of this section is to prove \autoref{thm:topstabintro}, whose statement is repeated below for the sake of clarity.

\begin{theorem}[Topological stability]\label{thm:topstabmain}
Let $(X,\dist,\haus^3)$ be a compact $\RCD(-2,3)$ space which is a topological manifold. There exists $\epsilon=\epsilon(X)>0$ such that if $(Y,\dist_Y,\haus^3)$ is an $\RCD(-2,3)$ space which is a topological manifold and $\dist_{\rm{GH}}(X,Y)<\epsilon$, then $X$ and $Y$ are homeomorphic.    
\end{theorem}

The proof of \autoref{thm:topstabmain} will be based on
\autoref{thm:unifcontrtopRCD3}, the homotopic stability results from \cite{Petersen90}, and some abstract results in geometric topology from \cite{Jakobsche}. 

We record below some of the relevant terminology and the results we will rely on.

\begin{definition}[$\eps$-equivalence]
Given metric spaces $(X,\dist_X)$ and $(Y,\dist_Y)$ and a continuous map $f:X\to Y$, we say that $f$ is an \emph{$\eps$-equivalence} if there exists a continuous map $g:Y\to X$ with the following properties. There exist homotopies $F$ and $G$ of $f\circ g$ and $g\circ f$ with the identities of $Y$ and $X$, respectively, such that the $F$-flow line of any point in $Y$ and the image through $f$ of the $G$-flow line of any point in $X$ have diameter less than $\eps$ in $Y$.
\end{definition}

The following statement is borrowed from \cite{Petersen90}.
The notion of dimension used in the statement is the Lebesgue covering dimension.

\begin{theorem}\label{thm:petersen}
Fix $n\in\setN$ and $\eps>0$. Let $\mathcal{F}$ be a family of uniformly contractible $n$-dimensional metric spaces. Then there exists $\delta=\delta(\eps,n)>0$ such that if $X,Y\in\mathcal{F}$ satisfy $\dist_{\mathrm{GH}}(X,Y)\le \delta$, then they are $\eps$-equivalent. 
\end{theorem}

The following is usually referred to as \emph{$\alpha$-approximation theorem}. In the 3-dimensional case it is due to \cite{Jakobsche}. We remark that the original statement in \cite{Jakobsche} had the validity of the Poincar\'e conjecture among the assumptions.

\begin{theorem}\label{thm:alphaapprox}
Let $(X,\dist)$ be a closed topological 3-manifold. For any $\alpha>0$ there exists $\eps=\eps(\alpha,X)>0$ such that if $(Y,\dist_Y)$ is a closed topological 3-manifold and $f:Y\to X$ is an $\eps$-equivalence, then there exists a homeomorphism $f':Y\to X$ such that $\dist_X(f,f')<\alpha$.
\end{theorem}

\begin{proof}[Proof of \autoref{thm:topstabmain}]
 By volume convergence, Bishop--Gromov and \autoref{thm:unifcontrtopRCD3main}, if $\overline{\epsilon}>0$ is sufficiently small then the family $\mathcal{F}$ of all $\RCD(-2,3)$ spaces $(Y,\dist_Y,\haus^3)$ that are topological $3$-manifolds and such that $\dist_{\rm{GH}}(X,Y)<\overline{\epsilon}$ is a family of uniformly contractible $3$-dimensional metric spaces. The conclusion follows combining \autoref{thm:petersen} and \autoref{thm:alphaapprox}.
\end{proof}

\subsection{Homeomorphism of (iterated) cross-sections}

In this section, we conclude the proof of \autoref{thm:omeo n-4 sym}, restated below for clarity.

\begin{theorem}\label{thm:omeo n-4 sym 2}
Let $(X^n,\dist)$ be a noncollapsed Ricci limit space of dimension $n\ge 4$.
\begin{itemize}
    \item[(i)] If $\R^{n-4}\times C(Z^3)$ is an $(n-4)$-symmetric tangent cone at $x\in X$, then $(Z^3,\dist_Z)$ is homeomorphic to a topological $3$-manifold whose universal cover is $S^3$.

    \item[(ii)] If all tangent cones at $x\in X$ are $(n-4)$-symmetric, i.e., each is isometric to $\mathbb{R}^{n-4}\times C(Z)$, then all the cross-sections $Z$ must be homeomorphic to each other.
\end{itemize}
\end{theorem}

\begin{remark}
\autoref{thm:omeo n-4 sym 2} implies that if all tangent cones at $x$ are $(n-4)$-symmetric, then they are all homeomorphic. However, the statement is strictly stronger as the following example illustrates. We let $Z_1:=S^3$ be the $3$-sphere endowed with the round metric of constant sectional curvature equal to $1$ and $Z_2:=\Sigma^3$ be the Poincar\'e homology sphere endowed with a metric of constant sectional curvature equal to $1$. Notice that $Z_1$ is not homeomorphic to $Z_2$. On the other hand, $\setR^2\times C(Z_1)=\setR^6$ is homeomorphic to $\setR^2\times C(Z_2)$. Indeed, denoting by $S^2\Sigma^3$ the double suspension over $\Sigma^3$, we have $\setR^2\times C(Z_2)=C(S^2\Sigma^3)$ and, by the double suspension theorem, $S^2\Sigma^3$ is homeomorphic to $S^5$. Hence, $\setR^2\times C(Z_2)$ is homeomorphic to $\setR^6$ as well.
\end{remark}

The proof of  \autoref{thm:omeo n-4 sym} will be based on \autoref{thm:topreg} and \autoref{thm:topstabmain}.
We are going to rely on the fact that the set of (possibly iterated) cross-sections of tangent cones at a given point of an $\RCD(K,n)$ space $(X,\dist,\haus^n)$ is compact and connected. Moreover, the volume is constant on this set. The statement is certainly known to experts, although it does not appear in the literature in this form. The argument is classical: see for instance the proofs of \cite[Theorem 4.2]{CheegerJiangNaber} or \cite[Lemma 2.1]{LytchakStadler}.

\begin{lemma}\label{lemma:crossconn}
Let $(X,\dist,\haus^n)$ be an $\RCD(K,n)$ metric measure space and let $x\in X$ be given. Let $0\le k\le n-2$ be such that all tangent cones at $x$ are $k$-symmetric, i.e., they are isometric to $\setR^k\times C(Z)$ for some metric space $(Z,\dist_Z)$.
Let $\mathcal{C}^k_x$ be the collection of all such iterated cross-sections $(Z,\dist_Z)$. Then $\mathcal{C}_x^k$ is a compact and connected subset of the class of $\RCD(n-k-2,n-k-1)$ metric measure spaces (endowed with the topology induced by the Gromov--Hausdorff distance). Moreover, the total volume $\haus^{n-k-1}$ is constant on $\mathcal{C}^k_x$.
\end{lemma}

\begin{proof}
    Precompactness of $\mathcal{C}_x^k$, as well as the fact that each of its elements is an $\RCD(n-k-2,n-k-1)$ space, are well known; the latter statement follows from \cite{Gigli13,Ketterer15}.
    Moreover, if $Z_i\to Z$ in the Gromov--Hausdorff sense and each $Z_i\in\mathcal{C}_x^k$,
    then each $\setR^k\times C(Z_i)$
    is a tangent cone. Since these converge to $\setR^k\times C(Z)$, a diagonal argument shows that the latter is also a tangent cone, establishing compactness of $\mathcal{C}_x^k$.

We let $\mathcal{C}_x$ be the set of all cross-sections of tangent cones at $x$. The proof of \cite[Theorem 4.2]{CheegerJiangNaber} generalizes verbatim to the present setting. Hence $\mathcal{C}_x$ is a compact and connected set, with respect to the topology induced by the Gromov--Hausdorff distance. Moreover, the volume $\haus^{n-1}$ is constant on $\mathcal{C}_x$, by Bishop--Gromov.

We consider the map $\Phi_k:\mathcal{C}^k_x\to \mathcal{C}_x$ defined by $\Phi_k(Z):=S^kZ$, where $S^kZ$ denotes the $k$-times iterated spherical suspension over $Z$. Under the present assumptions, $\Phi_k$ is a surjective map. Moreover, by \autoref{susp} below, it is easily seen that $\Phi_k$ is injective and continuous. Since $\mathcal{C}_x^k$ and $\mathcal{C}_x$ are compact, they are homeomorphic. The connectedness of $\mathcal{C}_x^k$ follows.  

In order to complete the proof, it is sufficient to note that $\mathcal{H}^{n-1}(S^kZ)=\mathcal{H}^{n-1}(S^kZ')$ if and only if $\haus^{n-k-1}(Z)=\haus^{n-k-1}(Z')$, for any $Z,Z'\in \mathcal{C}_x^k$. The constancy of $\haus^{n-k-1}$ on $\mathcal{C}_x^k$ follows from the constancy of $\haus^{n-1}$ on $\mathcal{C}_x$.
\end{proof}


\begin{lemma}\label{susp}
    Given two spaces $(Z,\dist_Z)$ and $(Z',\dist_{Z'})$
    as above, if $S^kZ$ is isometric to $S^kZ'$ then $Z$ is isometric to $Z'$.
\end{lemma}

\begin{proof}
    Up to increasing $k$, we can assume without loss of generality that $Z'$ is not a metric suspension.
    Taking the metric cones, we have a (bijective) isometry $h:\R^k\times C(Z)\to\R^k\times C(Z')$
    such that $h(0,o)=(0,o')$, where $o$ and $o'$ denote the tips of $C(Z)$ and $C(Z')$ respectively.
    We claim that $h(\R^k\times\{o\})\subseteq \R^k\times\{o'\}$. Once this is done, the inclusion must be an equality (as a distance-preserving map between Euclidean spaces is linear), and $h$ restricts to an isometry $\{0\}\times C(Z)\to\{0\}\times C(Z')$,
    as $\{0\}\times C(Z)$ consists of the points equidistant from all the points in the sphere $S^{k-1}\times\{o\}$, and similarly in the second product. Hence, $Z$ is isometric to $Z'$.

The claim follows from cone-splitting, since we assumed that $Z'$ is not a spherical suspension and hence $\setR^k\times C(Z')$ does not split a factor $\setR^{k-1}$; see for instance \cite{CheegerNaber13}.

\end{proof}

\begin{proof}[Proof of \autoref{thm:omeo n-4 sym}]
Let $x\in X$ be as in the assumptions of the theorem. Let $\mathcal{C}_x^{n-4}$ be the collection of all metric spaces $(Z,\dist_Z)$ such that there is some tangent cone at $x$ isometric to $\setR^{n-4}\times C(Z)$. By \autoref{lemma:crossconn}, $\mathcal{C}_x^{n-4}$ is a compact and connected subset of the class of $\RCD(2,3)$ spaces. Moreover, the $\haus^3$-volume is constant on $\mathcal{C}_x^{n-4}$. By \autoref{thm:topreg}, any $(Z,\dist_Z)\in \mathcal{C}_x^{n-4}$ is a topological $3$-manifold whose universal cover is homeomorphic to $S^3$. By \autoref{thm:unifcontrtopRCD3}, all the elements of $\mathcal{C}^{n-4}_x$ are locally uniformly contractible. Indeed, they are $3$-dimensional $\RCD(2,3)$ topological manifolds with constant and hence uniformly lower bounded $\haus^3$-volume.  

By \autoref{thm:topstabmain}, for any $(Z,\dist_Z)\in \mathcal{C}_x^{n-4}$ there is $\epsilon=\epsilon(Z)>0$ such that, for every $(Z',\dist_{Z'})\in \mathcal{C}^{n-4}_x$ with $\dist_{\rm{GH}}(Z,Z')\le \epsilon$, it holds that $Z$ and $Z'$ are homeomorphic. The statement follows by the connectedness of $\mathcal{C}^{n-4}_x$ that we obtained in \autoref{lemma:crossconn}.
\end{proof}

\subsection{Proof of \autoref{cor:4d}}
The first part of the statement follows from \autoref{thm:omeo n-4 sym}. Notice indeed that in dimension $4$ the assumption of $(n-4)$-symmetry of all tangent cones is trivially met at all points. 
\medskip

To conclude the proof, it is enough to check that for every $x\in X\setminus \mathcal{S}^0$ there exists a tangent cone whose cross-section is homeomorphic to $S^3$. By definition, if $x\in X\setminus \mathcal{S}^0$ there exists a split tangent cone $C(Z^3)=C(\Sigma^2)\times \R$. By \autoref{thm:noRP^2main}, we know that $\Sigma^2$ is homeomorphic to the two-sphere; hence, $Z^3$ is homeomorphic to the spherical suspension over $\Sigma^2$, itself homeomorphic to $S^3$.



\frenchspacing
\begin{thebibliography}{GMS13}





\bibitem{AlexanderKapovitchPetrunin}
\textsc{S. Alexander, V. Kapovitch, A. Petrunin:}
\textit{Alexandrov geometry: foundations.}
Graduate Studies in Mathematics,
vol. 236 (2024).


\bibitem{AmbrosioICM}
\textsc{L. Ambrosio:}
\textit{Calculus, heat flow and curvature-dimension bounds in metric measure spaces.}
Proceedings of the International Congress of Mathematicians, Rio de Janeiro 2018. Vol. I. Plenary lectures, 301--340. World Sci. Publ., Hackensack, NJ, 2018. 



\bibitem{AmbrosioHonda}
\textsc{L. Ambrosio, S. Honda:}
\textit{New stability results for sequences of metric measure spaces with uniform Ricci bounds from below.}
Measure theory in non-smooth spaces, 1--51. Partial Differ. Equ. Meas. Theory, De Gruyter Open, Warsaw, 2017.



\bibitem{AmbrosioMondinoSavare}
\textsc{L. Ambrosio, A. Mondino, G. Savarè:}
\textit{Nonlinear diffusion equations and curvature conditions in metric measure spaces.}
Mem. Amer. Math. Soc. {\bf 262} (2019), no. 1270, v+121 pp.

\bibitem{AncelEdwards}
\textsc{F.-D. Ancel, R.-D. Edwards:}
\textit{Is a monotone union of contractible open sets contractible?}
Topol. Appl. {\bf 214} (2016), 89--93.

\bibitem{Anderson90}
\textsc{M.-T. Anderson:}
\textit{Convergence and rigidity of manifolds under Ricci curvature bounds}.
Invent. Math. {\bf 102} (1990), no. 2, 429--445.


\bibitem{Anderson93}
\textsc{M.-T. Anderson:}
\textit{Degeneration of metrics with bounded curvature and applications to critical metrics of Riemannian functionals.} 
Differential geometry: Riemannian geometry (Los Angeles, CA, 1990), pp. 53--79.
Proc. Sympos. Pure Math., vol. 54, part 3.


\bibitem{Andersonshort}
\textsc{M.-T. Anderson:}
\textit{Short geodesics and gravitational instantons.}
J. Differential Geom. {\bf 31} (1990), no. 1, 265--275.

\bibitem{AndersonICM}
\textsc{M.-T. Anderson:}
\textit{Einstein metrics and metrics with bounds on Ricci curvature.}
Proceedings of the International Congress of Mathematicians (Z\"urich, 1994), vol. 1--2, 443--452.


\bibitem{Andersonsurvey}
\textsc{M.-T. Anderson:}
\textit{A survey of Einstein metrics on 4-manifolds,}
Adv. Lect. Math. (ALM), vol. {14}, 1--39.
International Press, Somerville, MA, 2010.


\bibitem{AntonelliBrueSemola}
\textsc{G. Antonelli, E. Bru\`e, D. Semola:}
\textit{Volume bounds for the quantitative singular strata of non-collapsed RCD metric measure spaces,}
Anal. Geom. Metr. Spaces {\bf 7} (2019), 158--178.

\bibitem{BassoMartiWenger}
\textsc{G. Basso, D. Marti, S. Wenger:}
\textit{Geometric and analytic structures on metric spaces homeomorphic to a manifold.}
Preprint arXiv:2303.13490 (2023).  


\bibitem{Begleduality}
\textsc{G.-E. Begle:}
\textit{Duality theorems for generalized manifolds.}
Amer. J. Math. {\bf 67} (1945), 59--70.



\bibitem{Besson21}
\textsc{G. Besson:}
\textit{Scalar curvature on some open 3-manifolds.}
Internat. J. Math. {\bf 32} (2021), no. 12, art. 2140014.


\bibitem{Bingtame}
\textsc{R.-H. Bing:}
\textit{Conditions under which a surface in  E3  is tame.}
Fund. Math. {\bf 47} (1959), 105--139.


\bibitem{Bingproduct}
\textsc{R.-H. Bing:}
\textit{The Cartesian product of a certain nonmanifold and a line is E4.} 
Ann. of Math. (2) {\bf 70} (1959), 399--412.


\bibitem{Bing 1ULC}
\textsc{R.-H. Bing:}
\textit{A surface is tame if its complement is  1-ULC.}
Trans. Amer. Math. Soc. {\bf 101} (1961), 294--305.


\bibitem{Bing appr}
\textsc{R.-H. Bing:}
\textit{Approximating surfaces from the side.}
Ann. of Math. (2) {\bf 77} (1963), 145--192.

\bibitem{Borelduality}
\textsc{A. Borel:}
\textit{The Poincar\'e duality in generalized manifolds.}
Michigan Math. J. {\bf 4} (1957), 227--239.

\bibitem{BrenaGigliHondaZhu}
\textsc{C. Brena, N. Gigli, S. Honda, X. Zhu:}
\textit{Weakly non-collapsed $\RCD$ spaces are strongly non-collapsed.}
J. Reine Angew. Math. {\bf 794}, 215--252 (2023).


\bibitem{BrueDengSemola}
\textsc{E. Bruè, Q. Deng, D. Semola:}
\textit{Improved regularity estimates for Lagrangian flows on $\RCD(K,N)$ spaces.}
Nonlinear Anal. {\bf 214} (2022), art. 112609.



\bibitem{BrueNaberSemolabdry}
\textsc{E. Bruè, A. Naber, D. Semola:}
\textit{Boundary regularity and stability for spaces with Ricci bounded below.}
Invent. Math. {\bf 228} (2022) no. 2, 777--891.







\bibitem{Brown}
\textsc{M. Brown:}
\textit{The monotone union of open n-cells is an open $n$-cell.}
Proc. Amer. Math. Soc. {\bf 12} (1961) 812--814.


\bibitem{BuragoBuragoIvanov}
\textsc{D. Burago, Y. Burago, S. Ivanov:}
\textit{A course in metric geometry.}
Grad. Stud. Math., vol. {33}. American Mathematical Society, Providence, RI, 2001.



\bibitem{BuragoGromovPerelman92}
\textsc{Y. Burago, M. Gromov, G. Perelman:}
\textit{Alexandrov spaces with curvature bounded below.}
Russian Mathematical Surveys, vol. 47.2, 1--58.




\bibitem{Cannon}
\textsc{J.-W. Cannon:}
\textit{The recognition problem: what is a topological manifold?}
Bull. Amer. Math. Soc. {\bf 84} (1978), no. 5, 832--866.



\bibitem{Cavicchiolietal}
\textsc{A. Cavicchioli, D. Repovš, T.-L. Thickstun:}
\textit{Geometric topology of generalized 3-manifolds.}
J. Math. Sci. {\bf 144} (2007), no. 5, 4413--4422.



\bibitem{CheegerFermi}
\textsc{J. Cheeger:}
\textit{Degeneration of Riemannian metrics under Ricci curvature bounds.}
Lezioni Fermiane [Fermi Lectures].
Scuola Normale Superiore, Pisa, 2001.


\bibitem{Cheeger03}
\textsc{J. Cheeger:}
\textit{Integral bounds on curvature elliptic estimates and rectifiability of singular sets.}
Geom. Funct. Anal. {\bf 13} (2003), no. 1, 20--72.


\bibitem{CheegerColding96}
\textsc{J. Cheeger, T.-H. Colding:}
\textit{Lower bounds on Ricci curvature and the almost rigidity of warped products.}
Ann. of Math. (2) {\bf 144} (1996), no. 1, 189--237.


\bibitem{CheegerColding97I}
\textsc{J. Cheeger, T.-H. Colding:}
\textit{On the structure of spaces with Ricci curvature bounded below. I.}
J. Differential Geom. {\bf 46} (1997), no. 3, 406--480.

\bibitem{CheegerColding00II}
\textsc{J. Cheeger, T.-H. Colding:}
\textit{On the structure of spaces with Ricci curvature bounded below. II.}
J. Differential Geom. {\bf 54} (2000), no. 1, 13--35.


\bibitem{CheegerColdingTian}
\textsc{J. Cheeger, T.-H. Colding, G. Tian:}
\textit{On the singularities of spaces with bounded Ricci curvature.}
Geom. Funct. Anal. {\bf 12} (2002), no. 5, 873--914.



\bibitem{CheegerJiangNaber}
\textsc{J. Cheeger, W. Jiang, A. Naber:}
\textit{Rectifiability of singular sets of noncollapsed limit spaces with Ricci curvature bounded below.}
Ann. of Math. (2) {\bf 193} (2021), no. 2, 407--538.


\bibitem{CheegerNaber13}
\textsc{J. Cheeger, A. Naber:}
\textit{Lower bounds on Ricci curvature and quantitative behavior of singular sets.}
Invent. Math. {\bf 191} (2013) no. 2, 321--339.

\bibitem{CheegerNaber15}
\textsc{J. Cheeger, A. Naber:}
\textit{Regularity of Einstein manifolds and the codimension 4 conjecture.}
Ann. of Math. (2) {\bf 182} (2015), no. 3, 1093--1165.

\bibitem{Coldingvol}
\textsc{T.-H. Colding:}
\textit{Ricci curvature and volume convergence.}
Ann. of Math. (2) {\bf 145} (1997), no. 3, 477--501.


\bibitem{ColdingGreen}
\textsc{T.-H. Colding:}
\textit{New monotonicity formulas for Ricci curvature and applications. I.}
Acta Math. {\bf 209} (2012), no. 2, 229--263.



\bibitem{ColdingMinicozziEinstein}
\textsc{T.-H. Colding, W.-P. Minicozzi II:}
\textit{On uniqueness of tangent cones for Einstein manifolds.}
Invent. Math. {\bf 196} (2014), no. 3, 515--588.



\bibitem{ColdingNaberHolder}
\textsc{T.-H. Colding, A. Naber:}
\textit{Sharp Hölder continuity of tangent cones for spaces with a lower Ricci curvature bound and applications.}
Ann. of Math. (2) {\bf 176} (2012), no. 2, 1173--1229.




\bibitem{ColdingNabercones}
\textsc{T.-H. Colding, A. Naber:}
\textit{Characterization of tangent cones of noncollapsed limits with lower Ricci bounds and applications,}
Geom. Funct. Anal. {\bf 23} (2013), no. 1, 134--148.


\bibitem{DavermanRepovs}
\textsc{R.-J. Daverman, D. Repovš:}
\textit{General position properties that characterize 3-manifolds.} 
Can. J. Math. {\bf 44} (1992), no. 2, 234--251.


\bibitem{DavermanThickstun}
\textsc{R.-J. Daverman, T.-L. Thickstun:} 
\textit{The 3-manifold recognition problem.}
Trans. Amer. Math. Soc. {\bf 358} (2006), no. 12, 5257--5270.

\bibitem{DavermanVenema}
\textsc{R.-J. Daverman, G.-A. Venema:}
\textit{Embeddings in manifolds.}
Graduate Studies in Mathematics, vol. {106}. American Mathematical Society, Providence, RI, 2009.

\bibitem{DePhilippisGigli}
\textsc{G. De Philippis, N. Gigli:}
\textit{From volume cone to metric cone in the nonsmooth setting.}
Geom. Funct. Anal. {\bf 26} (2016), no. 6., 1526--1587. 


\bibitem{DePhilippisGigli2}
\textsc{G. De Philippis, N. Gigli:}
\textit{Non-collapsed spaces with Ricci curvature bounded from below.}
J. Éc. Polytech. Math. {\bf 5} (2018), 613--650.



\bibitem{Ding02}
\textsc{Y. Ding:}
\textit{Heat kernels and Green’s functions on limit spaces.}
Comm. Anal. Geom.
{\bf 10} (2002), 475--514.


\bibitem{EguchiHanson}
\textsc{T. Eguchi, A.-J. Hanson:}
\textit{Self-dual solutions to Euclidean gravity.}
Ann. Physics {\bf 120} (1979), no. 1, 82--106.



\bibitem{Engelking}
\textsc{R. Engelking:}
\textit{Dimension theory.}
North-Holland Math. Library, vol. 19.
North-Holland Publishing Co., Amsterdam-Oxford-New York; PWN--Polish Scientific Publishers, Warsaw, 1978.



\bibitem{Fukaya87}
\textsc{K. Fukaya:}
\textit{Collapsing of Riemannian manifolds and eigenvalues of Laplace operator.}
Invent. Math. {\bf 87} (1987), no. 3, 517--547.




\bibitem{Gigli13}
\textsc{N. Gigli:}
\textit{The splitting theorem in non-smooth context.}
Preprint arXiv:1302.5555 (2013). 




\bibitem{Gigli18}
\textsc{N. Gigli:}
\textit{Nonsmooth differential geometry - An approach tailored for spaces with Ricci curvature bounded from below.}
Mem. Amer. Math. Soc. {\bf 251} (2018), v+161 pp.



\bibitem{GigliSurvey}
\textsc{N. Gigli:}
\textit{De Giorgi and Gromov working together.}
Preprint: arXiv:2306.14604.







\bibitem{GigliMondinoSavare15}
\textsc{N. Gigli, A. Mondino, G. Savarè:}
\textit{Convergence of pointed non-compact metric measure spaces and stability of Ricci curvature bounds and heat flows.}
Proc. Lond. Math. Soc. (3) {\bf 111} (2015), no. 5, 1071--1129.



\bibitem{GigliViolo}
\textsc{N. Gigli, I.-Y. Violo:}
\textit{Monotonicity formulas for harmonic functions in  $\RCD(0,N)$ spaces,}
J. Geom. Anal. {\bf 33} (2023), no. 3, art. 100.



\bibitem{GrovePetersen88}
\textsc{K. Grove, P.-V. Petersen:}
\textit{Bounding homotopy types by geometry,}
Ann. of Math. (2) {\bf 128} (1988), no. 1, 195--206.




\bibitem{Hatcherbook}
\textsc{A. Hatcher:}
\textit{Algebraic topology.}
Cambridge University Press, Cambridge, 2002.





\bibitem{Hubook}
\textsc{S.-T. Hu:}
\textit{Theory of retracts.}
Wayne State University Press, Detroit, MI, 1965.


\bibitem{HuppNaberWang}
\textsc{E. Hupp, A. Naber, K.-H. Wang:}
\textit{Lower Ricci Curvature and Nonexistence of Manifold Structure.}
Preprint arXiv:2308.03909, to appear in Geom. Topol.

\bibitem{HurewiczWallman}
\textsc{W. Hurewicz, H. Wallman:}
\textit{Dimension Theory.}
Princeton Math. Ser., vol. 4.
Princeton University Press, Princeton, NJ, 1941.

\bibitem{Jakobsche}
\textsc{W. Jakobsche:}
\textit{Approximating homotopy equivalences of 3-manifolds by homeomorphisms.}
Fund. Math. {\bf 130} (1988), no. 3, 157--168.

\bibitem{NaberJiang}
\textsc{W. Jiang, A. Naber:}
\textit{$L^2$ curvature bounds on manifolds with bounded Ricci curvature.}
Ann. of Math. (2) {\bf 193} (2021), no. 1, 107--222.


\bibitem{Kapovitchcross}
\textsc{V. Kapovitch:}
\textit{Regularity of limits of noncollapsing sequences of manifolds.}
Geom. Funct. Anal. {\bf 12} (2002), no. 1, 121--137.

\bibitem{KapovitchPer}
\textsc{V. Kapovitch:}
\textit{Perelman's stability theorem.}
Surv. Differ. Geom., vol. {11}, 103--136.
International Press, Somerville, MA, 2007.


\bibitem{KapovitchMondino}
\textsc{V. Kapovitch, A. Mondino:}
\textit{On the topology and the boundary of $N$-dimensional $\RCD(K,N)$ spaces.}
Geom. Topol. {\bf 25} (2021), no. 1, 445--495.




\bibitem{Ketterer15}
\textsc{C. Ketterer:}
\textit{Cones over metric measure spaces and the maximal diameter theorem.}
J. Math. Pures Appl. (9) {\bf 103} (2015), no. 5, 1228--1275.


\bibitem{Kronheimer}
\textsc{P.-B. Kronheimer:}
\textit{The construction of ALE spaces as hyper-K\"ahler quotients.}
J. Differential Geom. {\bf 29} (1989), no. 3, 665--683.



\bibitem{Lebedevaetal}
\textsc{N. Lebedeva, V. Matveev, A. Petrunin, V. Shevchishin:}
\textit{Smoothing 3-dimensional polyhedral spaces.}
Electron. Res. Announc. Math. Sci. {\bf 22} (2015), 12--19.

\bibitem{Li86}
\textsc{P. Li:}
\textit{Large time behavior of the heat equation on complete manifolds with nonnegative Ricci curvature.}
Ann. of Math. (2) {\bf 124} (1986), no. 1, 1--21.


\bibitem{LiNaber}
\textsc{N. Li, A. Naber:}
\textit{Quantitative estimates on the singular sets of Alexandrov spaces.}
Peking Math. J. {\bf 3} (2020), no. 2, 203--234.



\bibitem{LiYau}
\textsc{P. Li, S.-T. Yau:}
\textit{On the parabolic kernel of the Schr\"odinger operator.}
Acta Math. {\bf 156} (1986), 154--201.




\bibitem{Liu13}
\textsc{G. Liu}
\textit{3-manifolds with nonnegative Ricci curvature.}
Invent. Math. {\bf 193} (2013), no. 2, 367--375.



\bibitem{LiuSzekeI}
\textsc{G. Liu, G. Sz\'ekelyhidi:}
\textit{Gromov--Hausdorff limits of K\"ahler manifolds with Ricci curvature bounded below.}
Geom. Funct. Anal. {\bf 32} (2022), no. 2, 236--279.


\bibitem{LiuSzekeII}
\textsc{G. Liu, G. Sz\'ekelyhidi:}
\textit{Gromov--Hausdorff limits of K\"ahler manifolds with Ricci curvature bounded below II.}
Comm. Pure Appl. Math. {\bf 74} (2021), no. 5, 909--931.






\bibitem{LottVillani}
\textsc{J. Lott, C. Villani:}
\textit{Ricci curvature for metric-measure spaces via optimal transport.}
Ann. of Math. (2) {\bf 169} (2009), no. 3, 903--991.


\bibitem{LytchakNagano19}
\textsc{A. Lytchak, K. Nagano:}
\textit{Geodesically complete spaces with an upper curvature bound.}
Geom. Funct. Anal. {\bf 29} (2019), no. 1, 295--342.



\bibitem{LytchakNagano22}
\textsc{A. Lytchak, K. Nagano:}
\textit{Topological regularity of spaces with an upper curvature bound.}
J. Eur. Math. Soc. (JEMS) {\bf 24} (2022), no.1, 137--165.

\bibitem{LytchakNaganoStadler}
\textsc{A. Lytchak, K. Nagano, S. Stadler:}
\textit{$\rm{CAT}(0)$ $4$-manifolds are Euclidean.}
Preprint arXiv:2109.09438, to appear on Geom. Topol.


\bibitem{LytchakStadler}
\textsc{A. Lytchak, S. Stadler:}
\textit{Ricci curvature in dimension 2.}
J. Eur. Math. Soc. (JEMS) {\bf 25} (2023), no. 3, 845--867.



\bibitem{Menguy inftop}
\textsc{X. Menguy:}
\textit{Noncollapsing examples with positive Ricci curvature and infinite topological type.}
Geom. Funct. Anal. {\bf 10} (2000), no. 3, 600--627.







\bibitem{MondinoNaber19}
\textsc{A. Mondino, A. Naber:}
\textit{Structure theory of metric measure spaces with lower Ricci curvature bounds.}
J. Eur. Math. Soc. (JEMS) {\bf 21} (2019), no. 6, 1809--1854.




\bibitem{MondinoSemola23}
\textsc{A. Mondino, D. Semola:}
\textit{Weak Laplacian bounds and minimal boundaries in non-smooth spaces with Ricci curvature lower bound.}
Preprint arXiv:2107.12344, to appear on Mem. Amer. Math. Soc.




\bibitem{MondinoWei}
\textsc{A. Mondino, G. Wei:}
\textit{On the universal cover and the fundamental group of an $\RCD^*(K,N)$-space.}
J. Reine Angew. Math. {\bf 753} (2019), 211--237.


\bibitem{Naberconj}
\textsc{A. Naber:}
\textit{Conjectures and open questions on the structure and regularity of spaces with lower Ricci curvature bounds.}
SIGMA {\bf 16} (2020), no. 104.



\bibitem{NaberZhang}
\textsc{A. Naber, R. Zhang:}
\textit{Topology and  $\epsilon$-regularity theorems on collapsed manifolds with Ricci curvature bounds.}
Geom. Topol. {\bf 20} (2016), no. 5, 2575--2664.








\bibitem{Otsu}
\textsc{Y. Otsu:}
\textit{On manifolds of positive Ricci curvature with large diameter.}
Math. Z. {\bf 206} (1991), no. 2, 255--264.



\bibitem{PanWei}
\textsc{J. Pan, G. Wei:}
\textit{Semi-local simple connectedness of non-collapsing Ricci limit spaces.}
J. Eur. Math. Soc. (JEMS) {\bf 24} (2022), no. 12, 4027--4062.



\bibitem{Perelman99}
\textsc{G. Perelman:} 
\textit{Alexandrov spaces with curvatures bounded from below II.} 
Preprint, 1991.


\bibitem{PerelmanII}
\textsc{G. Perelman:}
\textit{Elements of Morse theory on Alexandrov spaces.}
St. Petersbg. Math. J. {\bf 5} (1993), no. 1, 205--213.

\bibitem{Perelmanmaximal}
\textsc{G. Perelman:}
\textit{Manifolds of positive Ricci curvature with almost maximal volume.}
J. Amer. Math. Soc. {\bf 7} (1994), no. 2, 299--305.


\bibitem{PerelmanPoincare}
\textsc{G. Perelman:}
\textit{The entropy formula for the Ricci flow and its geometric applications.}
Preprint arXiv:math/0211159 (2002).



\bibitem{Petersen90}
\textsc{P.-V. Petersen:}
\textit{A finiteness theorem for metric spaces.}
J. Differential Geom. {\bf 31} (1990), no. 2, 387--395.



\bibitem{SchoenYau}
\textsc{R. Schoen, S.-T. Yau:}
\textit{Complete three-dimensional manifolds with positive Ricci curvature and scalar curvature.}
Seminar on Differential Geometry, 209--228. Ann. of Math. Stud., vol. 102. Princeton Univ. Press, Princeton, NJ, 1982.






\bibitem{Simon14}
\textsc{M. Simon:}
\textit{Ricci flow of non-collapsed three manifolds whose Ricci curvature is bounded from below.}
J. Reine Angew. Math. {\bf 662} (2012), 59--94.


\bibitem{SimonTopping22}
\textsc{M. Simon, P.-M. Topping:}
\textit{Local mollification of Riemannian metrics using Ricci flow, and Ricci limit spaces.}
Geom. Topol. {\bf 25} (2021), no. 2, 913--948.


\bibitem{SimonTopping22a}
\textsc{M. Simon, P.-M. Topping:}
\textit{Local control on the geometry in 3D Ricci flow.}
J. Differential Geom. {\bf 122} (2022), no. 3, 467--518.

\bibitem{Sturm07I}
\textsc{K.-T. Sturm:}
\textit{On the geometry of metric measure spaces. I.}
Acta Math. {\bf 196} (2006), no. 1, 65--131.

\bibitem{Sturm07II}
\textsc{K.-T. Sturm:}
\textit{On the geometry of metric measure spaces. II.}
Acta Math. {\bf 196} (2006), no. 1, 133--177.

\bibitem{SturmECM}
\textsc{K.-T. Sturm:}
\textit{Metric measure spaces and synthetic Ricci bounds: fundamental concepts and recent developments.}
European Congress of Mathematics, 125--159.
EMS Press, Berlin, 2023.





\bibitem{Thickstuna}
\textsc{T.-L. Thickstun:}
\textit{An extension of the loop theorem and resolutions of generalized  3-manifolds with  0-dimensional singular set.}
Invent. Math. {\bf 78} (1984), no. 2, 161--222.

\bibitem{Thickstunb}
\textsc{T.-L. Thickstun:}
\textit{Resolutions of generalized 3-manifolds whose singular sets have general position dimension one.}
Topology Appl. {\bf 138} (2004), no. 1--3, 61--95.


\bibitem{Varopoulos}
\textsc{N.-T. Varopoulos:}
\textit{The Poisson kernel on a positively curved manifold.}
J. Funct. Anal. {\bf 44} (1981), 359--380.




\bibitem{WangRCD}
\textsc{J. Wang:}
\textit{$\RCD^*(K,N)$ spaces are semi-locally simply connected.}
J. Reine Angew. Math. (2023), early view.



\bibitem{Zhou23}
\textsc{S. Zhou:}
\textit{On the Gromov--Hausdorff limits of Tori with Ricci conditions.}
Preprint arXiv:2309.10997 (2023). 

\bibitem{Zhoub}
\textsc{S. Zhou:}
\textit{Examples of Ricci limit spaces with infinite holes.}
Preprint arXiv:2404.00619 (2024).

\bibitem{Zhu93}
\textsc{S.-H. Zhu:}
\textit{A finiteness theorem for Ricci curvature in dimension three.}
J. Differential Geom. {\bf 37} (1993), no. 3, 711--727.


\end{thebibliography}
\end{document}